\documentclass[reqno, 11pt]{amsart}
\usepackage[left=2cm, right=2cm, top=2cm, bottom=2cm]{geometry}
\usepackage[T1]{fontenc}
\usepackage[utf8]{inputenc}	
\usepackage[english]{babel}
\usepackage[babel]{csquotes}
\usepackage{amssymb}
\usepackage{amsmath}
\usepackage{amsthm}
\usepackage{latexsym}
\usepackage{xcolor}
\usepackage{mathrsfs}
\usepackage{bbm}
\usepackage{verbatim}
\usepackage[colorlinks=true, pdfstartview=FitV, linkcolor=blue, citecolor=blue, urlcolor=blue]{hyperref}
\usepackage[style=numeric, sorting=nty, sortcites=true, backend=bibtex, maxbibnames=12, giveninits=true]{biblatex}
\usepackage{faktor}
\usepackage[shortlabels]{enumitem}
\usepackage[english]{babel}

\numberwithin{equation}{section}
\newtheorem{thm}{Theorem}[section]

\newtheorem{lem}[thm]{Lemma}
\newtheorem{prop}[thm]{Proposition}
\theoremstyle{definition}
\newtheorem{defin}[thm]{Definition}
\newtheorem{remark}[thm]{Remark}

\def\T{\mathbb{T}}
\def\Td{\T^d}
\renewcommand{\d}{{\mathrm d}} 
\renewcommand{\div}{\operatorname{div}} 
\newcommand{\norm}[1]{\left\|#1\right\|} 
\newcommand{\ip}[2]{\left\langle #1,#2 \right\rangle} 
\newcommand{\ii}[4]{\int_{#1}^{#2} #3 \: \d#4} 
\def\laweq{\stackrel{\mathcal L}{=}}
\def\supp{\sup_{\tau \in [0,t]}}
\def\sups{\sup_{s \in [0,t]}}
\def\tom{\widetilde{\Omega}}
\def\tF{\widetilde{\mathscr{F}}}
\def\tP{\widetilde{\mathbb{P}}}
\def\tE{\widetilde{\mathbb{E}}}
\def\X{\mathcal{X}}
\def\hom{\widehat{\Omega}}

\def\bH{{\boldsymbol H}}

\def\b #1{{\boldsymbol #1}}
\def\enne{\mathbb{N}}

\def\erre{\mathbb{R}}
\def\R{\mathbb{R}}

\def\P{\mathbb{P}}

\def\E{\mathop{{}\mathbb{E}}}

\def\cL{\mathscr{L}}
\def\cF{\mathscr{F}}

\def\cP{\mathscr{P}}

\def\OO{\mathcal{O}}
\def\embed{\hookrightarrow}

\def\luca #1{{\color{blue} #1}}

\makeatletter
\DeclareFontFamily{OMX}{MnSymbolE}{}
\DeclareSymbolFont{MnLargeSymbols}{OMX}{MnSymbolE}{m}{n}
\SetSymbolFont{MnLargeSymbols}{bold}{OMX}{MnSymbolE}{b}{n}
\DeclareFontShape{OMX}{MnSymbolE}{m}{n}{
	<-6>  MnSymbolE5
	<6-7>  MnSymbolE6
	<7-8>  MnSymbolE7
	<8-9>  MnSymbolE8
	<9-10> MnSymbolE9
	<10-12> MnSymbolE10
	<12->   MnSymbolE12
}{}
\DeclareFontShape{OMX}{MnSymbolE}{b}{n}{
	<-6>  MnSymbolE-Bold5
	<6-7>  MnSymbolE-Bold6
	<7-8>  MnSymbolE-Bold7
	<8-9>  MnSymbolE-Bold8
	<9-10> MnSymbolE-Bold9
	<10-12> MnSymbolE-Bold10
	<12->   MnSymbolE-Bold12
}{}

\let\llangle\@undefined
\let\rrangle\@undefined
\DeclareMathDelimiter{\llangle}{\mathopen}%
{MnLargeSymbols}{'164}{MnLargeSymbols}{'164}
\DeclareMathDelimiter{\rrangle}{\mathclose}%
{MnLargeSymbols}{'171}{MnLargeSymbols}{'171}
\makeatother
\AtEveryBibitem{%
	\clearfield{doi}%
	\clearfield{url}%
	\clearfield{issn}%
	\clearfield{isbn}%
	\ifentrytype{article}
	{}
	{\clearfield{eprint}}
} 
\renewbibmacro*{publisher+location+date}{%
	\ifentrytype{book}
	{\printlist{publisher}%
		\iflistundef{location}
		{\setunit*{\addcomma\space}}
		{\setunit*{\addcomma\space}}%
		\printlist{location}}
	{\printlist{location}%
		\iflistundef{publisher}
		{\setunit*{\addcomma\space}}
		{\setunit*{\addcomma\space}}%
		\printlist{publisher}}
	\setunit*{\addcomma\space}%
	\usebibmacro{date}%
	\newunit} 

\usepackage[]{amsaddr}

\begin{document}
	\title[Stochastic conservative Cahn--Hilliard/Allen--Cahn equations]
	{Stochastic Cahn--Hilliard and conserved Allen--Cahn equations
	\\ with logarithmic potential and conservative noise}
	\author{Andrea Di Primio, Maurizio Grasselli, Luca Scarpa}
	\address{Dipartimento di Matematica,
		Politecnico di Milano, Via E.~Bonardi 9, 20133 Milano, Italy}
	\email{andrea.diprimio@polimi.it}
	\email{maurizio.grasselli@polimi.it}
	\email{luca.scarpa@polimi.it}
	\keywords{Cahn--Hilliard equations; Allen--Cahn equations; Flory--Huggins potential; stochastic flows; conservative noise. \\
	The present research is part of the activities of ``Dipartimento di Eccellenza 2023-2027''.}
	
	\begin{abstract}
		We investigate the Cahn--Hilliard and the conserved Allen--Cahn equations with logarithmic type potential and conservative noise in a periodic domain. These features ensure that the order parameter takes its values in the physical range and, albeit the stochastic nature of the problems, that the total mass is conserved almost surely in time.
		For the Cahn--Hilliard equation,
		existence and uniqueness of probabilistically-strong solutions
		is shown up to the three-dimensional case.
		For the conserved Allen--Cahn equation,
		under a restriction on the noise magnitude,
		existence of martingale solutions is proved even in dimension three,
		while existence and uniqueness of probabilistically-strong solutions
		holds in dimension one. The analysis is carried out by studying the
		Cahn--Hilliard/conserved Allen--Cahn equations jointly,
		that is a linear combination of both the equations,
		which has an independent interest.
	\end{abstract}
	\maketitle
	
	\section{Introduction} \label{sec:intro} \noindent

The Cahn--Hilliard equation models fairly well an important phenomenon called phase separation, that is, the competition between mixing entropy and demixing effects in a mixture of two (or more) immiscible (or partially miscible) substances at a certain temperature below a critical threshold (see \cite{cher-mir-zel, elliott, Mir-CH} and references therein). This phenomenon has also recently become important in Cell Biology (see, for instance, \cite{Boeynaems,brang-tompa,dolgin18, dolgin22, poly}). If we consider a binary mixture in a bounded domain $\OO \subset \mathbb{R}^d$, $d \in \{1,2,3\}$, and we denote by $\varphi$ the relative concentration difference of the two species, then the original form of the Cahn--Hilliard equation is the following (see \cite{Cahn-Hill,Cahn-Hill2})
\begin{equation} \label{CHE}
			\begin{cases}
				\partial_t\varphi - \div(m(\varphi)\nabla\mu) = 0 \\
				\mu = -\varepsilon\Delta \varphi + F'_{\text{log}}(\varphi)
		    \end{cases}
		\end{equation}
in $\OO \times (0,T)$, with $T>0$ being a given final time.
Here, $m(\cdot)$ is the mobility (e.g. $m(x)=1-x^2$ or $m$ constant), while $\mu$ is called chemical potential and it is the functional derivative of the free energy functional
$$
\mathcal{E}(\varphi) = \int_\OO \left(\frac{\varepsilon^2}{2}\vert \nabla\varphi\vert^2 + F_{\text{log}}(\varphi)\right) \d x
$$
where $\epsilon>0$ is related to the thickness of the diffuse interface separating the two species and $F_{\text{log}}$ is the so--called Flory--Huggins potential (see \cite{Flory42,Huggins41}), namely,
 \begin{equation}\label{log}
			F_{\text{log}}(s)=\theta\left[(1+s)\ln(1+s) + (1-s)\ln(1-s)\right] - \theta_0s^2, \quad s\in(-1,1),
\end{equation}
with $0<\theta<\theta_0$, $\theta_0$ being a critical temperature depending on the nature of the mixture. Equation \eqref{CHE} equipped with no-flux or periodic boundary conditions
entailing mass conservation, that is, $\int_\OO \varphi(t)\d x$ remains constant over time. This is an important feature of a phase separation process. A second--order equation
which has been proposed as an alternative to equation \eqref{CHE} with constant mobility is the so-called conserved Allen--Cahn equation (see \cite{RS1992})
\begin{equation} \label{CACE}
			\begin{cases}
				\partial_t\varphi + m(\mu - \overline{\mu})=0 \\
				\mu = -\varepsilon\Delta \varphi + F_{\text{log}}'(\varphi)
		    \end{cases}
		\end{equation}
in $\OO \times (0,T)$. Here $\overline{\mu}$ is the spatial average of $\mu$ so that the total mass is still conserved if $\varphi$ satisfies the homogeneous Neumann or periodic boundary conditions.

The literature on the theoretical aspects of equation \eqref{CHE} is vast, especially in the constant mobility case (see, for instance, \cite{cher-mir-zel,Mir-CH} and references therein). Instead, equation \eqref{CACE}
has been analyzed in fewer contributions (see, for instance, \cite{GGW20,GP} and references therein). In both cases, it is important to note that the classical regular double--well approximation (e.g., $P(s)= (1-s^2)^2$)
of the potential $F_{\text{log}}$  no longer ensures that $\varphi$ takes its values in the physical range $[-1,1]$.

In the pioneering paper \cite{cook}, the author proposed a stochastic variant of \eqref{CHE}, also known as Cahn--Hilliard--Cook equation, by incorporating a Wiener noise into the mass flux $J= -m\nabla\mu$ to account for random thermal fluctuations (see also \cite{BMW08} and references therein). Since then, there have been several papers devoted to stochastic versions of \eqref{CHE} with a double--well polynomial, but few with \eqref{log}:
see \cite{deb-goud, scar-SCH} in the case of constant mobility, and \cite{scarpa21} for non-constant (degenerate) mobility. Stochastic Allen--Cahn equations have also been studied, mostly in the non-conserved case (see \cite{ABDK16,bertacco21,BOS21,DGS,orr-scar,SZ}).
However, as far as we know, the joint case of a conservative noise and logarithmic potential has never been considered neither for equation \eqref{CHE} nor for equation \eqref{CACE}. This choice is closer to the original Cahn--Hilliard--Cook equation and allows the conservation of mass almost surely in time. Indeed, even if the deterministic model satisfies the mass conservation property, the same can not be inferred for the stochastic counterpart with an arbitrary noise term, since the mass is usually only conserved under expectations (see, however, \cite{deb-zamb,goud,goud-manca,goud-xie} where reflection measures are used).

The goal of this paper is to consider equations \eqref{CHE} and \eqref{CACE} with periodic boundary conditions and subject to a random forcing depending on $\varphi$, i.e.~in multiplicative form. This kind of noise can be finely tuned so that mass conservation holds for almost any realisation of the process starting from a given random variable. In this case, the mass conservation property holds almost everywhere in the probabilistic domain. The main idea is to consider a noise coefficient in divergence-form: this
was already proposed in other contexts, for which we refer e.g.~to
\cite{FG,grun-dirr,grun-gess,grun-metz}).

More precisely, letting $\T^d = \R^d / \: \mathbb{Z}^d$ be the $d$-dimensional torus, we consider the following problems in a filtered probability space
$(\Omega, \, \cF, (\cF_t)_{t \in [0,T]},\, \P)$ satisfying the usual conditions (i.e. the filtration is saturated and right-continuous):
	\begin{itemize}
		\item[(1)] the stochastic Cahn--Hilliard equations with conservative noise
		\begin{equation} \label{eq:ch}
			\begin{cases}
				\d \varphi - \Delta \mu \,\d t = \div \left( \b{G}(\varphi)\right) \d W & \quad \text{ in }\T^d \times (0,T) \\
				\mu = -\Delta \varphi + F'(\varphi) & \quad \text{ in }\T^d \times (0,T) \\
				\varphi(\cdot \:, 0) = \varphi_0 & \quad \text{ in } \Td.
			\end{cases}
		\end{equation}
		\item[(2)] the stochastic conserved Allen--Cahn equations with conservative noise
		\begin{equation} \label{eq:ac}
			\begin{cases}
				\d \varphi + \left(\mu - \overline{\mu}\right)\,\d t = \div \left( \b{G}(\varphi)\right)\,\d W & \quad \text{ in }\T^d \times (0,T) \\
				\mu = -\Delta \varphi + F'(\varphi) & \quad \text{ in }\T^d \times (0,T) \\
				\varphi(\cdot \:, 0) = \varphi_0 & \quad \text{ in } \Td.
			\end{cases}
		\end{equation}
	\end{itemize}
Here we have set $m=\varepsilon=1$ and, in both \eqref{eq:ch} and \eqref{eq:ac}, periodic boundary conditions are imposed.
Moreover, $W$ is some Wiener process with values in some separable Hilbert space,
while the vector field $\b{G}$ and the potential density function $F$, which is a generalization of \eqref{log}, are given.

We will not study problems \eqref{eq:ch} and \eqref{eq:ac} separately, but we shall consider first the following stochastic Allen--Cahn/Cahn--Hilliard equation
\begin{equation} \label{eq:chac}
		\begin{cases}
			\d \varphi + \left[-\alpha \Delta \mu + \beta(\mu - \overline{\mu}) \right]\d t = \div \left( \b{G}(\varphi)\right) \d W & \quad \text{ in }\T^d \times (0,T) \\
			\mu = -\Delta \varphi + F'(\varphi) & \quad \text{ in }\T^d \times (0,T) \\
			\varphi(\cdot \:, 0) = \varphi_0 & \quad \text{ in } \Td.
		\end{cases}
	\end{equation}
where $\alpha \geq 0$ and $\beta \in [0,1]$ be arbitrary with $\alpha+\beta>0$.
We recall that this equation in the deterministic case has been justified in \cite{kar-kat} (see also \cite{kar-nag})
and stochastic versions have been studied in \cite{ant-far-kar,ant-kar-mill}.
However, in all
those contributions the Allen-Cahn equation is the standard one, $F$ is a double--well polynomial and the total mass is not conserved.  Here, instead, conservation of mass holds.
Let us provide a formal proof that shall become rigorous once suitable assumptions on $\b{G}$ are made. Any solution $\varphi(\cdot) \in H^1(\Td)$ to \eqref{eq:chac} satisfies
\begin{equation*}
	( \varphi(t),\psi)_H +
	\int_0^t\!\int_{\Td} \alpha\nabla  \mu(s)\cdot \nabla \psi + \beta(\mu(s)-\overline{\mu(s)})\psi\,\d s
	= ( \varphi(0),\psi)_{H} +
	\left(\int_0^t \div \b{G}( \varphi(s))\,\d  W(s), \psi\right)_{H}
\end{equation*}
for all $\psi \in H^1(\Td)$, $t \in [0,T]$ and $\P$-almost surely. Observe that
choosing $\psi \equiv 1$ leads to
\begin{equation*}
	\int_{\Td} \varphi(t) \: \d x
	= \int_{\Td} \varphi_0 \: \d x +
	\left(\int_0^t \div \b{G}( \varphi(s))\,\d  W(s), 1 \right)_{H},
\end{equation*}
where by the structure of the noise and by the stochastic Fubini theorem (see \cite[Theorem 6.1.4]{LiuRo})
\[
\begin{split}
	\left(\int_0^t \div \b{G}( \varphi(s))\,\d  W(s), 1 \right)_{H} = \int_0^t \left( \div \b{G}( \varphi(s)), 1 \right)_{H}\d  W(s) = 0.
\end{split}
\]
The computation assumes formally that $\div \b{G}( \varphi(\cdot))$ defines a progressively measurable element in a suitable space of Hilbert--Schmidt operators
with values in $L^2(\Td)$ (precise assumptions will be made in the next section).

We first prove the uniqueness of a probabilistically--strong solution to problem \eqref{eq:chac} from which we deduce the same result for problem \eqref{eq:ch} (i.e., $\alpha=1$ and $\beta=0$).
We can also prove uniqueness for problem \eqref{eq:ac} but only in dimension one with a suitable smallness assumption on the magnitude of the noise. Then,
we investigate existence of solutions for both problems.
First, we stablish the existence of a probabilistically--strong
solution to \eqref{eq:chac} in the case $\alpha>0$ (hence to \eqref{eq:ch} as well):
this hinges on suitable elliptic regularity results in the equation
for the order parameter, which are indeed possible
given the choice $\alpha>0$. Secondly, we show
the existence of a martingale solution to problem \eqref{eq:ac}
via a vanishing viscosity argument as $\alpha\searrow0$ in
the mixed problem \eqref{eq:chac}: thanks to the uniqueness result,
in dimension $1$ thus is also unique and probabilistically-strong.
In the Allen--Cahn model \eqref{eq:ac} a smallness assumption on the noise intensity
is in order. This is very natural and is
due to the fact that in \eqref{eq:ac} the leading linear term
and the noise coefficient have the same differential order,
which is not the case for the Cahn--Hilliard model.

The analysis of the system \eqref{eq:ch} and \eqref{eq:ac}
is the starting point of several research directions, which are
already under investigation. For example, long-time behaviour
of solutions needs to be addressed, in terms of ergodicity,
existence/uniqueness of invariant measures, and Kolmogorov equations.
Furthermore, we note that the technical assumption on the magnitude of the
divergence-form noise
in the Allen--Cahn equation suggests that other possible choices of the noise
might be more suited to the Allen--Cahn dynamics: for example,
one might wonder whether it is possible to consider conservative noise coefficients
in the form $[G(\varphi)-\overline{G(\varphi)}]\d W$, by analogy with
the conserved Allen--Cahn operator acting on $\mu$. The
feasibility of the analysis in this case is far from obvious, due to
the presence of a logarithmic potential, but is currently under investigation.
Eventually, the extension of the above results to the physically-relevant case
of Neumann boundary conditions is also in progress.

The plan of the paper goes as follows. The necessary preliminary notions, the notation and the main results are stated throughout Section \ref{sec:main}.
Uniqueness results are established in Section \ref{sec:unique}, while the existence of solutions is proven in Sections \ref{sec:existence} and \ref{sec:existence2}.

	\section{Preliminaries and main results} \label{sec:main}
	\subsection{Functional setting and notation} \label{ssec:notation} The notation and the functional setting employed in the present work are illustrated in the following paragraphs.
	\paragraph{\textit{Probabilistic framework.}}	Let $(\Omega,\cF,(\cF_t)_{t\in[0,T]},\P)$ be a filtered probability space
	satisfying the usual conditions (namely the filtration is saturated and right-continuous),
	with $T>0$ being a prescribed final time. Let $\cP$ denote the progressive $\sigma$-algebra on $\Omega\times[0,T]$.
	For any couple of random variables defined on $\Omega$ and taking values in the same measurable space, the symbol $\laweq$ denotes the identity of their laws.
	Given a probability space $(E,\, \mathcal{M},\, m)$ and a Banach space $X$,
for any positive real quantity $p \geq 1$ the symbol $L^p(E, \mathcal{M}, m; X)$ (or simply $L^p(E; X)$) denotes the set of strongly measurable $X$-valued random variables on $E$ with finite moments up to order $p$. Sometimes, when the space $X$ is allowed to depend on time, we may denote the resulting space with $L^p_\cP(E;X)$
to stress that measurability is intended with respect to the progressive
$\sigma$-algebra $\cP$.
	Throughout the paper, $W$ denotes a cylindrical Wiener process on some separable, fixed Hilbert space $U$. The family $\{u_j\}_{j\in\enne} \subset U$ is a fixed orthonormal system for $U$. Let us precise the rigorous interpretation of the stochastic terms appearing in \eqref{eq:ch} and \eqref{eq:ac}.
	As a cylindrical process on $U$, the Wiener process $W$ admits the representation
	\begin{equation} \label{eq:representation}
		W = \sum_{k=0}^{+\infty} \beta_k u_k,
	\end{equation}
	where $\{\beta_k\}_{k \in \enne}$ is a family of real and independent Brownian motions. However, in order to make sense of the series \eqref{eq:representation}, it is necessary to enlarge the space $U$. In general, there exists some larger Hilbert space $U_0$ such that $U \embed U_0$ with Hilbert-Schmidt embedding $\iota$ and such that we can identify $W$ as a $Q^0$-Wiener process on $U_0$, for some trace-class operator $Q^0$ (see \cite[Subsection 2.5.1]{LiuRo}). Actually, it holds that $Q^0 = \iota \circ \iota^*$. In the following, we may implicitly assume this extension by simply saying that $W$ is a cylindrical process on $U$. This holds also for stochastic integration with respect to $W$. Indeed, the symbol
	\[
	\int_0^\cdot B(s)\,\d W(s) := \int_0^\cdot B(s) \circ \iota^{-1}(s)\,\d W(s),
	\]
	for every suitable progressively measurable process $B$. It is well known that such a  definition is well posed and does not depend on the choice of $U_0$ or $\iota$ (see \cite[Subsection 2.5.2]{LiuRo}).
	\paragraph{\textit{Functional spaces.}}
	Let $E$ be a Banach space. The symbol $\boldsymbol E$ denotes the product space $E^d$ (or $E^{d\times d}$), while its topological dual is denoted by $E^*$, while the duality pairing between $E^*$ and $E$ is denoted by $\ip{\cdot}{\cdot}_{E^*,E}$.
	If $E$ is a Hilbert space, then the scalar product of $E$ is denoted by $(\cdot,\cdot)_E$.
	Given two separable Hilbert spaces $E$ and $F$, the space of Hilbert-Schmidt operators
	from $E$ to $F$ is denoted by the symbol $\cL^2(E,F)$
	and endowed with its canonical norm $\norm{\cdot}_{\cL^2(E,F)}$.
	For every $s\in[1,+\infty]$,
	the symbol $L^s(X; Y)$ indicates the usual spaces of strongly measurable, Bochner-integrable functions
	defined on the Banach space $X$ and with values in the Banach spaces $Y$. If $Y$ is omitted, it is understood that $Y = \mathbb{R}$.
    For all $s\in(1,+\infty)$ and for every separable and reflexive Banach space $E$
	we also define
	\[
	L^s_w(\Omega; L^\infty(0,T; E^*)):=
	\left\{v:\Omega\to L^\infty(0,T; E^*) \text{ weakly*-measurable and }
	\norm{v}_{L^\infty(0,T; E^*)}\in L^s(\Omega)
	\right\}\,,
	\]
	which yields by
	\cite[Theorem 8.20.3]{edwards} the identification
	\[
	L^s_w(\Omega; L^\infty(0,T; E^*))=
	\left(L^{\frac{s}{s-1}}(\Omega; L^1(0,T; E))\right)^*\,.
	\]
	Concerning real Sobolev spaces, we employ the classical notation
	$W^{s,p}(\T^d)$, where $s\in\erre$ and $p\in[1,+\infty]$
	and we denote by $\norm{\cdot}_{W^{s,p}(\T^d)}$ their canonical norms. Here, $\mathbb T^d$ is the $d$-dimensional flat torus, with $d\in\{1,2,3\}$,
	so all that the elements of these spaces are periodic.
	We define the Hilbert space $H^s(\T^d):=W^{s,2}(\T^d)$ for all $s\in\erre$,
	endowed with its canonical norm $\norm{\cdot}_{H^s(\T^d)}$. Given $R > 0$, the closed ball in $W^{k,p}(\Td)$ centered at the origin with radius $R$ is denoted by $\mathbf{B}^{k,p}_R$, with the understanding that $W^{0,p}(\Td) = L^p(\Td)$ for every $p \geq 1$. Let us introduce further notation for the functional spaces
	\[
	H:=L^2(\T^d)\,, \qquad V:=H^1(\T^d)\,
	\]
	endowed with their standard norms $\norm{\cdot}_H$,
	$\norm{\cdot}_{V}$, respectively.
	As usual, we identify the Hilbert space $H$ with its dual through
	the corresponding Riesz isomorphism, so that we have the variational structure
	\[
	H^2(\Td) \embed V\embed H \embed V^* \embed (H^2(\Td))^*,
	\]
	with dense and compact embeddings (both in the cases $d = 2$ and $d = 3$). Also, we set the zero-mean spaces
	\[
	H_0 := \left\{ u \in H : \overline{u} := \dfrac{1}{|\Td|}\int_{\Td} u \: \d x = 0\right\}, \qquad V_0 := V \cap H_0.
	\]
	The space $H_0$ is endowed with the structure induced by $H$, hence we still carry over the same notation. Instead, owing to the Poincaré inequality, we set the $H^1$-seminorm structure on $V_0$, namely
	\[
	( u,\,v)_{V_0} := (\nabla u, \,\nabla v)_H, \qquad \|u\|_{V_0} := \|\nabla u\|_H, \qquad u,\,v \in V_0.
	\]
	This generates the zero-mean variational triplet
	\[
	V_0 \embed H_0 \embed V_0^*,
	\]
	once again with compact and dense embeddings in two and three dimensions.

	\paragraph{\textit{A family of linear operators.}} Let $\alpha > 0$ and $\beta \in [0,1]$, and consider the linear operator
	\[
	A_{\alpha\beta} : V \to V^*, \qquad \ip{A_{\alpha\beta}\psi}{\phi}_{V^*,V}=\alpha \int_{\T^d}\nabla\psi\cdot\nabla\phi \: \d x + \beta \int_{\Td} \left(\psi-\overline{\psi}\right)\left(\phi-\overline{\phi}\right) \: \d x
	\qquad \forall \: \psi,\,\phi\in V.
	\]
	For any $\alpha > 0$ and $\beta \in [0,1]$, the operator $A_{\alpha\beta}$ is a linear perturbation of	the variational realization of the negative Laplacian with periodic
	 boundary conditions. In particular, the restriction of $A_{\alpha\beta}$ to the space $V_0$ is an isomorphism between $V_0$ and its dual $V_0^*$, with inverse $\mathcal N_{\alpha \beta}:V_0^*\to V_0$. When $\alpha = 1$ and $\beta = 0$, given the Hilbert structure of $V_0$, the restriction of $A_{01}$ to $V_0$ is the Riesz isomorphism between $V_0$ and $V_0^*$.
	Finally, let us show the following basic result.
	\begin{prop} \label{prop:constants}
		For any fixed $\alpha > 0$ and $\beta \in [0,1]$, the norm
		\[\|\cdot\|_*: V_0^* \to [0,+\infty), \qquad v \mapsto \|\nabla \mathcal{N}_{\alpha\beta}v\|_\bH
		\]
		is an equivalent norm in $V_0^*$, while
		\[
		\|\cdot\|_{\sharp} : V^* \to [0,+\infty), \qquad
		v \mapsto (\|\nabla \mathcal{N}_{\alpha\beta}(v-\overline{v})\|_\bH^2 + |\overline{v}|^2)^\frac 12
		\]
		is an equivalent norm in $V^*$. In particular, there exist two constants $C_1, C_2 > 0$ only depending on $\alpha$ such that
		\begin{align*}
			C_1\|v\|_* \leq \|v\|_{V_0^*} \leq C_2\|v\|_*, & \qquad \forall  \: v \in V_0^*, \\
			C_1\|v\|_\sharp \leq \|v\|_{V^*} \leq C_2\|v\|_\sharp, & \qquad \forall  \: v  \in V^*.
		\end{align*}
	\end{prop}
	\begin{proof}
		The fact that both the above functionals are indeed norms is immediate. Let us prove the first claim. Let $v \in V^*_0$. We have
		\[
		\begin{split}
			\alpha\|v\|_{*}^2 & = \alpha\|\nabla \mathcal{N}_{\alpha\beta}v\|_\bH^2 \\
			& =  \ip{A_{\alpha\beta}\mathcal{N}_{\alpha\beta}v}{\mathcal{N}_{\alpha\beta}v}_{V^*_0,V_0} - \beta\|\mathcal{N}_{\alpha\beta}v\|_{H}^2 \\
			& \leq \|v\|_{V^*_0}\|\mathcal{N}_{\alpha\beta}v\|_{V_0} \\
			& \leq \dfrac{1}{2\alpha} \|v\|_{V^*_0}^2 + \dfrac{\alpha}{2}\|v\|_{*}^2.
		\end{split}
		\]
		Conversely, let $w \in V_0$ be of unitary norm. Then
		\[
		\begin{split}
			|\langle v, w \rangle_{V_0^*,V_0}| & = |\langle A_{\alpha \beta} \mathcal{N}_{\alpha \beta}v, w \rangle_{V_0^*,V_0}| \\
			& \leq  \alpha |(\nabla \mathcal{N}_{\alpha \beta}v, \nabla w)_\bH| + \beta|(\mathcal{N}_{\alpha \beta}v, w)_H| \\
			& \leq \alpha\|\nabla \mathcal{N}_{\alpha \beta}v\|_\bH\|\nabla w\|_\bH + \beta\|\mathcal{N}_{\alpha \beta}v\|_H\|w\|_H \\
			& \leq (\alpha + C_P^2)\|v\|_*,
		\end{split}
		\]
		where $C_P$ denotes the Poincaré constant of $\Td$. Hence,
		\[
		\alpha\|v\|_* \leq \|v\|_{V_0^*} \leq (\alpha + C_P^2)\|v\|_*, \qquad \forall \: v \in V_0^*.
		\]
		for all arbitrary but fixed $\alpha > 0$ and $\beta \in [0,1]$. Let now $v \in V^*$ and $w \in V$ be of unitary norm.
        In the general case, we have, recalling that $\overline{v} = |\Td|^{-1}\langle v, 1 \rangle_{V^*,V}$,
		\[
		\begin{split}
			\alpha^2\|v\|^2_\sharp & = \alpha^2 \|v - \overline{v}\|^2_* + \alpha^2|\overline{v}|^2 \\
			& \leq \|v - \overline{v}\|_{V_0^*}^2 + \alpha^2|\Td|^{-2}\|v\|_{V^*}^2\|1\|_V^2 \\
			&\leq
			(1+C_P^2)
			\|v - \overline{v}\|_{V^*}^2 + \alpha^2|\Td|^{-1}\|v\|_{V^*}^2\\
			& \leq 2(1+C_P^2)\|v\|^2_{V^*}
			 + 2(1+C_P^2)\|\overline{v}\|_{V^*}^2 + \alpha^2|\Td|^{-1}\|v\|_{V^*}^2 \\
			& \leq \left(4(1+C_P^2)
			 + \alpha^2|\Td|^{-1}\right)\|v\|_{V^*}^2.
		\end{split}
		\]
		Conversely, we have
		\[
		\begin{split}
			|\langle v, w \rangle_{V^*,V}| & \leq |\langle A_{\alpha \beta} \mathcal{N}_{\alpha \beta}(v-\overline{v}), w \rangle_{V^*,V}| + |(\overline{v}, w)_H| \\
			& \leq  \alpha |(\nabla \mathcal{N}_{\alpha \beta}(v-\overline{v}), \nabla w)_\bH| + \beta|(\mathcal{N}_{\alpha \beta}(v-\overline v), w-\overline{w})_H|+ |(\overline{v}, w)_H| \\
			& \leq \alpha\|\nabla \mathcal{N}_{\alpha \beta}(v-\overline{v})\|_\bH\|\nabla w\|_\bH + \beta\|\mathcal{N}_{\alpha \beta}(v-\overline{v})\|_H\|w-\overline{w}\|_H + \|\overline{v}\|_H\|w\|_H \\
			& \leq \alpha\|\nabla \mathcal{N}_{\alpha \beta}(v-\overline{v})\|_\bH + C_P^2\|\nabla \mathcal{N}_{\alpha \beta}(v-\overline{v})\|_H + |\Td|^\frac12 |\overline{v}| \\
			& \leq (\alpha + C_P^2 + |\Td|^\frac12)\|v\|_\sharp ,
		\end{split}
		\]
		hence
		\[
		\frac{\alpha}{\sqrt{4(1+C_P^2)
			 + \alpha^2|\Td|^{-1}}}
			 \|v\|_\sharp \leq \|v\|_{V^*} \leq (\alpha + C_P^2 + |\Td|^\frac12)\|v\|_\sharp, \qquad \forall \: v \in V^*,
		\]
		and the claims are proved.
	\end{proof}

	\subsection{Basic assumptions} \label{ssec:assumptions}
	The following assumptions are in order throughout the paper, regardless of the specific problem of interest.
	\begin{enumerate}[label = \textbf{(A\arabic*)}, ref = \textbf{(A\arabic*)}]
		\item\label{hyp:potential} The potential $F$ admits the decomposition
		$$
		F(s)= \Psi(s)+ R(s),
		$$
		where $\Psi: [-1,1]\to \mathbb{R}$, called singular part of $F$, is a strictly convex function satisfying the property $\Psi\in C^0([-1,1])\cap C^2((-1,1))$.
		Moreover, we assume that
		\begin{equation*}
			\qquad	\lim_{s \to -1^+} \Psi'(s)
			= -\infty, \qquad \lim_{s \to 1^-} \Psi'(s) = +\infty, \qquad
			\lim_{s \to -1^+} \Psi''(s) = +\infty, \qquad \lim_{s \to 1^-} \Psi''(s) = +\infty.
		\end{equation*}
		In the following, we shall consider, without relabeling,
		an extension of $\Psi$ over the whole real line by setting
		$\Psi(s) = +\infty$ whenever $s \notin [-1,1]$.
		The function $R$, called regular part of $F$,
		is such that $R\in C^2(\mathbb{R})$ and
		$$
		|R''(s)|\leq C_R,\qquad \forall\, s\in\mathbb{R},
		$$
		for some given $C_R > 0$. Without loss of generality, we assume that $F$ is non-negative.\\[0.3\baselineskip]
		\item\label{hyp:G} The operator
		$\b{G}:\mathbf B^{0,\infty}_1 \to \cL^2(U,\b{H})$ satisfies
		\[
		\b{G}(\psi)[u_k]= \b{g}_k(\psi) =
		(g_k^i(\psi))_{i=1}^d
		\qquad\forall\,k\in\enne_+\quad\forall\,\psi\in\mathbf B^{0,\infty}_1,
		\]
		where the sequence
		$\{\b{g}_k\}_{k\in\enne_+}\subset \b{W}^{2,\infty}(-1,1)$ satisfies
		\[
		L^2_\b{G}:= \sum_{k=1}^{+\infty}\norm{\b{g}_k}_{\b{W}^{2,\infty}(-1,1)}^2
		<+\infty.
		\]
	\end{enumerate}
	It is straightforward to observe that, on account of the representation given by Assumption \ref{hyp:G}, one can introduce formal operators acting on the operator $\b{G}$. Since $\mathbf B^{0,\infty}_1 \subset H$, it is possible to define the weak divergence process
	\[
	\div \b{G} : \mathbf B^{0,\infty}_1 \to \cL^2(U, V_0^*)
	\]
	acting according to
	\[
	\langle \div \b{G}(\psi)[u_k], \xi \rangle_{V_0^*,V_0} := ( \b{g}_k(\psi),\, \nabla \xi)_H, \quad k\in\enne_+.
	\]
	On elements of $\mathbf B^{0,\infty}_1 \cap V$, it is straightforward to prove by interpolation that
	\[
	\div \b{G} : \mathbf B^{0,\infty}_1 \cap V \to \cL^2(U, H_0)
	\]
	is well defined and can be interpreted strongly, namely
	\[
	\div \b{G}(\psi)[u_k] := \div \b{g}_k(\psi) = \b{g}'_k(\psi) \cdot \nabla \psi\,, \quad k\in\enne_+.
	\]
	This is the formal interpretation of the stochastic diffusions in \eqref{eq:ch} and \eqref{eq:ac}.
	Let us analyze the Lipschitz continuity of the divergence operator we just introduced. First, we point out why we introduce the divergence operator in a weak setting. The strong version acting on functions of $\mathbf B^{0,\infty}_1 \cap V$ is only locally Lipschitz continuous with respect to a very strong norm, as shown below.
	\begin{lem} \label{lem:divlip1}
		The operator $\div \b{G}:\mathbf B^{0,\infty}_1 \cap V \to \cL^2(U, H_0)$
		is locally Lipschitz-continuous,
		on bounded subsets of $\mathbf B^{0,\infty}_1 \cap V$,
		with respect to the natural norm of $L^\infty(\Td) \cap V$.
		More precisely, for every $v,w\in \mathbf B^{0,\infty}_1 \cap V$ it holds that
		\[
		\|\div\b{G}(v) - \div \b{G}(w)\|_{\cL^2(U, H_0)}^2
		\leq
		2L_\b{G}^2\left(\|\nabla w(v-w)\|^2_{\b{L}^2(\Td)}+
		\|\nabla(v-w)\|^2_H\right).
		\]
	\end{lem}
	\begin{proof}
		Fix $R > 0$ and let $v,\,w \in \mathbf B^{0,\infty}_1 \cap \mathbf B^{1,2}_R$. Then
		\[
		\begin{split}
			\|\div\b{G}(v) - \div \b{G}(w)\|_{\cL^2(U, H_0)}^2&  = \sum_{k  = 1}^{+\infty} \|\left( \div\b{G}(v) - \div \b{G}(w) \right)[u_k]\|^2_H  \\
			& = \sum_{k  = 1}^{+\infty} \|\b{g}'_k(v)\cdot \nabla v - \b{g}'_k(w)\cdot \nabla w\|^2_H \\
			& = \sum_{k  = 1}^{+\infty} \|(\b{g}'_k(v)-\b{g}'_k(w)) \cdot \nabla w + \b{g}'_k(w) \cdot \nabla(v-w)\|^2_H \\
			& \leq 2\sum_{k  = 1}^{+\infty} \|\nabla w\|_H^2\| \b{g}'_k(v)-\b{g}'_k(w)\|^2_{\b{L}^\infty(\Td)}+ \|\b{g}'_k(w)\|_{\b{L}^\infty(\Td)}^2 \|\nabla(v-w)\|^2_H \\
			& \leq 2\sum_{k  = 1}^{+\infty} \|\nabla w\|_H^2\| \b{g}'_k(v)-\b{g}'_k(w)\|^2_{\b{L}^\infty(\Td)}+ \|\b{g}'_k(w)\|_{\b{L}^\infty(\Td)}^2 \|\nabla(v-w)\|^2_H \\
			& \leq 2R^2\sum_{k  = 1}^{+\infty} \|\b{g}''_k\|_{\b{L}^\infty(-1,1)}^2\|v-w\|_{L^\infty(\Td)}^2 +  2\sum_{k  = 1}^{+\infty} \|\b{g}'_k\|_{\b{L}^\infty(-1,1)}^2\|v-w\|_{V}^2 \\
			& \leq 2L_\b{G}^2(R^2+1)\max(\|v-w\|_{L^\infty(\Td)}^2,\,\|v-w\|_{V}^2),
		\end{split}
		\]
		and the conclusion follows.
	\end{proof} \noindent
	However, the natural extension of $\div \b{G}$ over the whole ball $\mathbf{B}^{0,\infty}_1$ is globally Lipschitz continuous with respect to the usual metric of $L^2(\Td)$.
	\begin{lem} \label{lem:divlip2}
		The operator $\div \b{G} : \mathbf{B}^{0,\infty}_1 \to \cL^2(U,V_0^*)$ is globally Lipschitz continuous with respect to the $H$-metric on $\mathbf{B}^{0,\infty}_1$. The Lipschitz constant is bounded above by $L_\b{G}$.
	\end{lem}
	\begin{proof}
		Fix any $v,\,w \in \mathbf{B}^{0,\infty}_1$. Then
		\[
		\begin{split}
			\|\div\b{G}(v) - \div \b{G}(w)\|_{\cL^2(U, V_0^*)}^2&
			= \sum_{k  = 1}^{+\infty}
			\sup_{\substack{\xi \in V_0 \\ \|\xi\|_{V_0}=1}}
			\left|
			(\b{g}_k(v)-\b{g}_k(w), \nabla \xi)_H\right|^2  \\
			& \leq \sum_{k  = 1}^{+\infty} \|\b{g}_k(v)-\b{g}_k(w)\|^2_H \\
			& \leq \sum_{k  = 1}^{+\infty} \|\b{g}'_k\|_{\b{L}^\infty(-1,1)}^2\|v-w\|_H^2 \\
			& \leq L_\b{G}^2\|v-w\|_H^2,
		\end{split}
		\]
		and the conclusion follows.
	\end{proof} \noindent
	Furthermore, observe that for any given $\psi \in\mathbf B^{0,\infty}_1 \cap H^2(\Td)$ we have that
	\[
	\nabla \div \b{g}_k(\psi) = \nabla(\b{g}'_k(\psi) \cdot \nabla \psi) = \b{g}''_k(\psi) |\nabla \psi|^2 + D^2\psi\,\b{g}'_k(\psi) \in \b{H}
	\]
	by the H\"{o}lder inequality in both two and three dimensions.
	This implies that the restriction of $\div \b G$ to $\mathbf B^{0,\infty}_1 \cap H^2(\Td)$ is well defined as an operator
	\[
	\div\b G: \mathbf B^{0,\infty}_1 \cap H^2(\Td) \to \cL^2(U, V_0)
	\]
	and the gradient process
	\[
	\nabla \div \b{G} :  \mathbf B^{0,\infty}_1 \cap H^2(\Td) \to \cL^2(U, \b{H})
	\]
	is also well defined by
	\[
	\nabla \div \b{G}(\psi)[u_k] := \nabla \div \b{g}_k(\psi)  = \b{g}''_k(\psi) |\nabla \psi|^2 + D^2\psi\,\b{g}'_k(\psi), \quad k \in \enne_+.
	\]
	\subsection{Main results}
	In order to prove the forthcoming results, it turns out to be useful to study the mixed problem \eqref{eq:chac} whose notion of strong solution is given by
	\begin{defin}[strong solution to the mixed problem] \label{def:solCH}
		Let $\alpha \geq 0$ and $\beta \in [0,1]$ with $\alpha+\beta>0$.
		Let $p \geq 2$ and assume that $\varphi_0$ satisfies
		\[
		\varphi_0 \in L^p(\Omega, \cF_0, \P; V), \qquad
		F(\varphi_0) \in L^\frac p2(\Omega, \cF_0, \P; L^1(\Td)), \qquad
		\exists\,\delta_0\in(0,1):\;|\overline{\varphi_0}|\leq1-\delta_0\;\;\P\text{-a.s.}
		\]
		A strong solution to the mixed conservative
		Cahn--Hilliard--Allen--Cahn problem
		\eqref{eq:chac} originating from the initial datum $\varphi_0$
		is a stochastic process $\varphi$ such that
		\begin{align}
			\label{phi}
			&\varphi \in L^p_\cP(\Omega; C^0([0,T]; H))\cap
			L^p_w(\Omega; L^\infty(0,T; V)) \cap
			L^p_\cP(\Omega; L^2(0,T; H^2(\Td)))\,,\\
			& |{\varphi}(\omega, x, t)| < 1
			\text{ for a.a.~}(\omega, x,t) \in \Omega \times \OO \times (0,T)\,, \\
			\label{mu}
			&\mu:=-\Delta\varphi+F'(\varphi)
			\in L^{\frac p2}_\cP(\Omega; L^2(0,T; H))\,,
			\qquad\alpha\mu\in L^{\frac p2}_\cP(\Omega; L^2(0,T; V))\,,
			\\
			\label{initial}
			&\varphi(0) = \varphi_0\,,
		\end{align}
		and
		\begin{align} \label{eq:weakCH_ps}
			\notag
			&( \varphi(t),\psi)_H +
			\int_0^t\!\int_{\Td}
			\left[
			\alpha\nabla  \mu(s)\cdot \nabla \psi +
			\beta(\mu(s)-\overline{\mu(s)})\psi\right]\,\d s\\
			&\qquad= ( \varphi(0),\psi)_{H} +
			\left(\int_0^t \div \b{G}( \varphi(s))\,\d  W(s), \psi\right)_{H}
		\end{align}
		for every $\psi\in V$, $t \in [0,T]$, $\P$-almost surely.
	\end{defin} \noindent
\begin{remark}
	Observe that, setting $\alpha = 1$ and $\beta = 0$ in \eqref{eq:chac}
	(resp.~\eqref{eq:weakCH_ps}), we recover the conservative
	Cahn--Hilliard problem \eqref{eq:ch} (resp.~its weak formulation).
	Viceversa, setting $\alpha = 0$ and $\beta = 1$ yields the conservative
	Allen--Cahn system \eqref{eq:ac}.
	That being said, for technical reasons
	we will cover initially the mixed problem with $\alpha>0$
	for any arbitrary $\beta\in[0,1]$, thus including as
	trivial corollary the analysis of the Cahn-Hilliard model.
	Then the analysis of the limit case $\alpha=0$ and $\beta>0$ will follow
	from a vanishing viscosity argument $\alpha\searrow0$ in the mixed problem.
\end{remark}\noindent
As the uniqueness of solutions is not trivial in the case $\alpha=0$,
we also include a notion of martingale solution to the Allen--Cahn problem \eqref{eq:ac}
corresponding to the choices $\alpha=0$ and $\beta=1$.
\begin{defin}[martingale solution to the Allen--Cahn problem]
	\label{def:martsolAC}
	Let $p \geq 2$ and let the initial condition $\varphi_0$ satisfy
	\begin{equation*}
		\label{eq:acphi0}
		\varphi_0 \in L^p(\Omega,\cF_0, \P; V), \qquad
		F(\varphi_0)\in L^{\frac p2}(\Omega,\cF_0, \P; L^1(\Td)),
		\qquad
	\exists\,\delta_0\in(0,1):\;|\overline{\varphi_0}|\leq1-\delta_0\;\;\P\text{-a.s.}
	\end{equation*}
	A martingale solution to the conserved Allen--Cahn
	problem \eqref{eq:ac} originating from the law
	of the initial datum $\varphi_0$ is a family
	\[
	\left(\widehat\Omega, \widehat\cF, (\widehat\cF_t)_{t\in[0,T]}, \widehat\P,
	\widehat W, \widehat\varphi\right)\,,
	\]
	where $(\widehat\Omega, \widehat\cF, (\widehat\cF_t)_{t\in[0,T]}, \widehat\P)$
	is a filtered
	probability space satisfying the usual conditions; $\widehat W$
	is a cylindrical Wiener processes on $U$; the process $\widehat \varphi$ satisfies
	\begin{align}
		\label{eq:acphi_hat}
		&\widehat\varphi \in L^p_\cP(\widehat\Omega; C^0([0,T]; H))\cap
		L^p_w(\widehat\Omega; L^\infty(0,T; V)) \cap
		L^p_\cP(\widehat\Omega; L^2(0,T; H^2(\Td)))\,,\\
		& |\widehat{\varphi}(\omega, x, t)| < 1 \text{ for a.a. }
		(\omega, x,t) \in \hom \times \OO \times (0,T)\,, \\
		\label{eq:acmu_hat}
		&\widehat\mu:=-\Delta\widehat\varphi+F'(\widehat\varphi) \in
		L^{\frac p2}_\cP(\widehat\Omega; L^2(0,T; H))\,,\\
		\label{eq:acinitial_hat}
		&\widehat \varphi(0)\laweq\varphi_0
		\quad\text{on } V;
	\end{align}
	and it holds that
	\begin{equation} \label{eq:weakAC}
		(\widehat \varphi(t),\psi)_H +
		\int_0^t\!\int_{\Td} (\widehat \mu(s)-\overline{\widehat \mu(s)})\psi\,\d s
		= (\widehat \varphi(0),\psi)_{H} +
		\left(\int_0^t \div \b{G}(\widehat \varphi(s))\,\d \widehat W(s), \psi\right)_{H}
	\end{equation}
	 for every $\psi\in V$, for all $t\in[0,T],$ $\P\text{-almost surely.}$
\end{defin} \noindent
First, let us state two a priori uniqueness results.
	\begin{thm}[uniqueness for the mixed problem] \label{thm:uniqueCH}
		Let $d \in \{1,2,3\}$, let $\alpha > 0$ and $\beta \in [0,1]$,
		let $p \geq 2$, and assume \ref{hyp:potential}--\ref{hyp:G}.
		For $i = 1,\,2$, consider two initial conditions $\varphi_{0i}$
		defined on the same stochastic basis $(\Omega,\cF,(\cF_t)_{t\in[0,T]},\P)$,
		complying with the hypotheses listed in Definition \ref{def:solCH},
		and such that $\overline{\varphi_{01}}=\overline{\varphi_{02}}$.
		Let $\varphi_i$ denote some strong solution of \eqref{eq:chac},
		originating from the initial datum $\varphi_{0i}$,
		in the sense of Definition~\ref{def:solCH}.
		Then, the following continuous dependence estimate holds
		\begin{align*}
			\|\varphi_1 - \varphi_2\|_{L^p_\cP(\Omega;
				C^0([0,T];V^*)) \cap L^p_\cP(\Omega; L^2(0,T;V))}
			\leq C \|\varphi_{01}-\varphi_{02}\|_{L^p(\Omega; V^*)}.
		\end{align*}
		In particular, the periodic stochastic mixed problem
		\eqref{eq:chac} (and the periodic stochastic Cahn--Hilliard
		problem \eqref{eq:ch}) admits at most one solution.
	\end{thm} \noindent

	\begin{thm}[uniqueness for the Allen--Cahn problem] \label{thm:uniqueAC}
		Let $d = 1$, let $\alpha=0$ and $\beta\in(0,1]$, let $p \geq 2$, assume
		\ref{hyp:potential}--\ref{hyp:G}, and that
		$L_{\b G}^2<\frac12$.
		Then,
		for every initial condition $\varphi_{0}$
		defined on the stochastic basis $(\Omega,\cF,(\cF_t)_{t\in[0,T]},\P)$,
		complying with the hypotheses listed in Definition \ref{def:solCH},
		and for every pair of strong solutions $\varphi_i$ of \eqref{eq:chac},
		originating from the same initial datum $\varphi_{0}$,
		in the sense of Definition~\ref{def:solCH},
		the following uniqueness result holds:
		\[
		\P\left(\|\varphi_1(t)-\varphi_2(t)\|_H=0 \quad\forall\,t\in[0,T]\right) =1.
		\]
		In particular, the
		periodic stochastic Allen--Cahn problem
		\eqref{eq:ac} admits at most one solution
	\end{thm}

	\begin{remark}
		Proving uniqueness of martingale solutions to the Allen--Cahn problem in a dimension greater than one seems challenging due to the $L^2$-gradient flow structure of the drift part of the system and the fact that, in order to exploit the Lipschitz-continuity given by Lemma \ref{lem:divlip1}, a higher-order control on the difference of two solutions is needed.
\end{remark} \noindent
The main results about existence of solutions are presented hereafter.
	\begin{thm}[existence of solutions to the Cahn--Hilliard problem] \label{thm:weakCH}
		Let $d \in \{1, 2,3\}$, let $p \geq 2$, and assume
		\ref{hyp:potential}--\ref{hyp:G}.
		Then, for any $\alpha > 0$ and $\beta \in [0,1]$,
		there exists a strong solution to the periodic stochastic
		mixed problem with conservative noise \eqref{eq:chac}.
		In particular, there exists a strong solution to the conservative
		Cahn--Hilliard problem \eqref{eq:ch}.
	\end{thm}

\begin{thm}[existence of solutions to the Allen--Cahn problem] \label{thm:weakAC}
	Let $d \in \{1, 2, 3\}$, let $\alpha=0$,
	let $p \geq 2$, and assume \ref{hyp:potential}--\ref{hyp:G}.
	Then, there exists a constant $C_0>0$,
	only depending on the structural
	data, such that if $L_{\b G}^2<C_0$ then
	there exists a martingale solution to the periodic stochastic
	Allen--Cahn problem with conservative noise \eqref{eq:ac}.
	If also $d = 1$ and $L_{\b G}^2<\min\{\frac12,C_0\}$,
	then the solution is unique and probabilistically-strong.
\end{thm}

\section{Proof of Theorems \ref{thm:uniqueCH} and \ref{thm:uniqueAC}} \label{sec:unique}
\subsection{The mixed problem.}
Here, we prove an a priori continuous dependence estimate for \eqref{eq:chac}.
Let $\alpha > 0$ and $\beta \in [0,1]$ be arbitrary but fixed. Observe that, as a byproduct by setting $\alpha = 1$ and $\beta = 0$, we have that problem \eqref{eq:ch} admits at most one solution.
 Let $\varphi_{01}$ and $\varphi_{02}$ be two initial states satisfying the assumptions listed in Definition \ref{def:solCH} and such that $\overline{\varphi_{01}} = \overline{\varphi_{02}}$, and consider two different solutions to the stochastic mixed problem \eqref{eq:chac}, labeled with the index $i \in \{1,2\}$, that is
\begin{equation*}
	\begin{cases}
		\d \varphi_i +
		A_{\alpha\beta} \mu_i\,  \d t
		= \div \left( \b{G}(\varphi_i)\right) \d W & \quad \text{ in }\T^d \times (0,T) \\
		\mu_i = -\Delta \varphi_i + F'(\varphi_i) & \quad \text{ in }\T^d \times (0,T) \\
		\varphi_i(\cdot \:, 0) = \varphi_{0i} & \quad \text{ in } \Td.
	\end{cases}
\end{equation*}
Define then
\begin{align*}
	\varphi & := \varphi_1 - \varphi_2, \\
	\mu & := \mu_1 - \mu_2, \\
	\varphi_{0} & := \varphi_{01}-\varphi_{02}.
\end{align*}
By testing the equation for $\varphi$ by $1$, recalling that $\overline{\varphi_0}=0$ by assumption,
it is immediate to see that $\overline{\varphi(t)}=0$ for all $t\in[0,T]$.
Hence, if we apply the It\^{o} lemma in its version for twice differentiable functionals (see \cite[Theorem 4.32]{dapratozab}) to the functional
\[
\varphi \mapsto \dfrac{\alpha}{2}\|\nabla \mathcal{N}_{\alpha \beta}\varphi\|_{\bH}^2 + \dfrac{\beta}{2}\|\mathcal{N}_{\alpha \beta}\varphi\|^2_H
=\frac12\langle \varphi, \mathcal{N}_{\alpha \beta}\varphi\rangle,
\]
recalling that $\mathcal{N}_{\alpha \beta}=A_{\alpha\beta}^{-1}$ we obtain
\begin{align}
	\nonumber
	&\dfrac{\alpha}{2}\|\nabla \mathcal{N}_{\alpha \beta}\varphi(t)\|_{\bH}^2  + \dfrac{\beta}{2}\|\mathcal{N}_{\alpha \beta}\varphi(t)\|^2_H
	+\int_0^t \left[ \|\nabla \varphi(\tau)\|^2_H
	+ \left( \varphi(\tau), F'(\varphi_{1}(\tau))
	-F'(\varphi_{2}(\tau))\right)_H \right] \: \mathrm{d}\tau \\
	\nonumber
	&\qquad= \dfrac{\alpha}{2}\|\nabla \mathcal{N}_{\alpha \beta}\varphi_0 \|_{\bH}^2   + \dfrac{\beta}{2}\|\mathcal{N}_{\alpha \beta}\varphi_0\|^2_H
	+ \int_0^t \left(\varphi(\tau),\, \mathcal{N}_{\alpha\beta}[\div \b{G}(\varphi_{1}(\tau))
	- \div \b{G}(\varphi_{2}(\tau))]\,\mathrm{d}W(\tau)\right)_H  \\
	&\qquad\qquad+ \dfrac{\alpha}{2}\int_0^t
	\|\nabla \mathcal{N}_{\alpha\beta}( \div \b{G}(\varphi_{1}(\tau))
	- \div \b{G}(\varphi_{2}(\tau)))\|^2_{\cL^2(U, H)} \: \mathrm{d}\tau\\
	&\qquad\qquad+
	\dfrac{\beta}{2}\int_0^t
	\| \mathcal{N}_{\alpha\beta}( \div \b{G}(\varphi_{1}(\tau))
	- \div \b{G}(\varphi_{2}(\tau)))\|^2_{\cL^2(U, H)} \: \mathrm{d}\tau
	\label{eq:uniqch1}
\end{align}
Observe that for all $\tau\in[0,T]$, by \ref{hyp:G} and interpolation,
\begin{align}
	\nonumber
	\left( \varphi(\tau),
	F'(\varphi_{1}(\tau)) -F'(\varphi_{2}(\tau))\right)_H
	& \geq \left( \varphi(\tau),
	R'(\varphi_{1}(\tau)) -R'(\varphi_{2}(\tau))\right)_H\\
	\nonumber
	&\geq -C_R\|\varphi(\tau)\|_H^2 \\
	& \geq -\delta_1\|\nabla \varphi(\tau)\|^2_\bH - C\|\varphi(\tau)\|^2_*,
	\label{eq:uniqch2}
\end{align}
where $\delta_1 > 0$ is arbitrary and $C$ depends on $\delta_1$. Moreover, we have by Lemma \ref{lem:divlip2}
\begin{align}
	\nonumber
	&\|\nabla \mathcal{N}_{\alpha\beta}( \div \b{G}(\varphi_{1}(\tau))
	- \div \b{G}(\varphi_{2}(\tau)))\|^2_{\cL^2(U, \bH)}
	+\|\mathcal{N}_{\alpha\beta}( \div \b{G}(\varphi_{1}(\tau))
	- \div \b{G}(\varphi_{2}(\tau)))\|^2_{\cL^2(U, H)}\\
	\nonumber
	&\qquad \leq  C\|\div \b{G}(\varphi_{1}(\tau))- \div \b{G}(\varphi_{2}(\tau))\|^2_{\cL^2(U,V_0^*)} \\
	\nonumber
	& \qquad\leq C\|\varphi(\tau)\|^2_H \\
	& \qquad\leq \delta_1\|\nabla\varphi(\tau)\|^2_\bH + C\|\varphi(\tau)\|^2_*.
	\label{eq:uniqch3}
\end{align}
Fix $\delta_1 = \frac 14$. Let us exploit \eqref{eq:uniqch2} and \eqref{eq:uniqch3} in \eqref{eq:uniqch1}, also recalling that $\|\varphi(t)\|_*=\|\varphi(t)\|_\sharp$
for all $t\in[0,T]$:
raising the result to the power $\frac p2$, then taking suprema and expectations yields
\begin{align*}
	&\E \supp \|\varphi(\tau)\|_{\sharp}^p
	+\E \left| \int_0^t  \|\nabla \varphi(\tau)\|^2_H \: \mathrm{d}\tau \right|^\frac p2  \\
	&\qquad\leq C \bigg[ \E\|\varphi_0\|_{\sharp}^p
	+ \E \left| \int_0^t \|\varphi(\tau)\|^2_\sharp \: \mathrm{d}\tau \right|^\frac p2\\
	&\qquad\qquad+ \E \sups \left|\int_0^s \left(\varphi(\tau),\,
	\mathcal{N}_{\alpha\beta}[\div \b{G}(\varphi_{1}(\tau))
	- \div \b{G}(\varphi_{2}(\tau))]\,\mathrm{d}W(\tau)\right)_H \right|^\frac p2
	\bigg]
\end{align*}
The Burkholder--Davis--Gundy inequality and \eqref{eq:uniqch3} entail that
\begin{align}
	\nonumber
	& \E \sups \left|\int_0^s
	\left(\varphi(\tau),\,
	\mathcal{N}_{\alpha\beta}[\div \b{G}(\varphi_{1}(\tau))
	- \div \b{G}(\varphi_{2}(\tau))]\,\mathrm{d}W(\tau)\right)_H \right|^\frac p2 \\
	\nonumber
	& \qquad \leq C
	\E \left|\int_0^t \|\varphi(\tau)\|^2_{V^*_0}
	\|\mathcal{N}_{\alpha\beta}[\div \b{G}(\varphi_{1}(\tau))
	- \div \b{G}(\varphi_{2}(\tau))]\|^2_{\cL^2(U,V_0)} \:\d\tau \right|^\frac p4 \\
	\nonumber
	& \qquad \leq C\E \left|\int_0^t \|\varphi(\tau)\|^2_\sharp\|\varphi(\tau)\|^2_{H} \:\d\tau \right|^\frac p4 \\
	\nonumber
	& \qquad \leq C \E\left[ \supp\|\varphi(\tau)\|^\frac p2_{\sharp} \left|\int_0^t \| \varphi(\tau)\|^2_H  \:\d\tau \right|^\frac p4  \right] \\
	\nonumber
	& \qquad \leq C \E\left[ \supp\|\varphi(\tau)\|^p_{\sharp} \right]^\frac{1}{2} \E \left[ \left|\int_0^t \|\varphi(\tau)\|^2_H  \:\d\tau \right|^\frac p2  \right]^\frac 12 \\
	& \qquad \leq \delta_2\E \supp\|\varphi(\tau)\|_{\sharp}^p + \delta_2\E \left|\int_0^t \|\nabla \varphi(\tau)\|^2_\bH  \:\d\tau \right|^\frac p2 + C\E \left|\int_0^t \|\varphi(\tau)\|^2_\sharp  \:\d\tau \right|^\frac p2,
	\label{eq:uniqch4}
\end{align}
where $\delta_2 > 0$ is arbitrary and $C$ depends on $\delta_2$. Setting $\delta_2 = \frac 12$, we arrive at the continuous dependence estimate
through an application of the Gronwall lemma and the Jensen inequality ($p\geq2$).
\subsection{The Allen--Cahn problem.}
This section is devoted to proving Theorem \ref{thm:uniqueAC}, i.e.
the uniqueness for the Allen--Cahn problem \eqref{eq:ac}. Let us point out once and for all that here we only consider the one-dimensional case. Once again, let $\varphi_{01}$ and $\varphi_{02}$ be two initial states satisfying the assumptions listed in Definition \ref{def:solCH} and such that $\overline{\varphi_{01}} = \overline{\varphi_{02}}$,
and let $\varphi_1$ and $\varphi_2$ be the corresponding solutions of the stochastic conserved Allen--Cahn problem \eqref{eq:ac}, namely
\begin{equation*}
	\begin{cases}
		\d \varphi_i + (\mu_i - \overline{\mu_i}) \,\d t = \div \left( \b{G}(\varphi_i)\right) \d W & \quad \text{ in }\T^d \times (0,T) \\
		\mu_i = -\Delta \varphi_i + F'(\varphi_i) & \quad \text{ in }\T^d \times (0,T) \\
		\varphi_i(\cdot \:, 0) = \varphi_{0i} & \quad \text{ in } \Td,
	\end{cases} \qquad i = 1,\,2.
\end{equation*}
Consider then the differences
\begin{align*}
	\varphi & := \varphi_1 - \varphi_2, \\
	\mu & := \mu_1 - \mu_2, \\
	\varphi_{0} & := \varphi_{01}-\varphi_{02}.
\end{align*}
Let $\sigma \geq 0$ be arbitrary and define the functional $\Xi:[0,T] \to \mathbb{R}$
\[
\Xi(t) := \sigma \int_0^t \|\nabla \varphi_2(\tau)\|^2_{\b{L}^\infty(\T^1)} \: \d \tau.
\]
Observe that, since $\varphi_2(t) \in H^2(\T^1)$ for all times $t \in (0,T)$, the functional $\Xi$ is well defined with the understanding that $\Xi(0) = 0$.
By testing the equation for $\varphi$ by $1$,
	recalling that $\overline{\varphi_0}=0$ by assumption,
	it is immediate to see that $\overline{\varphi(t)}=0$ for all $t\in[0,T]$.
	Hence,
	an application of the It\^{o} lemma
	(see \cite[Theorem 4.32]{dapratozab})
	to
	\[
	t \mapsto \dfrac 12 e^{-\Xi(t)}\|\varphi(t)\|_H^2
	\] yields
	\begin{multline}
		\dfrac 12 e^{-\Xi(t)}\|\varphi(t)\|_H^2 + \dfrac{\sigma}{2}\int_0^t e^{-\Xi(\tau)}\|\nabla \varphi_2(\tau)\|^2_{\b{L}^\infty(\T^1)}\|\varphi(\tau)\|_H^2 \: \d \tau + \int_0^t e^{-\Xi(\tau)}\|\nabla \varphi(\tau)\|_\bH^2 \: \d\tau \\
		+ \int_0^t e^{-\Xi(\tau)}\left( \varphi(\tau), F'(\varphi_{1}(\tau)) -F'(\varphi_{2}(\tau))\right)_H \: \d \tau \\= \dfrac 12 \|\varphi_0\|_H^2 + \int_0^t e^{-\Xi(\tau)}\left(\nabla\varphi(\tau),\,
		\b{G}(\varphi_{1}(\tau))
		- \b{G}(\varphi_{2}(\tau))\,\mathrm{d}W(\tau)\right)_{\bH}  \\ + \int_0^t e^{-\Xi(\tau)}\|\div \b{G}(\varphi_1(\tau)) - \div \b{G}(\varphi_2(\tau))\|^2_{\cL^2(U,H)} \: \d \tau
		\label{eq:uniqac1}
	\end{multline}
	Observe that, by \ref{hyp:G},
	\begin{align}
		\nonumber
		\left( \varphi(\tau), F'(\varphi_{1}(\tau)) -F'(\varphi_{2}(\tau))\right)_H
		& \geq
		\left( \varphi(\tau), R'(\varphi_{1}(\tau)) -R'(\varphi_{2}(\tau))\right)_H \\
		&\geq
		-C_R\| \varphi(\tau)\|_H^2.
		\label{eq:uniqac2}
	\end{align}
	Thanks to Lemma \ref{lem:divlip1} and
	the one-dimensional embedding $H^2(\T^1) \embed W^{1,\infty}(\T^1)$ we deduce
	\begin{equation} \label{eq:divestimate}
		\|\div \b{G}(\varphi_1(\tau)) - \div \b{G}(\varphi_2(\tau))\|^2_{\cL^2(U,H)} \leq 2L_\b{G}^2(\|\nabla \varphi_2\|^2_{\b{L}^\infty(\T^1)}\|\varphi(\tau)\|^2_H + \|\nabla \varphi(\tau)\|^2_\bH).
	\end{equation}
	The positivity of the exponential function lets us exploit \eqref{eq:uniqac2} and \eqref{eq:divestimate} in \eqref{eq:uniqac1}, giving
	\begin{multline*}
		\dfrac 12 e^{-\Xi(t)}\|\varphi(t)\|_H^2
		 + \dfrac{\sigma}{2}\int_0^t e^{-\Xi(\tau)}\|\nabla \varphi_2(\tau)\|^2_{\b{L}^\infty(\T^1)}\|\varphi(\tau)\|_H^2 \: \d \tau + \int_0^t e^{-\Xi(\tau)}\|\nabla \varphi(\tau)\|_\bH^2 \: \d\tau \\
		 \leq \dfrac 12 \|\varphi_0\|_H^2 + \int_0^t e^{-\Xi(\tau)}\left(\nabla\varphi(\tau),\,
		 \b{G}(\varphi_{1}(\tau))
		 - \b{G}(\varphi_{2}(\tau))\,\mathrm{d}W(\tau)\right)_{\bH}  \\ + 2L_\b{G}^2\int_0^t e^{-\Xi(\tau)}(\|\nabla \varphi_2\|^2_{\b{L}^\infty(\T^1)}\|\varphi(\tau)\|^2_H + \|\nabla \varphi(\tau)\|^2_\bH) \: \d \tau + C_R\int_0^t e^{-\Xi(\tau)}\| \varphi(\tau)\|_H^2 \: \d \tau.
	\end{multline*}
	As $2L_\b{G}^2 \leq 1$ by assumption, we can choose $\sigma = 4L_\b{G}^2$ to get
	\begin{multline*}
		\dfrac 12 e^{-\Xi(t)}\|\varphi(t)\|_H^2
		+ \left( 1 - 2L_\b{G}^2 \right)\int_0^t e^{-\Xi(\tau)}\|\nabla \varphi(\tau)\|_\bH^2 \: \d\tau \\
		\leq \dfrac 12 \|\varphi_0\|_H^2 + \int_0^t e^{-\Xi(\tau)}\left(\nabla\varphi(\tau),\,
		\b{G}(\varphi_{1}(\tau))
		- \b{G}(\varphi_{2}(\tau))\,\mathrm{d}W(\tau)\right)_{\bH}  + C_R\int_0^t e^{-\Xi(\tau)}\| \varphi(\tau)\|_H^2 \: \d \tau.
	\end{multline*}
	Taking expectations yields
	\[
	\E \supp e^{-\Xi(\tau)}\|\varphi(\tau)\|_H^2
	+ \E\int_0^t e^{-\Xi(\tau)}\|\nabla \varphi(\tau)\|_\bH^2 \: \d\tau
	\leq C\left[  \E \dfrac 12 \|\varphi_0\|_H^2  + C_R\E\int_0^t e^{-\Xi(\tau)}\| \varphi(\tau)\|_H^2 \: \d \tau \right]
	\]
	so that the Gronwall lemma and the positivity of the exponential yield
	\[
	\varphi=0 \quad\text{a.e.~in } \Omega\times(0,T)\times\mathbb T^d.
	\]
	Going then back to the inequality before taking expectations and exploiting
	this information, we can take
	suprema in time and then expectations, obtaining  also
	\[
	\varphi(t)=0 \quad\text{in } H, \quad\forall\,t\in[0,T],\quad\P\text{-a.s.}
	\]
	and uniqueness is proved.
	
\section{Proof of Theorems \ref{thm:weakCH} and \ref{thm:weakAC}} \label{sec:existence}
The present section is devoted to proving the existence result given by Theorem \ref{thm:weakCH}. Throughout the whole section, we deal then with the mixed problem \eqref{eq:chac}, assuming that $\alpha > 0$ and $\beta \in [0,1]$ are arbitrary but fixed. Moreover, in order to make the details clearer, the argument is split into several steps.

At first, we shall assume without explicit mentioning the following additional requirements:
\begin{enumerate}[label=(\roman*)]
	\item\label{i}
	the parameter $p$ appearing in Definition \ref{def:solCH} satisfies $p \geq 4$;
	\item\label{ii} the initial condition $\varphi_0$ satisfies the exponential estimate
	\[
	\E \exp (q\|\varphi_0\|_H^2) < + \infty
	\]
	for all $q > 0$;
	\item\label{iii} it holds that
	\[
		\b{g}_k'(\pm1)=\b{0}\,, \quad
		F'|\b{g}'_k|^2, \, F''|\b{g}'_k|^2\in L^\infty(-1,1) \qquad\forall\,k\in\enne_+,
	\]
	and
	\[
	\sum_{k  = 1}^{+\infty} \left[\|F'|\b{g}'_k|\|^2_{L^\infty(-1,1)}
	+\|F''|\b{g}'_k|^2\|_{L^\infty(-1,1)}
	\right] < +\infty.
	\]
\end{enumerate}
The additional assumptions \ref{i}--\ref{iii} will be removed later, in Subsection~\ref{ssec:higher}.

\subsection{The approximation scheme.} \label{ssec:approx}
On account of the properties of the singular part $\Psi$ listed in Assumption \ref{hyp:potential}, we have that $\Psi$ is a proper, convex and lower semicontinuous function on $[-1,1]$. It is then possibile to introduce an approximation of $\Psi$ through a one-parameter family of globally-defined and non-negative functions $\{\Psi_\lambda\}_{\lambda > 0}$, called the Yosida approximation of $\Psi$. The following are useful properties of the functions $\Psi_\lambda$ (cf. \cite{barbu-monot}):
\begin{enumerate}[label=(\alph*)]
	\item \label{pty:monot}$\Psi_\lambda$ is convex and converges to $\Psi$ pointwise in $\erre$ and monotonically increasing as $\lambda \to 0^+$;
	\item $\Psi_\lambda \in C^{1,1}(\mathbb{R})$ and the Lipschitz constant of $\Psi_\lambda'$ is $\frac{1}{\lambda}$;
	\item \label{pty:monotonicityder} $|\Psi_\lambda'|$ converges to $|\Psi'|$ pointwise in $\erre$ and monotonically increasing as $\lambda \to 0^+$,
	\item $\Psi_\lambda(0) = \Psi_\lambda'(0) = 0$ for all $\lambda > 0$;
	\item \label{pty:belowbound} there exists
	$M>0$ such that, for every $\lambda\in(0,1)$,
	\[
	\Psi_\lambda(s) \geq \dfrac{1}{M}s^2-M \quad\forall\,s\in\mathbb{R}.
	\]
\end{enumerate}
It is then natural to consider the globally defined and regularized potential
\[
F_\lambda(s) := \Psi_\lambda(s) + R(s), \qquad s \in \mathbb{R}.
\]
In order to formulate a suitable approximation of \eqref{eq:ch}, it is also appropriate to regularize the stochastic diffusion $\div\b{G}$, in such a way that its domain is suitably enlarged. Let $J_\lambda$ denote the 1-Lipschitz continuous resolvent operator
\[
J_\lambda: \mathbb{R} \to \mathbb{R}, \qquad J_\lambda(s) := (I+\lambda\Psi')^{-1}(s),
\quad s\in\erre,
\]
and, for any fixed $\xi > 0$ and any $f \in H$, let $\mathcal{R}_\xi f$ be the unique solution of the equation
\begin{equation} \label{eq:elliptic}
	-\xi \Delta u + u = f  \qquad \text{in }\Td
\end{equation}
with periodic boundary conditions. In light of elliptic regularity results, we have that the operator
\[
\mathcal{R}_\xi = (I-\xi \Delta)^{-1}: H \to H^2(\Td)
\]
is well defined, linear and continuous. In particular, we have the following result.
\begin{lem} \label{lem:resolvent} For any $\xi > 0$, $k \geq 0$ and $f \in H^k(\Td)$, we have
	\[
	\|\mathcal{R}_\xi f\|_{H^{k+2}(\Td)} \leq \dfrac{C_1}{\min\{1,\xi\}} \|f\|_{H^k(\Td)},
	\]
	where $C_1 > 0$ does not depend on $\xi$. Moreover, there exists a constant $C_2 > 0$ independent of $\xi$ such that
	\[
	\|\mathcal{R}_\xi f\|_{H^k(\Td)} \leq C_2 \|f\|_{H^k(\Td)}
	\]
	for all $k \in \{0,1,2\}$.
\end{lem}  \noindent
Let us introduce the regularized diffusion
\[
K_{\lambda,\xi}: H^2(\Td) \to \cL^2(U,H), \qquad  K_{\lambda,\xi}(\psi)[u_k]:= \b{g}'_k(J_\lambda(\psi)) \cdot \nabla \mathcal{R}_\xi(\psi).
\]
It is immediate to verify that $K_{\lambda,\xi}$ takes values even in $\cL^2(U,V)$. Before moving on, we analyze an important property of $K_{\lambda,\xi}$.
\begin{lem} \label{lem:lipschitz}
	Let $\Lambda \in (0,1]$. For any $\lambda \in (0,\Lambda)$, $\xi \in (0,\Lambda)$ and $v,\,w \in H^2(\Td)$, the operator $K_{\lambda,\xi}$ satifies
	\[
	\|K_{\lambda,\xi}(v)-K_{\lambda,\xi}(w)\|_{\cL^2(U,H)}^2\leq\dfrac{2L_\b{G}^2}{\xi^2}\left(1 + \|w\|^2_V\right)\|v-w\|^2_H.
	\]
	Moreover, we have
	\[
	\|K_{\lambda,\xi}(v)\|^2_{\cL^2(U,H)} \leq \dfrac{L_\b{G}^2}{\xi^2}\|v\|^2_H, \qquad \|K_{\lambda,\xi}(v)\|^2_{\cL^2(U,H)} \leq C\|v\|^2_V
	\]
	for every $v \in H^2(\Td)$, where $C$ is a constant independent of $\lambda$ and $\xi$.
\end{lem}
\begin{proof}
	In order to prove the first estimate we proceed as follows. Chosen $v,\,w \in H^2(\Td)$,
	\[
	\begin{split}
		\|K_{\lambda,\xi}(v)-K_{\lambda,\xi}(w)\|_{\cL^2(U,H)}^2 & = \sum_{k=1}^{+\infty} \|\left[K_{\lambda,\xi}(v)-K_{\lambda,\xi}(w)\right] [u_k]\|^2_{H} \\
		& = \sum_{k=1}^{+\infty} \|\b{g}'_k(J_\lambda(v)) \cdot \nabla \mathcal{R}_\xi(v) - \b{g}'_k(J_\lambda(w)) \cdot \nabla \mathcal{R}_\xi(w)\|^2_{H} \\
		& \leq 2L_\b{G}^2 \left( \|\nabla \mathcal{R}_\xi(v-w)\|^2_\bH +  \|J_\lambda(v) - J_\lambda(w)\|_H^2\|\nabla \mathcal{R}_\xi w\|_{\b{L}^\infty(\Td)}^2\right) \\
		& \leq \dfrac{2L_\b{G}^2}{\xi^2}\left(1 + \|w\|^2_V\right)\|v-w\|^2_H
	\end{split}
	\]
	where the last line follows by the fact that $J_\lambda$ is 1-Lipschitz continuous, the fact that $R_\xi w \in H^3(\Td) \embed W^{1,\infty}(\Td)$ and Lemma \ref{lem:resolvent}. The second claim can be proven analogously.
\end{proof} \noindent
For any fixed $\Lambda \in (0,1]$, $\lambda \in (0,\Lambda)$ and $\xi \in (0,\Lambda)$, we consider the following regularized stochastic Cahn--Hilliard problem
\begin{equation} \label{eq:strongformreg}
	\begin{cases}
		\d \varphi_{\lambda,\xi} + [-\alpha \Delta \mu_{\lambda,\xi} + \beta(\mu_{\lambda,\xi}-\overline{\mu_{\lambda,\xi}}) ] \,\d t =K_{\lambda,\xi}(\varphi_{\lambda,\xi}) \,\d W & \quad \text{ in }\T^d \times (0,T), \\
		\mu_{\lambda,\xi} = -\Delta \varphi_{\lambda,\xi} + F'_\lambda(\varphi_{\lambda,\xi}) & \quad \text{ in }\T^d \times (0,T), \\
		\varphi_{\lambda,\xi}(\cdot \:, 0) = \varphi_0 & \quad \text{ in } \Td,
	\end{cases}
\end{equation}
endowed with periodic boundary conditions.
We now search for a martingale solution to problem \eqref{eq:strongformreg}. Such a solution enjoys all the properties listed in Definition \ref{def:solCH}, up to substituting the potential $F$ with its Yosida approximation $F_\lambda$ and the diffusion coefficient $\div \b{G}$ with $K_{\lambda,\xi}$.
\subsection{Existence of approximating solutions}
In light of Lemma \ref{lem:lipschitz}, we show the existence of a solution to \eqref{eq:strongformreg} applying the general theory (see, for instance \cite[Theorem 5.1.3]{LiuRo}). Introducing the operator $\mathcal{A}_\lambda: H^2(\Td) \to (H^2(\Td))^*$, which is independent of $\xi$ and whose action is given by
\begin{equation*}
	 \langle \mathcal{A}_\lambda u, v \rangle_{(H^2(\Td))^*, H^2(\Td)} := \alpha(-\Delta u + F'_\lambda(u),\, -\Delta v)_H + \beta(-\Delta u + F'_\lambda(u) - \overline{F'_\lambda(u)}, v),
\end{equation*}
we have the following result.
\begin{lem} \label{lem:Alambda}
	Fixed any $\lambda\in (0,\Lambda)$, the operator $\mathcal{A}_\lambda$ satisfies the following properties:
	\begin{enumerate}[(i)]
		\itemsep0.5em
		\item the operator $\mathcal{A}_\lambda$ is progressively measurable;
		\item (hemicontinuity) the map
		\[
		\eta \mapsto \langle \mathcal{A}_\lambda(\omega,t,v+\eta w), \psi \rangle_{(H^2(\Td))^*,\: H^2(\Td)},\qquad \eta \in \erre
		\]
		is continuous on $\erre$ for every $\omega \in \Omega$, $t \in [0,T]$ and $v,w,\psi \in H^2(\Td)$;
		\item (weak monotonicity) there exists a constant $c_1 > 0$ such that
		\[
		\langle \mathcal{A}_\lambda(\omega,t,v)-\mathcal{A}_\lambda(\omega,t,w), v-w \rangle_{(H^2(\Td))^*,\: H^2(\Td)} \geq -c_1\|v-w\|^2_H
		\]
		for every $\omega \in \Omega$, $t \in [0,T]$ and $v,w \in H^2(\Td)$;
		\item (weak coercivity) there exist two constants $c_2>0$ and $c_2' >0$ and an adapted process $f \in L^1(\Omega \times (0,T))$ such that
		\[
		\langle \mathcal{A}_\lambda(\omega,t,v), v \rangle_{(H^2(\Td))^*,\: H^2(\Td)} \geq c_2\|v\|^2_{H^2(\Td)} - c_2'\|v\|_H^2 - f(\omega, t)
		\]
		for every $\omega \in \Omega$, $t \in [0,T]$ and $v \in H^2(\Td)$;
		\item (boundedness) there exists a constant $c_3 >0$ and an adapted process $g \in L^2(\Omega \times (0,T))$ such that
		\[
		\|\mathcal{A}_\lambda(\omega,t,v)\|_{(H^2(\Td))^*} \leq c_2\|v\|_{H^2(\Td)} + g(\omega,t)
		\]
		for every $\omega \in \Omega$, $t \in [0,T]$ and $v \in H^2(\Td)$.
	\end{enumerate}
\end{lem}
\begin{proof}
	Progressive measurability is immediate. In order to prove hemicontinuity, observe that, fixed $\omega \in \Omega$, $t \in [0,T]$ and $v,w,\psi \in H^2(\Td)$
	\begin{multline*}
		\langle \mathcal{A}_\lambda(\omega,t,v+\eta w), \psi \rangle_{(H^2(\Td))^*,\: H^2(\Td)} \\ = \alpha(\Delta v, \Delta \psi)_H + \alpha\eta(\Delta w, \Delta \psi)_H - (F'_\lambda(v+\eta w), \Delta \psi)_H  \\ + \beta(-\Delta v, \psi)_H + \beta\eta(-\Delta w,\psi )_H + (F'_\lambda(v + \eta w) - \overline{F'_\lambda(v + \eta w)}, \psi)_H.
	\end{multline*}
	Since $F_\lambda'$ is Lipschitz continuous, continuity with respect to $\eta \in \erre$ is easily proved. Next, we move to weak monotonicity. Notice that
	\[
	\begin{split}
		& \langle \mathcal{A}_\lambda(\omega,t,v)-\mathcal{A}_\lambda(\omega,t,w), v-w \rangle_{(H^2(\Td))^*,\: H^2(\Td)}  \\ & \hspace{2cm} = \alpha\|\Delta v - \Delta w\|^2_H + \alpha(F'_\lambda(v)-F'_\lambda(w), \Delta w-\Delta v)_H + \beta\|\nabla v - \nabla w\|^2_\bH \\ & \hspace{5cm}+ \beta(F'_\lambda(v)-F'_\lambda(w), v-w)_H + \beta(\overline{F'_\lambda(w)-F'_\lambda(v)}, v-w)_H\\
		& \hspace{2cm} \geq \alpha\|\Delta v - \Delta w\|^2_H - \alpha\left( \dfrac 1 \lambda + C_R \right)\| v-w\|_H\| \Delta w-\Delta v\|_H + \beta\|\nabla v - \nabla w\|^2_\bH \\ & \hspace{5cm} - \beta\left( \dfrac 1 \lambda + C_R \right)\|v-w\|_H\|v-w\|_{L^1(\Td)}\\
		& \hspace{2cm} \geq \dfrac \alpha 2\|\Delta v - \Delta w\|^2_H - \left[ \beta|\Td|^\frac 12 + \dfrac \alpha2\left( \dfrac 1 \lambda + C_R \right) \right]\left( \dfrac 1 \lambda + C_R \right)\|v-w\|_H^2
	\end{split}
	\]
	on account of the convexity of $\Psi_\lambda$ and Assumption \ref{hyp:potential}. Hence, we have shown that (iii) holds with $c_1 = [ \beta|\Td|^\frac 12 + \frac \alpha2( \frac 1 \lambda + C_R ) ]( \frac 1 \lambda + C_R )$. Let $C_P$ denote the Poincaré constant of $\Td$. Coercivity is shown by the following computation:
	\[
	\begin{split}
		& \langle \mathcal{A}_\lambda(\omega,t,v), v \rangle_{(H^2(\Td))^*,\: H^2(\Td)} \\ & \hspace{2cm}= \alpha\|\Delta v \|^2_H -\alpha(F'_\lambda(v), \Delta v)_H + \beta\|\nabla v\|^2_\bH + \beta(F'_\lambda(v) - \overline{F'_\lambda(v)}, v)_H \\
		& \hspace{2cm}\geq \alpha\| \Delta v\|^2_{H^2(\Td)}- {\alpha}\left( \dfrac 1 \lambda + C_R \right)\|v\|_H\|\Delta v\|_H + \beta\|\nabla v\|^2_\bH - \beta \left( \dfrac 1 \lambda + C_R \right) C_P\|\nabla v\|_\bH\|v\|_H
		\\ & \hspace{2cm} \geq \dfrac{\alpha}{2}\|\Delta v\|_H^2 + \dfrac \beta 2\|\nabla v\|^2_\bH- \left( \dfrac{\alpha+\beta C_P^2}{2\lambda^2} \right)\|v\|^2_H
		\\ & \hspace{2cm} \geq \dfrac{\alpha}{2}\| v\|_{H^2(\Td)}^2 + \dfrac \beta 2\|\nabla v\|^2_\bH- \left( \alpha + \dfrac{\beta C_P^2}{2} \right)\left( \dfrac 1 \lambda + C_R \right)^2\|v\|^2_H \\
		\end{split}
	\]
	hence we can choose $c_2 = \frac \alpha 2$, $c_2' =  ( \alpha + \frac{\beta C_P^2}{2} )( \frac 1 \lambda + C_R )^2 $ and $f = 0$ in (iv). Finally, assuming that $\psi \in H^2(\Td)$ has unitary norm, we have
	\[
	\begin{split}
		|\langle \mathcal{A}_\lambda(\omega,t,v), \psi \rangle_{(H^2(\Td))^*,\: H^2(\Td)}| & = \alpha(\Delta v, \Delta \psi)_H -\alpha (F'_\lambda(v), \Delta\psi)_H + \beta (-\Delta v, \psi)_H + \beta(F'(v)-\overline{F'(v)}, \psi)_H \\
		& \leq \alpha\|\Delta v\|_H + \alpha\|F'_\lambda(v)-F'_\lambda(0)\|_H + \beta C_P \left( \dfrac 1 \lambda + C_R \right)\|\nabla v\|_\bH\\
		& \leq (\alpha + \beta)\|\Delta v\|_H + \alpha \left( \dfrac 1 \lambda + C_R \right)\|v\|_H +\beta C_P \left( \dfrac 1 \lambda + C_R \right)\|\nabla v\|_\bH\\
		& \leq \max\left\{\alpha + \beta,\,\alpha \left( \dfrac 1 \lambda + C_R \right)\!,\,\beta C_P \left( \dfrac 1 \lambda + C_R \right)\right\}\| v\|_{H^2(\Td)}.
	\end{split}
	\]
	Thus, we set $c_3 = \max\{\alpha + \beta,\,\alpha ( \frac 1 \lambda + C_R ),\,\beta C_P ( \frac 1 \lambda + C_R )\}$ and $g(\omega,t) = 0$. The proof is complete.
\end{proof} \noindent
As a direct consequence of Lemma \ref{lem:Alambda} and the fact that the map
\[
w \mapsto \dfrac{2L_\b{G}^2}{\xi^2}\left(1 + \|w\|^2_V\right)
\]
is locally bounded in $H^2(\Td)$, we obtain the following existence result.
\begin{prop} \label{prop:existencereg}
	For any $\lambda \in (0,\Lambda)$ and $\xi \in (0,\Lambda)$ there exists an adapted process
	\[
	\varphi_{\lambda,\xi} \in L^p_\cP(\Omega;C^0([0,T];H)) \cap L^p_\cP(\Omega;L^2(0,T;H^2(\Td)))
	\]
	that is the unique strong solution of problem \eqref{eq:strongformreg}. The exponent $p$ is determined by the initial datum $\varphi_0$ and is the one appearing in Definition \ref{def:solCH}.
\end{prop}
\begin{remark} \label{rem:chemical}
	Observe that the chemical potential $\mu_{\lambda,\xi}$ only belongs to the space $L^p_\cP(\Omega; L^2(0,T;H))$. This follows from the Lipschitz continuity of $F'_\lambda$. This entails that the rigorous weak formulation of the problem is given by
	\begin{multline} \label{eq:weakformreg}
		(\varphi_{\lambda,\xi}(t),\psi)_{H} - \alpha
		\int_0^t\!\int_{\Td}  \mu_{\lambda,\xi}(s) \Delta\psi\,\d s + \beta \int_0^t\!\int_{\Td}  (\mu_{\lambda,\xi}(s) - \overline{\mu_{\lambda,\xi}(s)})\psi\,\d s
		\\ = ( \varphi_0,\psi)_{H} +
		\left(\int_0^t K_{\lambda,\xi}(\varphi_{\lambda,\xi})\,\d  W(s), \psi\right)_{H}
	\end{multline}
	for all $\psi \in H^2(\Td)$, $t \in [0,T]$ and $\P$-almost surely.
\end{remark}
\subsection{Uniform estimates.} \label{ssec:unifestCH} Let $\lambda \in (0,\Lambda)$ be fixed. Without loss of generality, for technical reasons that will be clear in the following, assume that $\xi = \xi(\lambda)$ is fixed in the interval $(0,\lambda)$. Here, we prove some uniform estimates with respect to $\lambda$ and $\xi$ that enable us to construct a solution to the original mixed problem \eqref{eq:ch} through a stochastic compactness argument. Throughout this subsection, the symbols $C$ and $C_i$, with $i \in \mathbb{N}$, denote positive constants which are independent of $\lambda$ and $\xi$, but may depend on other parameters that are explicitly pointed out, if necessary. Moreover, for the sake of clarity, the same symbol may denote different constants (always independent of $\lambda$ and $\xi$) within the same argument.
\paragraph{\textit{First estimate.}} Applying the It\^{o} formula (see \cite[Theorem 4.2.5]{LiuRo}) to the $H$-norm of $\varphi_{\lambda,\xi}$, we get
\begin{multline} \label{eq:unif10}
	\dfrac{1}{2}\|\varphi_{\lambda,\xi}(t)\|_{H}^2 +\alpha\int_0^t \left[ \|\Delta \varphi_{\lambda,\xi}(\tau)\|^2_H + \left( -\Delta\varphi_{\lambda,\xi}(\tau), F'_\lambda(\varphi_{\lambda,\xi}(\tau))\right)_H \right] \: \mathrm{d}\tau \\
	+ \beta \int_0^t \left[ \|\nabla \varphi_{\lambda,\xi}(\tau)\|^2_\bH + (F'_\lambda(\varphi_{\lambda,\xi}(\tau)) - \overline{F'_\lambda(\varphi_{\lambda,\xi}(\tau))}, \varphi_{\lambda,\xi}(\tau))_H \right]\: \d \tau
	\\= \dfrac{1}{2}\|\varphi_{0}\|_{H}^2 + \int_0^t \left(\varphi_{\lambda,\xi}(\tau), K_{\lambda,\xi}(\varphi_{\lambda,\xi}(\tau))\,\mathrm{d}W(\tau)\right)_H + \dfrac{1}{2}\int_0^t \|K_{\lambda,\xi}(\varphi_{\lambda,\xi}(\tau))\|^2_{\cL^2(U, H)} \: \mathrm{d}\tau.
\end{multline}
First, observe that
\begin{equation} \label{eq:unif11}
	\begin{split}
		& \ii{0}{t}{\left( -\Delta\varphi_{\lambda,\xi}(\tau), F'_\lambda(\varphi_{\lambda,\xi}(\tau))\right)_H}{\tau} \\
		& \hspace{2.5cm}= \ii{0}{t}{(\Psi_\lambda''(\varphi_{\lambda,\xi}(\tau)), |\nabla \varphi_{\lambda,\xi}(\tau)|^2)_H}{\tau} + \ii{0}{t}{\left( -\Delta\varphi_{\lambda,\xi}(\tau), R'(\varphi_{\lambda,\xi}(\tau))\right)_H}{\tau}
	\end{split}
\end{equation}
and
\begin{equation} \label{eq:unif12}
	\begin{split}
		\left|\ii{0}{t}{\left( -\Delta\varphi_{\lambda,\xi}(\tau),
			R'(\varphi_{\lambda,\xi}(\tau))\right)_H}{\tau} \right| &
		\leq  C_R\int_0^t\|\nabla \varphi_{\lambda,\xi}(\tau)\|^2_\bH\,\d \tau \\
		& \leq \varepsilon_1\int_0^t
		\|\Delta \varphi_{\lambda,\xi}(\tau)\|^2_H\,\d\tau
		+ C\int_0^t\|\varphi_{\lambda,\xi}(\tau)\|^2_H\,\d \tau,
	\end{split}
\end{equation}
where $\varepsilon_1 > 0$ is arbitrary and $C$ depends only on $C_R$ and $\varepsilon_1$. Also, recalling Lemma \ref{lem:lipschitz}, we have
\begin{equation} \label{eq:unif13}
	\|K_{\lambda,\xi}(\varphi_{\lambda,\xi}(\tau))\|^2_{\cL^2(U, H)} \leq \varepsilon_1\|\Delta \varphi_{\lambda,\xi}(\tau)\|^2_H + C\|\varphi_{\lambda,\xi}(\tau)\|^2_H,
\end{equation}
where $C$ depends only on $\varepsilon_1$.
Finally, observe that by monotonicity of $F'_\lambda$
\[
\begin{split}
	0 & \leq \int_{\Td} \int_{\Td} (\Psi'_\lambda(\varphi_{\lambda,\xi}(x))-\Psi'_\lambda(\varphi_{\lambda,\xi}(y)))(\varphi_{\lambda,\xi}(x)-\varphi_{\lambda,\xi}(y)) \: \d x \d y \\
	& = 2|\Td|\int_{\Td}\varphi_{\lambda,\xi}(x)\Psi'_\lambda(\varphi_{\lambda,\xi}(x)) \: \d x - 2|\Td|^2 \overline{\Psi'_\lambda(\varphi_{\lambda,\xi})}\overline{\varphi_{\lambda,\xi}} \\
	& = 2|\Td|^2\left( \overline{\Psi'_\lambda(\varphi_{\lambda,\xi})\varphi_{\lambda,\xi}} - \overline{\Psi'_\lambda(\varphi_{\lambda,\xi})}\overline{\varphi_{\lambda,\xi}}\right)
\end{split}
\]
implying in turn that
\begin{equation} \label{eq:unif13_bis}
	\begin{split}
		\int_0^t (F'_\lambda(\varphi_{\lambda,\xi}(\tau)) - \overline{F'_\lambda(\varphi_{\lambda,\xi}(\tau))}, \varphi_{\lambda,\xi}(\tau))_H \: \d \tau \geq \int_0^t (R'_\lambda(\varphi_{\lambda,\xi}(\tau)) - \overline{R'_\lambda(\varphi_{\lambda,\xi}(\tau))}, \varphi_{\lambda,\xi}(\tau))_H \: \d \tau,
	\end{split}
\end{equation}
and
\begin{equation} \label{eq:unif13_ter}
	\begin{split}
		\left| \int_0^t (R'_\lambda(\varphi_{\lambda,\xi}(\tau)) - \overline{R'_\lambda(\varphi_{\lambda,\xi}(\tau))}, \varphi_{\lambda,\xi}(\tau))_H \right| & \leq C_RC_P\int_0^t\|\nabla \varphi_{\lambda,\xi}(\tau)\|_\bH\|\varphi_{\lambda,\xi}(\tau)\|_H \: \d \tau \\
		& \leq \dfrac 12 \int_0^t\|\nabla \varphi_{\lambda,\xi}(\tau)\|_\bH^2 \: \d\tau + C\int_0^t\|\varphi_{\lambda,\xi}(\tau)\|_H^2 \: \d \tau.
	\end{split}
\end{equation}
Let $\varepsilon_1 = \frac{\alpha}{4}$. Exploiting \eqref{eq:unif11}-\eqref{eq:unif13_ter} in \eqref{eq:unif10}, taking the supremum over $[0,T]$, then raising the inequality to the power $\frac p2$ and finally taking expectations with respect to $\P$, we arrive at
\begin{multline} \label{eq:unif14}
	\E\supp\|\varphi_{\lambda,\xi}(\tau)\|_{H}^p + \E \left| \int_0^t \|\Delta \varphi_{\lambda,\xi}(\tau)\|^2_H \:\d\tau \right|^\frac p2 + \E \left| \int_0^t (\Psi_\lambda''(\varphi_{\lambda,\xi}(\tau)), |\nabla \varphi_{\lambda,\xi}(\tau)|^2)_H \: \mathrm{d}\tau \right|^\frac p2 \\ \leq C\left[ \E\|\varphi_{0}\|_{H}^p + \E \sups\left| \int_0^s  \left(\varphi_{\lambda,\xi}(\tau), K_{\lambda,\xi}(\varphi_{\lambda,\xi}(\tau))\,\mathrm{d}W(\tau)\right)_H \right|^\frac p2 + \E\left|\int_0^t \|\varphi_{\lambda,\xi}(\tau)\|^2_H \: \mathrm{d}\tau \right|^\frac p2 \right],
\end{multline}
where the constant $C$ may depend on $p$. Finally, we employ the Burkholder-Davis-Gundy inequality to show that
\begin{align}
	\nonumber
	&\E \sups\left| \int_0^s  \left(\varphi_{\lambda,\xi}(\tau), K_{\lambda,\xi}(\varphi_{\lambda,\xi}(\tau))\,\mathrm{d}W(\tau)\right)_H \right|^\frac p2  \\
	\nonumber
	& \qquad\leq C \E \left| \int_0^t  \|\varphi_{\lambda,\xi}(\tau)\|^2_H\| K_{\lambda,\xi}(\varphi_{\lambda,\xi}(\tau))\|^2_{\cL^2(U,H)}\:\d\tau \right|^\frac p4 \\
	\nonumber
	& \qquad \leq C\E \left| \supp \|\varphi_{\lambda,\xi}(\tau)\|_H^2\int_0^t \| K_{\lambda,\xi}(\varphi_{\lambda,\xi}(\tau))\|^2_{\cL^2(U,H)}\:\d\tau \right|^\frac p4 \\
	\nonumber
	& \qquad \leq C\E\left[ \supp \|\varphi_{\lambda,\xi}(\tau)\|_H^\frac p2 \left|\int_0^t  \| \varphi_{\lambda,\xi}(\tau)\|_V^2\:\d\tau \right|^\frac p4 \right] \\
	\nonumber
	& \qquad
	\leq C\left[ \E\supp  \|\varphi_{\lambda,\xi}(\tau)\|^p_H \right]^\frac 12 \left[\E \left|\int_0^t  \| \varphi_{\lambda,\xi}(\tau)\|_V^2\:\d\tau \right|^\frac p2 \right]^\frac 12 \\
	& \qquad \leq \dfrac{1}{4} \E\supp  \|\varphi_{\lambda,\xi}(\tau)\|_H^p + \dfrac{1}{4}\E \left| \int_0^t \|\Delta \varphi_{\lambda,\xi}(\tau)\|^2_H \:\d\tau\right|^\frac p2 + C\E\left|\int_0^t \| \varphi_{\lambda,\xi}(\tau)\|^2_H \:\d\tau\right|^\frac p2.
	\label{eq:unif15}
\end{align}
On account of the convexity of $\Psi_\lambda$, the estimates \eqref{eq:unif14} and \eqref{eq:unif15} yield
\begin{multline*} \label{eq:unif16}
	\E\supp\|\varphi_{\lambda,\xi}(\tau)\|_{H}^p + \E \left| \int_0^t \|\Delta \varphi_{\lambda,\xi}(\tau)\|^2_H \:\d\tau \right|^\frac p2 + \E \left| \int_0^t (\Psi_\lambda''(\varphi_{\lambda,\xi}(\tau)), |\nabla \varphi_{\lambda,\xi}(\tau)|^2)_H \: \mathrm{d}\tau \right|^\frac p2 \\ \leq C\left[ \E\|\varphi_{0}\|_{H}^p+ \E\left|\int_0^t \|\varphi_{\lambda,\xi}(\tau)\|^2_H \: \mathrm{d}\tau \right|^\frac p2 \right].
\end{multline*}
The Gronwall lemma implies that there exists a constant $C_1$ independent of $\lambda$ and $\xi$ such that
\begin{equation} \label{eq:unif1}
	\|\varphi_{\lambda,\xi}\|_{L^p_\cP(\Omega;C^0([0,T];H))} + \|\sqrt{\Psi''_\lambda(\varphi_{\lambda,\xi})}\nabla \varphi_{\lambda,\xi}\|_{L^p_\cP(\Omega; L^2(0,T;H))} + \|\varphi_{\lambda,\xi}\|_{L^p_\cP(\Omega; L^2(0,T;H^2(\Td)))} \leq C_1.
\end{equation}
\paragraph{\textit{Second estimate.}} It is apparent that the approximating problem \eqref{eq:strongformreg} does not have the mass conservation property. This is easily seen taking $\psi \equiv 1$ in \eqref{eq:weakformreg}, yielding
\begin{equation} \label{eq:unif20}
	\overline{\varphi_{\lambda,\xi}(t)} = \overline{\varphi_0} + \int_0^t \overline{K_{\lambda,\xi}(\varphi_{\lambda,\xi}(\tau))} \:\d W(\tau),
\end{equation}
so that mass is only conserved under expectations. In order to make sense of \eqref{eq:unif20}, we clarify that, given $L \in \cL^2(U, H)$, we set
\[
\overline{L} \in \cL^2(U,\erre), \qquad \overline{L}[u_k] = \overline{L[u_k]} \qquad \forall \: k \in \enne_+.
\] Observe that, since
\[
\overline{\b{g}'_k(J_\lambda(\varphi_{\lambda,\xi}(\tau))) \cdot \nabla J_\lambda(\varphi_{\lambda,\xi}(\tau))} = \dfrac{1}{|\Omega|}(\div \b{g}_k(J_\lambda(\varphi_{\lambda,\xi}(\tau))), 1)_H =  0,
\]
setting
\[
I_\lambda = \div [\b{G} \circ J_\lambda],
\]
we easily infer from \eqref{eq:unif20} and the Burkh\"{o}lder--Davis--Gundy inequality that
\begin{align}
	\nonumber
	\E \supp |\overline{\varphi_{\lambda,\xi}(\tau)} - \overline{\varphi_0}|^p  &\leq C\E \left[  \int_0^t \|\overline{K_{\lambda,\xi}(\varphi_{\lambda,\xi})} - \overline{I_\lambda(\varphi_{\lambda,\xi})} \|^2_{\cL^2(U, \erre)} \: \d \tau \right]^\frac p2 \\ & = C\E \left[  \int_0^t \sum_{k  = 1}^{+\infty}  |\overline{\b{g}'_k(J_\lambda(\varphi_{\lambda,\xi}(\tau))) \cdot \left( \nabla \mathcal{R}_\xi\varphi_{\lambda,\xi}(\tau) - \nabla J_\lambda(\varphi_{\lambda,\xi}(\tau)) \right)}|^2 \: \d \tau \right]^\frac p2.
	\label{eq:unif21}
\end{align}
Next, we investigate two useful auxiliary estimates. First,
\[
\begin{split}
	\sum_{k  = 1}^{+\infty} |\overline{\b{g}'_k(J_\lambda(\varphi_{\lambda,\xi}(\tau)))
		\cdot \left( \nabla \mathcal{R}_\xi\varphi_{\lambda,\xi}(\tau)
		- \nabla \varphi_{\lambda,\xi}(\tau) \right)}|^2
	& = \dfrac{1}{|\Td|^2}\sum_{k  = 1}^{+\infty} \left(\b{g}'_k(J_\lambda(\varphi_{\lambda,\xi}(\tau))),
	\xi \nabla \Delta \mathcal{R}_\xi \varphi_{\lambda,\xi}(\tau) \right)_H^2 \\
	& \leq \dfrac{L_\b{G}^2 \xi^2}{{|\Td|^2}}\|\nabla \Delta\mathcal{R}_\xi \varphi_{\lambda,\xi}\|^2_{\bH} \\
	& \leq \dfrac{L_\b{G}^2\lambda}{2|\Td|^2}\|\Delta \varphi_{\lambda,\xi}\|_{H}^2,
\end{split}
\]
where the last passage is justified by taking the $L^2$-product of \eqref{eq:elliptic} by $\Delta^2\mathcal{R}_\xi\varphi_{\lambda,\xi} \in H$
and exploiting the fact that $\xi < \lambda$. Indeed, using the Cauchy--Schwarz inequality we obtain
\[
\xi\|\nabla \Delta\mathcal{R}_\xi \varphi_{\lambda,\xi}\|^2_{\bH} + \|\Delta \mathcal{R}_\xi \varphi_{\lambda,\xi}\|^2_H \leq \|\Delta \varphi_{\lambda,\xi}\|_H \| \Delta \mathcal{R}_\xi\varphi_{\lambda,\xi}\|_H,
\]
implying by the Young inequality
\begin{equation} \label{eq:ellipticestimate}
	\xi\|\nabla \Delta\mathcal{R}_\xi \varphi_{\lambda,\xi}\|^2_{\bH} \leq \dfrac 12 \|\Delta \varphi_{\lambda,\xi}\|_H^2.
\end{equation}
Secondly, we have
\[
\begin{split}
	\sum_{k  = 1}^{+\infty} |\overline{\b{g}'_k(J_\lambda(\varphi_{\lambda,\xi}(\tau)))
		\cdot \left( \nabla \varphi_{\lambda,\xi}(\tau)
		- \nabla J_\lambda(\varphi_{\lambda,\xi}(\tau)) \right)}|^2
	& = \dfrac{1}{|\Td|^2}\sum_{k  = 1}^{+\infty} \left(\b{g}'_k(J_\lambda(\varphi_{\lambda,\xi}(\tau))), \lambda \nabla \Psi_\lambda'(\varphi_{\lambda,\xi}) \right)_H^2 \\
	& = \dfrac{\lambda^2}{|\Td|^2}\sum_{k  = 1}^{+\infty} \left(\b{g}'_k(J_\lambda(\varphi_{\lambda,\xi}(\tau))), \Psi_\lambda''(\varphi_{\lambda,\xi})\nabla \varphi_{\lambda,\xi} \right)_H^2 \\
	& \leq \dfrac{L_\b{G}^2 \lambda^2}{|\Td|}
	\sup_{s \in \erre} \Psi_\lambda''(s)\|\sqrt{ \Psi''_\lambda(\varphi_{\lambda,\xi})}\nabla \varphi_{\lambda,\xi}\|^2_{H} \\
	& = \dfrac{L_\b{G}^2\lambda}{|\Td|}
	\|\sqrt{ \Psi''_\lambda(\varphi_{\lambda,\xi})}\nabla \varphi_{\lambda,\xi}\|^2_{H},
\end{split}
\]
where the last line follows from the Lipschitz continuity of $\Psi'_\lambda$ and \eqref{eq:unif1}. Therefore,
summing and subtracting $\nabla \varphi_{\lambda,\xi}(\tau)$ in the right hand side of \eqref{eq:unif21}, we end up with the estimate
\begin{equation*}
	\E \supp |\overline{\varphi_{\lambda,\xi}(\tau)} - \overline{\varphi_0}|^p \leq C \lambda^\frac p2 \left[ 1 + \E \left| \int_0^t \|\Delta \varphi_{\lambda,\xi}(\tau)\|^2_H  \: \d\tau \right|^\frac p2 + \E \left| \int_0^t \|\sqrt{ \Psi''_\lambda(\varphi_{\lambda,\xi})}\nabla \varphi_{\lambda,\xi}\|^2_{H}  \: \d\tau \right|^\frac p2 \right],
\end{equation*}
implying, by \eqref{eq:unif1}, that
\begin{equation} \label{eq:unif2}
	\E \supp |\overline{\varphi_{\lambda,\xi}(\tau)} - \overline{\varphi_0}|^p \leq C_2 \lambda^\frac p2,
\end{equation}
for some constant $C_2 > 0$ independent of $\lambda$ and $\xi$, but possibly depending on $p$.
\begin{remark}
	Observe that \eqref{eq:unif2} implies asymptotic conservation of mass, as expected from the original problem. Notice also that estimate \eqref{eq:unif2} can be refined, owing to Assumption \ref{hyp:G}, in such a way that the left hand side is proportional to $\lambda^p$ (however, we also need $\xi < \lambda^2$). Nonetheless, the minimal needed decay rate for the mass gap is the one given by \eqref{eq:unif2}.
\end{remark}
\paragraph{\textit{Third estimate.}} The It\^{o} formula for twice-differentiable functionals, in its classical version given, for instance, in \cite[Theorem 4.32]{dapratozab}, can be applied
to the functional
\[
t \mapsto \dfrac{\alpha}{2}\|\nabla \mathcal{N}_{\alpha\beta}(\varphi_{\lambda,\xi}(t)- \overline{\varphi_{\lambda,\xi}(t)})\|_\bH^2 + \dfrac{\beta}{2}\| \mathcal{N}_{\alpha\beta}(\varphi_{\lambda,\xi}(t)- \overline{\varphi_{\lambda,\xi}(t)})\|_H^2,
\]
resulting in
\begin{multline}
	\dfrac{\alpha}{2}\|\nabla \mathcal{N}_{\alpha\beta}(\varphi_{\lambda,\xi}(t) - \overline{\varphi_{\lambda,\xi}(t)})\|_{H}^2 + \dfrac{\beta}{2}\| \mathcal{N}_{\alpha\beta}(\varphi_{\lambda,\xi}(t)- \overline{\varphi_{\lambda,\xi}(t)})\|_H^2 \\
	+ \int_0^t \left[ \|\nabla \varphi_{\lambda,\xi}(\tau)\|^2_H + \left( \varphi_{\lambda,\xi}(\tau) - \overline{\varphi_{\lambda,\xi}(\tau)}, F'_\lambda(\varphi_{\lambda,\xi}(\tau))\right)_H \right] \: \mathrm{d}\tau \\
	= \dfrac{\alpha}{2}\|\nabla \mathcal{N}_{\alpha\beta}(\varphi_0 - \overline{\varphi_0})\|_{H}^2 + \dfrac{\beta}{2}\| \mathcal{N}_{\alpha\beta}(\varphi_0 - \overline{\varphi_0})\|_{H}^2 \\ + \int_0^t \left(\mathcal{N}_{\alpha\beta}(\varphi_{\lambda,\xi}(\tau) - \overline{\varphi_{\lambda,\xi}(\tau)}), [K_{\lambda,\xi}(\varphi_{\lambda,\xi}(\tau)) - \overline{K_{\lambda,\xi}(\varphi_{\lambda,\xi}(\tau))}]\,\mathrm{d}W(\tau)\right)_H  \\
	+ \dfrac{\alpha}{2}\int_0^t \|\nabla \mathcal{N}_{\alpha\beta}( K_{\lambda,\xi}(\varphi_{\lambda,\xi}(\tau)) - \overline{K_{\lambda,\xi}(\varphi_{\lambda,\xi}(\tau))})\|^2_{\cL^2(U, H)} \: \mathrm{d}\tau\\
	+ \dfrac{\beta}{2}\int_0^t \| \mathcal{N}_{\alpha\beta}( K_{\lambda,\xi}(\varphi_{\lambda,\xi}(\tau)) - \overline{K_{\lambda,\xi}(\varphi_{\lambda,\xi}(\tau))})\|^2_{\cL^2(U, H)} \: \mathrm{d}\tau.
	\label{eq:unif30}
\end{multline}
Let us address the terms appearing in \eqref{eq:unif30} one by one. First, observe that by the Poincaré and Young inequalities
\begin{equation} \label{eq:unif31}
	\begin{split}
		|(R'(\varphi_{\lambda,\xi}(\tau)), \varphi_{\lambda,\xi}(\tau) - \overline{\varphi_{\lambda,\xi}(\tau)})_H| & \leq C_R\|\varphi_{\lambda,\xi}(\tau)\|_H \|\varphi_{\lambda,\xi}(\tau) - \overline{\varphi_{\lambda,\xi}(\tau)}\|_H  \\
		& \leq \varepsilon_2 \|\nabla \varphi_{\lambda,\xi}(\tau)\|^2_\bH + C\|\varphi_{\lambda,\xi}(\tau)\|_\sharp^2,
	\end{split}
\end{equation}
where $\varepsilon_2 > 0$ is arbitrary and $C$ depends on $\varepsilon_2$.
Next, we have by the Jensen inequality that
\begin{align}
	\nonumber
	& \|\nabla \mathcal{N}_{\alpha\beta}( K_{\lambda,\xi}(\varphi_{\lambda,\xi}(\tau)) - \overline{K_{\lambda,\xi}(\varphi_{\lambda,\xi}(\tau)})\|^2_{\cL^2(U, H)}
	+
	\| \mathcal{N}_{\alpha\beta}( K_{\lambda,\xi}(\varphi_{\lambda,\xi}(\tau)) - \overline{K_{\lambda,\xi}(\varphi_{\lambda,\xi}(\tau)})\|^2_{\cL^2(U, H)} \\
	\nonumber
	& \qquad= \sum_{k  = 1}^{+\infty} \|\nabla \mathcal{N}_{\alpha\beta}(\b{g}'_k(J_\lambda(\varphi_{\lambda,\xi}(\tau))) \cdot \nabla \mathcal{R}_\xi\varphi_{\lambda,\xi}(\tau) - \overline{\b{g}'_k(J_\lambda(\varphi_{\lambda,\xi}(\tau))) \cdot \nabla \mathcal{R}_\xi\varphi_{\lambda,\xi}(\tau)})\|^2_H \\
	\nonumber
	&\qquad\qquad
	+\sum_{k  = 1}^{+\infty} \| \mathcal{N}_{\alpha\beta}(\b{g}'_k(J_\lambda(\varphi_{\lambda,\xi}(\tau))) \cdot \nabla \mathcal{R}_\xi\varphi_{\lambda,\xi}(\tau) - \overline{\b{g}'_k(J_\lambda(\varphi_{\lambda,\xi}(\tau))) \cdot \nabla \mathcal{R}_\xi\varphi_{\lambda,\xi}(\tau)})\|^2_H\\
	\nonumber
	&\qquad\leq C \sum_{k  = 1}^{+\infty} \|\b{g}'_k(J_\lambda(\varphi_{\lambda,\xi}(\tau))) \cdot \nabla \mathcal{R}_\xi\varphi_{\lambda,\xi}(\tau) - \overline{\b{g}'_k(J_\lambda(\varphi_{\lambda,\xi}(\tau))) \cdot \nabla \mathcal{R}_\xi\varphi_{\lambda,\xi}(\tau)}\|^2_H \\
	\nonumber
	& \qquad
	\leq 2C\sum_{k  = 1}^{+\infty}
	\left[\|\b{g}'_k(J_\lambda(\varphi_{\lambda,\xi}(\tau))) \cdot \nabla \mathcal{R}_\xi\varphi_{\lambda,\xi}(\tau)\|^2_H
	+ |\Td||\overline{\b{g}'_k(J_\lambda(\varphi_{\lambda,\xi}(\tau)))
		\cdot \nabla \mathcal{R}_\xi\varphi_{\lambda,\xi}(\tau)}|^2\right]
	\\
	\nonumber
	& \qquad
	\leq 4C\sum_{k  = 1}^{+\infty}
	\|\b{g}'_k(J_\lambda(\varphi_{\lambda,\xi}(\tau))) \cdot \nabla \mathcal{R}_\xi\varphi_{\lambda,\xi}(\tau)\|^2_H
	\\
	\nonumber
	& \qquad \leq 4C \sum_{k  = 1}^{+\infty}
	\|\b{g}'_k\|^2_{\b{L}^\infty(-1,1)}
	\|\nabla\mathcal{R}_\xi\varphi_{\lambda,\xi}(\tau)\|^2_H  \\
	& \qquad
	\leq 4CL_{\b G}^2\|\nabla\mathcal{R}_\xi\varphi_{\lambda,\xi}(\tau)\|^2_H
	\leq 4CL_{\b G}\| \varphi_{\lambda,\xi}(\tau)\|^2_{V}.
	\label{eq:unif32}
\end{align}
Observe then that
\begin{align}
	\nonumber
	& \left( \varphi_{\lambda,\xi}(\tau) - \overline{\varphi_{\lambda,\xi}(\tau)}, \Psi'_\lambda(\varphi_{\lambda,\xi}(\tau))\right)_H \\
	\nonumber
	&\qquad= \left( \varphi_{\lambda,\xi}(\tau) - J_\lambda(\varphi_{\lambda,\xi}(\tau)), \Psi'_\lambda(\varphi_{\lambda,\xi}(\tau))\right)_H + \left(J_\lambda(\varphi_{\lambda,\xi}(\tau)),
	\Psi'_\lambda(\varphi_{\lambda,\xi}(\tau))\right)_H \\
	& \qquad\qquad+ \left( \overline{\varphi_{0}} - \overline{\varphi_{\lambda, \xi}(\tau)}, \Psi'_\lambda(\varphi_{\lambda,\xi}(\tau))\right)_H - \left( \overline{\varphi_{0}}, \Psi'_\lambda(\varphi_{\lambda,\xi}(\tau))\right)_H.
	\label{eq:unif33}
\end{align}
Among the terms to the right hand side of \eqref{eq:unif33}, we have
\[
\left( \varphi_{\lambda,\xi}(\tau) - J_\lambda(\varphi_{\lambda,\xi}(\tau)), \Psi'_\lambda(\varphi_{\lambda,\xi}(\tau))\right)_H  = \lambda \|\Psi'_\lambda(\varphi_{\lambda, \xi}(\tau))\|^2_H,
\]
and besides, owing to the Fenchel--Young inequality
\[
\left(J_\lambda(\varphi_{\lambda,\xi}(\tau)),\Psi'_\lambda(\varphi_{\lambda,\xi}(\tau))\right)_H = \Psi(J_\lambda(\varphi_{\lambda, \xi}(\tau))) + \Psi^*(\Psi'_\lambda(\varphi_{\lambda,\xi}(\tau)))
\]
where $\Psi^*$ is the convex conjugate of $\Psi$. Observe that equality holds since $\Psi'_\lambda(\varphi_{\lambda, \xi}(\tau)) = \Psi'(J_\lambda(\varphi_{\lambda, \xi}(\tau)))$ is the only element of the subdifferential $\partial\Psi(J_\lambda(\varphi_{\lambda, \xi}(\tau)))$. Moreover, by the same result
\[
\left( \overline{\varphi_{0}}, \Psi'_\lambda(\varphi_{\lambda,\xi}(\tau))\right)_H \leq \Psi(\overline{\varphi_0}) + \Psi^*(\Psi'_\lambda(\varphi_{\lambda,\xi}(\tau))).
\]
Considering back \eqref{eq:unif30}, taking \eqref{eq:unif31}-\eqref{eq:unif33} as well as the previous relations into account, adding $\frac 12|\overline{\varphi_{\lambda,\xi}(t)}|^2$ to both sides, raising the resulting inequality to the power $\frac p2$, then taking suprema in time and $\P$-expectations, we arrive at
\begin{align}
	\nonumber
	&\E \supp \|\varphi_{\lambda, \xi}(\tau)\|_{\sharp}^p
	+ \E \left|  \int_0^T \|\nabla \varphi_{\lambda, \xi}(\tau)\|^2_H \:\d \tau \right|^\frac p2 \\
	\nonumber
	&\qquad+ \lambda^\frac p2\E \left|  \int_0^T
	\|\Psi'_\lambda(\varphi_{\lambda, \xi}(\tau))\|^2_H \:\d \tau \right|^\frac p2
	+\E \left|  \int_0^T
	\|\Psi(J_\lambda(\varphi_{\lambda, \xi}(\tau)))\|_{L^1(\T^d)} \:\d \tau \right|^\frac p2\\
	\nonumber
	&\leq C \left[ 1 + \E \|\varphi_0\|_{\sharp}^p
	+\E|\Psi(\overline{\varphi_0})|^{\frac p2}
	+ \E\supp |\overline{\varphi_{\lambda, \xi}(\tau)}
	- \overline{\varphi_{0}}|^\frac p2 \left| \int_0^T
	\|\Psi'_\lambda(\varphi_{\lambda,\xi}(\tau))\|_H  \: \d\tau \right|^\frac p2\right.\\
	\nonumber
	&\qquad\left.
	+ \E \sups \left| \int_0^s \left(\mathcal{N}_{\alpha\beta}(\varphi_{\lambda, \xi}(\tau)
	- \overline{\varphi_{\lambda, \xi}(\tau)}),
	[K_{\lambda,\xi}(\varphi_{\lambda,\xi}(\tau))
	- \overline{K_{\lambda,\xi}(\varphi_{\lambda,\xi}(\tau))}]
	\,\mathrm{d}W(\tau)\right)_H \right|^\frac p2 \right. \\
	&\qquad\left. + \E \left| \int_0^T \|\varphi_{\lambda, \xi}(\tau)\|^2_V \: \d \tau
	\right|^\frac p2 + \E \left| \int_0^T
	\|\varphi_{\lambda, \xi}(\tau)\|^2_\sharp\: \mathrm{d}\tau\right|^\frac p2
	+ \E \supp \|\varphi_{\lambda, \xi}(\tau)\|^p_H \right]
	\label{eq:unif34}
\end{align}
The remaining terms can be controlled as follows. First, by \eqref{eq:unif2}, we have	
\begin{align}
	\nonumber
	& \E\left[
	\supp |\overline{\varphi_{\lambda, \xi}(\tau)} - \overline{\varphi_{0}}|^\frac p2 \left| \int_0^T \|\Psi'_\lambda(\varphi_{\lambda,\xi}(\tau))\|_H  \: \d\tau \right|^\frac p2\right] \\
	\nonumber	
	& \qquad\leq \left[ \E\supp |\overline{\varphi_{\lambda, \xi}(\tau)} - \overline{\varphi_{0}}|^p \right]^\frac 12 \left[ \E \left| \int_0^T \|\Psi'_\lambda(\varphi_{\lambda,\xi}(\tau))\|_H  \: \d\tau \right|^p \right]^\frac 12 \\
	\label{eq:unif35}
	& \qquad\leq C \lambda^\frac p4\left[ \E \left| \int_0^T \|\Psi'_\lambda(\varphi_{\lambda,\xi}(\tau))\|_H^2  \: \d\tau \right|^{\frac p2} \right]^\frac 12
	\leq \dfrac{C^2}{2} +  \dfrac{\lambda^\frac p2}{2}\E \left| \int_0^T \|\Psi'_\lambda(\varphi_{\lambda,\xi}(\tau))\|_H^2  \: \d\tau \right|^\frac{p}{2},	
\end{align}
where we also exploited the embedding $L^2(0,T;H) \embed L^1(0,T;H)$. In particular, the constant $C$ appearing in \eqref{eq:unif35} also depends on $T$. The Burkholder--Davis--Gundy inequality and \eqref{eq:unif32} entail that
\begin{align}
	\nonumber
	& \E \sups \left| \int_0^s \left(\mathcal{N}_{\alpha\beta}(\varphi_{\lambda, \xi}(\tau) - \overline{\varphi_{\lambda, \xi}(\tau)}), [K_{\lambda,\xi}(\varphi_{\lambda,\xi}(\tau)) - \overline{K_{\lambda,\xi}(\varphi_{\lambda,\xi}(\tau))}]\,\mathrm{d}W(\tau)\right)_H \right|^\frac p2 \\
	\nonumber
	& \hspace{1cm} \leq C\E \left| \int_0^t \|\mathcal{N}_{\alpha\beta}(\varphi_{\lambda, \xi}(\tau) - \overline{\varphi_{\lambda, \xi}(\tau)})\|^2_H\| K_{\lambda,\xi}(\varphi_{\lambda,\xi}(\tau)) - \overline{K_{\lambda,\xi}(\varphi_{\lambda,\xi}(\tau))}\|^2_{\cL^2(U,H)} \:\d\tau \right|^\frac p4 \\
	\nonumber
	& \hspace{1cm} \leq C\E\left[ \supp \|\mathcal{N}_{\alpha\beta}(\varphi_{\lambda, \xi}(\tau) - \overline{\varphi_{\lambda, \xi}(\tau)})\|^\frac{p}{2}_H \left| \int_0^t \| K_{\lambda,\xi}(\varphi_{\lambda,\xi}(\tau)) - \overline{K_{\lambda,\xi}(\varphi_{\lambda,\xi}(\tau))}\|^2_{\cL^2(U,H)} \:\d\tau \right|^\frac p4\right] \\
	\nonumber
	& \hspace{1cm} \leq C\E \left[\supp \|\mathcal{N}_{\alpha\beta}(\varphi_{\lambda, \xi}(\tau) - \overline{\varphi_{\lambda, \xi}(\tau)})\|^\frac{p}{2}_H \left| \int_0^t \| \varphi_{\lambda,\xi}(\tau)\|^2_{V} \:\d\tau \right|^\frac p4\right] \\
	\nonumber
	& \hspace{1cm} \leq C\left[ \E  \supp \|\mathcal{N}_{\alpha\beta}(\varphi_{\lambda, \xi}(\tau) - \overline{\varphi_{\lambda, \xi}(\tau)})\|^p_H \right]^\frac12 \left[ \E\left| \int_0^t \| \varphi_{\lambda,\xi}(\tau)\|^2_{V} \:\d\tau \right|^\frac p2 \right]^\frac 12 \\
	& \hspace{1cm} \leq \varepsilon_2\E  \supp \|\varphi_{\lambda, \xi}(\tau)\|^p_{\sharp} + C\E\left| \int_0^t \| \varphi_{\lambda,\xi}(\tau)\|^2_{V} \:\d\tau \right|^\frac p2.
	\label{eq:unif36}
\end{align}
Set now $\varepsilon_2 = \frac 12$. Collecting \eqref{eq:unif35} and \eqref{eq:unif36} into \eqref{eq:unif34}, we have
\begin{multline*}
	\E \supp \|\varphi_{\lambda, \xi}(\tau)\|_{\sharp}^p + \E \left|  \int_0^T \|\nabla \varphi_{\lambda, \xi}(\tau)\|^2_H \:\d \tau \right|^\frac p2 + \lambda^\frac p2\E \left|  \int_0^T  \|\Psi'_\lambda(\varphi_{\lambda, \xi}(\tau))\|^2_H \:\d \tau \right|^\frac p2 \\ \leq
	C \left[ 1  + \E \left| \int_0^T \|\varphi_{\lambda, \xi}(\tau)\|^2_V \: \d \tau \right|^\frac p2 + \E \left| \int_0^T \|\varphi_{\lambda, \xi}(\tau)\|^2_\sharp\: \mathrm{d}\tau\right|^\frac p2 + \E \supp \|\varphi_{\lambda, \xi}(\tau)\|^p_H \right],
\end{multline*}
which exploiting \eqref{eq:unif1} becomes
\begin{multline*}
	\E \supp \|\varphi_{\lambda, \xi}(\tau)\|_{\sharp}^p + \E \left|  \int_0^T \|\nabla \varphi_{\lambda, \xi}(\tau)\|^2_H \:\d \tau \right|^\frac p2 + \lambda^\frac p2\E \left|  \int_0^T  \|\Psi'_\lambda(\varphi_{\lambda, \xi}(\tau))\|^2_H \:\d \tau \right|^\frac p2 \\ \leq
	C \left[ 1 + \E \left| \int_0^T \|\varphi_{\lambda, \xi}(\tau)\|^2_\sharp\: \mathrm{d}\tau\right|^\frac p2 \right],
\end{multline*}
and thus the Gronwall lemma implies the existence of a constant $C_3 > 0$ such that
\begin{equation} \label{eq:unif3}
	\sqrt \lambda \|\Psi'_\lambda(\varphi_{\lambda,\xi})\|_{L^p_\cP(\Omega; L^2(0,T;H))} \leq C_3.
\end{equation}
\paragraph{\textit{Fourth estimate.}} Now, we give a formal estimate of the $V$-norm of the chemical potential $\mu_{\lambda, \xi}$. This is usually a consequence of a standard argument (see, for instance, \cite{MZ04}). Here we adapt the argument to our problem. In order to ease notation, we omit the time variable. Multiplying the equation for $\mu_{\lambda, \xi}$ in \eqref{eq:strongformreg} by $J_\lambda(\varphi_{\lambda, \xi}) - \overline{J_\lambda(\varphi_{\lambda, \xi})}$ and applying the Poincaré inequality, we obtain the estimate
\begin{equation} \label{eq:unif40}
	(F_\lambda'(\varphi_{\lambda, \xi}), J_\lambda(\varphi_{\lambda, \xi}) - \overline{J_\lambda(\varphi_{\lambda, \xi})})_H + \|\sqrt{J_\lambda'(\varphi_{\lambda, \xi})}\nabla\varphi_{\lambda, \xi}\|^2_\bH \leq (\mu_{\lambda, \xi}, J_\lambda(\varphi_{\lambda, \xi}) - \overline{J_\lambda(\varphi_{\lambda, \xi})})_H.
\end{equation}
First, observe that using the Poincaré inequality yields
\begin{equation} \label{eq:unif41}
	\begin{split}
		(\mu_{\lambda, \xi}, J_\lambda(\varphi_{\lambda, \xi}) - \overline{J_\lambda(\varphi_{\lambda, \xi})})_H & = (\mu_{\lambda, \xi}- \overline{\mu_{\lambda, \xi}}, J_\lambda(\varphi_{\lambda, \xi}) - \overline{J_\lambda(\varphi_{\lambda, \xi})})_H \\
		& \leq C\|\nabla \mu_{\lambda, \xi}\|_\bH\|J_\lambda(\varphi_{\lambda, \xi})\|_H \\
		& \leq C\|\nabla \mu_{\lambda, \xi}\|_\bH
	\end{split}
\end{equation}
as $J_\lambda(\varphi_{\lambda, \xi}) \in \mathbf{B}^{0,\infty}_1$. Moreover, we have
\begin{equation} \label{eq:unif42}
	\begin{split}
		|(R'(\varphi_{\lambda, \xi}), J_\lambda(\varphi_{\lambda, \xi}) - \overline{J_\lambda(\varphi_{\lambda, \xi})})_H| & \leq C_R\|\varphi_{\lambda, \xi}\|_H\|J_\lambda(\varphi_{\lambda, \xi})\|_H \\
		& \leq C_R\|\varphi_{\lambda, \xi}\|_H
	\end{split}
\end{equation}
and
\begin{multline} \label{eq:unif43}
	(\Psi_\lambda'(\varphi_{\lambda,	 \xi}), J_\lambda(\varphi_{\lambda, \xi}) - \overline{J_\lambda(\varphi_{\lambda, \xi})})_H  = (\Psi_\lambda'(\varphi_{\lambda,	 \xi}), J_\lambda(\varphi_{\lambda, \xi}) - \overline{\varphi_0})_H + (\Psi_\lambda'(\varphi_{\lambda,	 \xi}), \overline{\varphi_0} - \overline{\varphi_{\lambda, \xi}})_H \\
	+ (\Psi_\lambda'(\varphi_{\lambda,	 \xi}), \overline{\varphi_{\lambda, \xi}} - \overline{J_\lambda(\varphi_{\lambda, \xi})})_H.
\end{multline}
Concerning the right hand side of \eqref{eq:unif43}, we have, recalling \cite{MZ04},
\begin{equation} \label{eq:unif44}
	(\Psi_\lambda'(\varphi_{\lambda,	 \xi}), J_\lambda(\varphi_{\lambda, \xi}) - \overline{\varphi_0})_H \geq C\|\Psi'_\lambda(\varphi_{\lambda,	 \xi})\|_{L^1(\Td)} - C
\end{equation}
and
\begin{equation} \label{eq:unif45}
	(\Psi_\lambda'(\varphi_{\lambda,	 \xi}), \overline{\varphi_{\lambda, \xi}} - \overline{J_\lambda(\varphi_{\lambda, \xi})})_H = \lambda|\overline{\Psi_\lambda'(\varphi_{\lambda,\xi})}|^2.
\end{equation}
Thus, exploiting \eqref{eq:unif41}-\eqref{eq:unif45} in \eqref{eq:unif40} we arrive at
\begin{multline} \label{eq:unif46}
	\|\Psi'_\lambda(\varphi_{\lambda,	 \xi})\|_{L^1(\Td)} + \lambda|\overline{\Psi_\lambda'(\varphi_{\lambda,\xi})}|^2 + \|\sqrt{J_\lambda'(\varphi_{\lambda, \xi})}\nabla\varphi_{\lambda, \xi}\|^2_\bH  \\ \leq C\left( 1 + \|\varphi_{\lambda, \xi}\|_H + \|\nabla \mu_{\lambda, \xi}\|_\bH + |(\Psi_\lambda'(\varphi_{\lambda,	 \xi}), \overline{\varphi_0} - \overline{\varphi_{\lambda, \xi}})_H\right)|.
\end{multline}
Therefore, observing that $\overline{\mu_{\lambda, \xi}} = \overline{F'_\lambda(\varphi_{\lambda, \xi})}$, we easily obtain from \eqref{eq:unif46} and Assumption \ref{hyp:potential} that there exists $C_4 > 0$ such that
\begin{align}
	\nonumber
	|\overline{\mu_\lambda}| & \leq |\overline{\Psi'_\lambda(\varphi_{\lambda, \xi})}| + |\overline{R'(\varphi_{\lambda, \xi})}| \\
	\nonumber
	& \leq C\left( 1 +  \|\Psi'_\lambda(\varphi_{\lambda,	 \xi})\|_{L^1(\Td)} + \|\varphi_{\lambda, \xi}\|_H \right) \\
	& \leq  C_4\left( 1 + \|\varphi_{\lambda, \xi}\|_H + \|\nabla \mu_{\lambda, \xi}\|_\bH + |(\Psi_\lambda'(\varphi_{\lambda,	 \xi}), \overline{\varphi_0} - \overline{\varphi_{\lambda, \xi}})_H|\right).
	\label{eq:unif4}
\end{align}
\paragraph{\textit{Fifth estimate.}} The previous estimates enable us to prove an energy inequality. In light of Remark \ref{rem:chemical}, the computations that follow are only formal. However, they can easily be made rigorous by means of a suitable approximation scheme. Considering the twice Frechét-differentiable regularized energy functional (see \cite{scarpa21})
\[
\mathcal{E}_\lambda : V \to \mathbb{R}, \qquad
\mathcal{E}_\lambda(v) := \frac12\|\nabla v \|^2_\b{H} + \ii{\Td}{}{F_\lambda(v)}{x}
\]
and applying the It\^{o} formula once again in its classical version (see \cite[Theorem 4.32]{dapratozab}), recalling that $D\mathcal{E}_\lambda(\varphi_{\lambda,	 \xi}) = \mu_{\lambda,\xi}$, yields
\begin{multline} \label{eq:unif50}
	\mathcal{E}_\lambda(\varphi_{\lambda,	 \xi}(t)) +
	\alpha\int_0^t \|\nabla \mu_{\lambda,	 \xi}(\tau)\|^2_\bH \:\d \tau +
	\beta\int_0^t \|\mu_{\lambda,	 \xi}(\tau)-
	\overline{\mu_{\lambda,	 \xi}(\tau)}\|^2_\bH \:\d \tau\\
	=  \mathcal{E}_\lambda(\varphi_0) + \int_0^t (\mu_{\lambda,	 \xi}(\tau), K_{\lambda,	 \xi}(\varphi_{\lambda,	 \xi}(\tau))\:\d W(\tau))_H
	+ \dfrac{1}{2}\int_0^t  \|\nabla K_{\lambda,	 \xi}(\varphi_{\lambda,	 \xi}(\tau))\|^2_\b{\cL^2(U,\bH)}
	\:\d \tau \\
	 + \dfrac{1}{2}\int_0^t\sum_{k  = 1}^{+\infty}\int_{\Td} F''_\lambda(\varphi_{\lambda,	 \xi}(\tau))|\b{g}'_k(J_\lambda(\varphi_{\lambda,	 \xi}(\tau))) \cdot \nabla \mathcal{R}_\xi(\varphi_{\lambda,	 \xi}(\tau))|^2  \:\d \tau.
\end{multline}
Let $\gamma \geq 0$ be arbitrary but fixed, and set
\[
\Lambda(t) := \gamma\int_0^t \|\varphi_{\lambda,\xi}(\tau)\|^2_{H^2(\Td)} \:\d\tau.
\]
The expression for the It\^{o} differential of the energy functional given by \eqref{eq:unif50} enables us to apply the It\^{o} product rule to $e^{-\Lambda(t)}\mathcal{E}_\lambda(\varphi_{\lambda, \xi}(t))$, resulting in
\begin{multline} \label{eq:unif51}
	e^{-\Lambda(t)}\mathcal{E}_\lambda(\varphi_{\lambda, \xi}(t))
	+ \gamma \int_0^t e^{-\Lambda(\tau)}
	\|\varphi_{\lambda, \xi}(\tau)\|^2_{H^2(\Td)}
	\mathcal{E}_\lambda(\varphi_{\lambda, \xi}(\tau))\:\d\tau
	+\alpha \int_0^t e^{-\Lambda(\tau)}\|\nabla \mu_{\lambda,	 \xi}(\tau)\|^2_\bH \:\d \tau \\
	\leq \mathcal{E}_\lambda(\varphi_0) + \int_0^t e^{-\Lambda(\tau)}(\mu_{\lambda,	 \xi}(\tau), K_{\lambda,	 \xi}(\varphi_{\lambda,	 \xi}(\tau))\:\d W(\tau))_H + \dfrac{1}{2}\int_0^t e^{-\Lambda(\tau)} \|\nabla K_{\lambda,	 \xi}(\varphi_{\lambda,	 \xi}(\tau))\|^2_\b{\cL^2(U,\bH)} \:\d\tau \\ + \dfrac 12\int_0^te^{-\Lambda(\tau)} \sum_{k  = 1}^{+\infty}\int_{\Td} F''_\lambda(\varphi_{\lambda,	 \xi}(\tau))|\b{g}'_k(J_\lambda(\varphi_{\lambda,	 \xi}(\tau))) \cdot \nabla \mathcal{R}_\xi(\varphi_{\lambda,	 \xi}(\tau))|^2 \:\d \tau.
\end{multline}
In particular, adding to both sides of \eqref{eq:unif51} the quantity
\[
\dfrac{1}{2C_4}\int_0^t e^{-\Lambda(\tau)}|\overline{\mu_{\lambda,	 \xi}(\tau)}|^2 \:\d\tau,
\]
recalling also that $v \mapsto (\|\nabla v\|^2_\bH + |\overline{v}|^2)^\frac 12$ is an equivalent norm in $V$, we are able to reconstruct the full $V$-norm of $\mu_{\lambda,	 \xi}$ to the left hand side of \eqref{eq:unif51}. Moreover, owing to \cite[Theorem A.2]{GMT2019} and the fact that $\Lambda$ is increasing, we have
\[
\int_0^t e^{-\Lambda(\tau)}\|\varphi_{\lambda, \xi}(\tau)\|^2_{W^{2,6}(\Td)} \:\d\tau + \int_0^t e^{-\Lambda(\tau)}\|\Psi'_\lambda(\varphi_{\lambda, \xi}(\tau))\|^2_{L^6(\Td)} \:\d\tau \leq C\left(1 + \int_0^t e^{-\Lambda(\tau)}\|\mu_{\lambda,	 \xi}(\tau)\|^2_V \:\d \tau\right).
\]
Hence, recalling also \eqref{eq:unif4}, \eqref{eq:unif50} becomes
\begin{align}
	\nonumber
	&e^{-\Lambda(t)}\mathcal{E}_\lambda(\varphi_{\lambda, \xi}(t))
	+ \gamma \int_0^t e^{-\Lambda(\tau)}
	\|\varphi_{\lambda, \xi}(\tau)\|^2_{H^2(\Td)}
	\mathcal{E}_\lambda(\varphi_{\lambda, \xi}(\tau))\:\d\tau \\
	\nonumber
	&\qquad+ \int_0^t e^{-\Lambda(\tau)}
	\left( \|\varphi_{\lambda, \xi}(\tau)\|^2_{W^{2,6}(\Td)}
	+ \|\Psi'_\lambda(\varphi_{\lambda, \xi}(\tau))\|^2_{L^6(\Td)}
	+ \|\mu_{\lambda,	 \xi}(\tau)\|^2_V\right) \:\d \tau \\
	\nonumber
	&\leq C\bigg[ 1 + \mathcal{E}_\lambda(\varphi_0)
	+ \int_0^t e^{-\Lambda(\tau)}
	(\mu_{\lambda,\xi}(\tau), K_{\lambda, \xi}(\varphi_{\lambda,\xi}(\tau))\:\d W(\tau))_H
	+ \int_0^t e^{-\Lambda(\tau)}
	\|\nabla K_{\lambda,	 \xi}(\varphi_{\lambda, \xi}(\tau))\|^2_\b{\cL^2(U,\bH)} \:\d\tau  \\
	\nonumber
	&\qquad+ \int_0^t e^{-\Lambda(\tau)} \|\varphi_{\lambda,\xi}(\tau)\|^2_H \:\d\tau
	+ \int_0^t e^{-\Lambda(\tau)}
	|(\Psi_\lambda'(\varphi_{\lambda, \xi}(\tau)),
	\overline{\varphi_0} - \overline{\varphi_{\lambda, \xi}(\tau)})_H|^2 \:\d\tau \\
	&\qquad+  \int_0^te^{-\Lambda(\tau)}
	\sum_{k  = 1}^{+\infty}\int_{\Td}
	F''_\lambda(\varphi_{\lambda, \xi}(\tau))
	|\b{g}'_k(J_\lambda(\varphi_{\lambda, \xi}(\tau)))
	\cdot \nabla \mathcal{R}_\xi\varphi_{\lambda, \xi}(\tau)|^2 \:\d \tau \bigg].
	\label{eq:unif52}
\end{align}
Next, we give controls for the terms to the right hand side of \eqref{eq:unif52}. First of all, recalling Lemma \ref{lem:resolvent} we have
\begin{equation*} 
	\begin{split}
		\|\nabla K_{\lambda,	 \xi}(\varphi_{\lambda,\xi}(\tau))\|^2_\b{\cL^2(U,\bH)} & =
		\sum_{k  = 1}^{+\infty} \|\nabla(\b{g}_k'(J_\lambda(\varphi_{\lambda,	 \xi}(\tau)))
		\cdot \nabla \mathcal{R}_\xi\varphi_{\lambda,	 \xi}(\tau))\|_\bH^2 \\
		& \leq 2\sum_{k  = 1}^{+\infty}
		\left[\|
		\left[ \b{g}_k''(J_\lambda(\varphi_{\lambda,	 \xi}(\tau))) \otimes
		\nabla J_\lambda(\varphi_{\lambda,\xi}(\tau)) \right]\nabla \mathcal{R}_\xi
		\varphi_{\lambda,	 \xi}(\tau))\|_\bH^2\right.  \\
		& \qquad \left.+ \|D^2\mathcal{R}_\xi\varphi_{\lambda,\xi}(\tau)
		\b{g}'_k(J_\lambda(\varphi_{\lambda,	 \xi}(\tau)))\|^2_\bH\right] \\
		& \leq 2L_\b{G}^2\left( \|D^2\mathcal{R}_\xi\varphi_{\lambda, \xi}(\tau)\|_\bH^2 + \|\nabla J_\lambda(\varphi_{\lambda,	 \xi}(\tau)) \cdot \nabla \mathcal{R}_\xi\varphi_{\lambda, \xi}(\tau)\|^2_H \right) \\
		& \leq C\left( \|\varphi_{\lambda, \xi}(\tau)\|_{H^2(\Td)}^2 + \|\nabla J_\lambda(\varphi_{\lambda,	 \xi}(\tau)) \cdot \nabla \mathcal{R}_\xi\varphi_{\lambda, \xi}(\tau)\|^2_H \right) \\
		& = C\left( \|\varphi_{\lambda, \xi}(\tau)\|_{H^2(\Td)}^2 + \ii{\Td}{}{J'_\lambda(\varphi_{\lambda,	 \xi}(\tau))\nabla J_\lambda(\varphi_{\lambda,	 \xi}(\tau)) \cdot \nabla \varphi_{\lambda,	 \xi}(\tau)|\nabla \mathcal{R}_\xi \varphi_{\lambda,	 \xi}(\tau)|^2}{x} \right).
	\end{split}
\end{equation*}
As $\nabla J_\lambda(\varphi_{\lambda,	 \xi}) \cdot \nabla \varphi_{\lambda,	 \xi} \geq 0$,
$0\leq J'_\lambda(\varphi_{\lambda,	 \xi}) \leq 1$,
and $|J_\lambda(\varphi_{\lambda,	 \xi})| \leq 1$,
we have, integrating by parts and
exploiting the Sobolev embedding $W^{2,6}(\Td) \embed W^{1,\infty}(\Td)$
\begin{align} 
	\nonumber
	\|\nabla K_{\lambda,	 \xi}(\varphi_{\lambda, \xi}(\tau))\|^2_\b{\cL^2(U,\bH)}  & \leq
	C\left( \|\varphi_{\lambda, \xi}(\tau)\|_{H^2(\Td)}^2 +
	\ii{\Td}{}{\nabla J_\lambda(\varphi_{\lambda,\xi}(\tau))
		\cdot \nabla \varphi_{\lambda,\xi}(\tau)
		|\nabla \mathcal{R}_\xi \varphi_{\lambda,\xi}(\tau)|^2}{x} \right) \\
	\nonumber
	& = C\bigg( \|\varphi_{\lambda, \xi}(\tau)\|_{H^2(\Td)}^2
	- \ii{\Td}{}{J_\lambda(\varphi_{\lambda,\xi}(\tau))
		\Delta\varphi_{\lambda,\xi}(\tau)|\nabla \mathcal{R}_\xi
		\varphi_{\lambda,\xi}(\tau)|^2}{x} \\
	\nonumber
	& \qquad-  \ii{\Td}{}{J_\lambda(\varphi_{\lambda,\xi}(\tau))
		\nabla \varphi_{\lambda,\xi}(\tau)\cdot
		\nabla|\nabla \mathcal{R}_\xi \varphi_{\lambda,\xi}(\tau)|^2 }{x} \bigg)\\
	\nonumber
	& \leq C\bigg( \|\varphi_{\lambda, \xi}(\tau)\|_{H^2(\Td)}^2
	+ \ii{\Td}{}{ |\Delta\varphi_{\lambda,\xi}(\tau)|
		|\nabla \mathcal{R}_\xi \varphi_{\lambda,\xi}(\tau)|^2}{x} \\
	\nonumber
	& \qquad+  \ii{\Td}{}{|D^2R_\xi\varphi_{\lambda,\xi}(\tau)|
		|\nabla \mathcal{R}_\xi \varphi_{\lambda,\xi}(\tau)|
		|\nabla \varphi_{\lambda,	 \xi}(\tau)|}{x} \bigg)\\
	\nonumber
	& \leq C\bigg( \|\varphi_{\lambda, \xi}(\tau)\|_{H^2(\Td)}^2
	+ \|\varphi_{\lambda,\xi}(\tau)\|_{H^2(\Td)}
	\|\nabla \varphi_{\lambda,\xi}(\tau)\|_\bH
	\|\nabla \varphi_{\lambda,\xi}(\tau)\|_{\b{L}^\infty(\Td)} \bigg) \\
	& \leq \varepsilon_3\|\varphi_{\lambda,\xi}(\tau)\|^2_{W^{2,6}(\Td)}
	+ C\|\varphi_{\lambda, \xi}(\tau)\|_{H^2(\Td)}^2
	\left( 1 + \|\nabla \varphi_{\lambda,\xi}(\tau)\|_\bH^2 \right),
	\label{eq:unif53}
\end{align}
where $\varepsilon_3 > 0$ is arbitrary and $C$ depends on $\varepsilon_3$. Then, we have, owing to the extra assumption \ref{iii} introduced at the beginning of Section~\ref{sec:existence},
\begin{align}
	\nonumber
	&\dfrac 12\int_0^te^{-\Lambda(\tau)} \sum_{k  = 1}^{+\infty}\int_{\Td} F''_\lambda(\varphi_{\lambda,	 \xi}(\tau))|\b{g}'_k(J_\lambda(\varphi_{\lambda,	 \xi}(\tau))) \cdot \nabla \mathcal{R}_\xi\varphi_{\lambda,	 \xi}(\tau)|^2\:\d\tau \\
	&\qquad
	\leq C\int_0^te^{-\Lambda(\tau)}
	\|\nabla \varphi_{\lambda,\xi}(\tau)\|_\bH^2 \:\d\tau.
	\label{eq:unif54}
\end{align}
Finally, we have
\begin{equation} \label{eq:unif54bis}
	\int_0^t e^{-\Lambda(\tau)} |(\Psi_\lambda'(\varphi_{\lambda,	 \xi}(\tau)), \overline{\varphi_0} - \overline{\varphi_{\lambda, \xi}(\tau)})_H|^2 \:\d\tau \leq C\sup_{\tau \in [0,t]}|\overline{\varphi_0} - \overline{\varphi_{\lambda, \xi}(\tau)}|^2\int_0^t e^{-\Lambda(\tau)}\|\Psi_\lambda'(\varphi_{\lambda,	 \xi}(\tau))\|^2_H \:\d\tau.
\end{equation}
Let $\varepsilon_3 = \frac{1}{2}$.
Observing that by property \ref{pty:belowbound} of the Yosida approximation
one has
\begin{align}
	\nonumber
	& \gamma \int_0^t e^{-\Lambda(\tau)}
	\|\varphi_{\lambda,\xi}(\tau)\|^2_{H^2(\Td)}
	\mathcal{E}_\lambda(\varphi_{\lambda, \xi}(\tau))\:\d\tau \\
	\nonumber
	& \geq \gamma \int_0^t e^{-\Lambda(\tau)}
	\|\varphi_{\lambda,\xi}(\tau)\|^2_{H^2(\Td)}
	\|\nabla \varphi_{\lambda, \xi}(\tau)\|^2_\bH\:\d\tau
	+ \frac\gamma{M} \int_0^t e^{-\Lambda(\tau)}
	\|\varphi_{\lambda, \xi}(\tau)\|^2_{H^2(\Td)}
	\|\varphi_{\lambda, \xi}(\tau)\|^2_{H}\\
	& \qquad
	-M\gamma \int_0^t
	e^{-\Lambda(\tau)}\|\varphi_{\lambda, \xi}(\tau)\|^2_{H^2(\Td)}\:\d\tau.
	\label{eq:unif55}
\end{align}
on account of \eqref{eq:unif52}-\eqref{eq:unif55} and Assumption \ref{hyp:potential}, choosing $\gamma$ sufficiently large yields the following estimate
\begin{align}
	\nonumber
	&e^{-\Lambda(t)}\mathcal{E}_\lambda(\varphi_{\lambda, \xi}(t))
	+ \int_0^t e^{-\Lambda(\tau)}
	\|\varphi_{\lambda,\xi}(\tau)\|^2_{H^2(\Td)}
	\|\varphi_{\lambda, \xi}(\tau)\|^2_V\:\d\tau \\
	\nonumber
	&\qquad+ \int_0^t e^{-\Lambda(\tau)}
	\left( \|\varphi_{\lambda, \xi}(\tau)\|^2_{W^{2,6}(\Td)}
	+ \|\Psi'_\lambda(\varphi_{\lambda, \xi}(\tau))\|^2_{L^6(\Td)}
	+ \|\mu_{\lambda,	 \xi}(\tau)\|^2_V\right) \:\d \tau \\
	\nonumber
	&\leq C\bigg[ 1 + \mathcal{E}_\lambda(\varphi_0)
	+ \int_0^t e^{-\Lambda(\tau)}(\mu_{\lambda,\xi}(\tau),
	K_{\lambda, \xi}(\varphi_{\lambda,	 \xi}(\tau))\:\d W(\tau))_H
	+ \int_0^t e^{-\Lambda(\tau)}
	\mathcal{E}_\lambda(\varphi_{\lambda, \xi}(\tau)) \:\d\tau \\
	&\qquad+ \int_0^t e^{-\Lambda(\tau)} \|\varphi_{\lambda, \xi}(\tau)\|^2_{H^2(\Td)} \:\d\tau
	+ \sup_{\tau \in [0,t]}|\overline{\varphi_0}
	- \overline{\varphi_{\lambda, \xi}(\tau)}|^2
	\int_0^t e^{-\Lambda(\tau)}
	\|\Psi_\lambda'(\varphi_{\lambda,\xi}(\tau))\|^2_H \:\d\tau\bigg].
	\label{eq:unif56}
\end{align}
Recalling that $e^{-\Lambda(t)} \leq 1$ for every $t \in [0,T]$, we take $\P$-expectations and suprema in the interval $[0,t]$ of the resulting inequality, yielding, thanks also to the fact that the stochastic integral has null mean,
\begin{align}
	\nonumber
	&\supp \E e^{-\Lambda(\tau)}\mathcal{E}_\lambda(\varphi_{\lambda, \xi}(\tau))
	+ \E \int_0^t e^{-\Lambda(\tau)}
	\|\varphi_{\lambda,\xi}(\tau)\|^2_{H^2(\Td)}
	\|\varphi_{\lambda, \xi}(\tau)\|^2_V\:\d\tau
	+\E \int_0^t e^{-\Lambda(\tau)}\|\mu_{\lambda,	 \xi}(\tau)\|^2_V \:\d \tau\\
	\nonumber
	&\qquad
	+ \E \int_0^t e^{-\Lambda(\tau)}
	\|\varphi_{\lambda, \xi}(\tau)\|^2_{W^{2,6}(\Td)} \:\d\tau
	+ \E \int_0^t e^{-\Lambda(\tau)}
	\|\Psi'_\lambda(\varphi_{\lambda, \xi}(\tau))\|^2_{L^6(\Td)} \:\d\tau  \\
	\nonumber
	&\leq C\bigg[ 1 +  \E\mathcal{E}_\lambda(\varphi_0)
	+ \E \int_0^t e^{-\Lambda(\tau)}
	\mathcal{E}_\lambda(\varphi_{\lambda, \xi}(\tau)) \:\d\tau
	+  \E \int_0^t \|\varphi_{\lambda,\xi}(\tau)\|^2_{H^2(\Td)} \:\d\tau\\
	&\qquad
	+ \E \sup_{\tau \in [0,t]}
	|\overline{\varphi_0} - \overline{\varphi_{\lambda, \xi}(\tau)}|^2
	\int_0^t \|\Psi_\lambda'(\varphi_{\lambda,\xi}(\tau))\|^2_H \:\d\tau\bigg].
	\label{eq:unif57}
\end{align}
Observe now that the last term satisfies,
by \eqref{eq:unif3}, \eqref{eq:unif4}, and \eqref{eq:unif1}
with $p =4$, as well as the H\"{o}lder inequality
\begin{align*}
	\E \sup_{\tau \in [0,t]}|\overline{\varphi_0} - \overline{\varphi_{\lambda, \xi}(\tau)}|^2\int_0^t \|\Psi_\lambda'(\varphi_{\lambda,	 \xi}(\tau))\|^2_H \:\d\tau & \leq \left[\E \sup_{\tau \in [0,t]}|\overline{\varphi_0} - \overline{\varphi_{\lambda, \xi}(\tau)}|^4\right]^\frac 12\left[\E \left| \int_0^t \|\Psi_\lambda'(\varphi_{\lambda,	 \xi}(\tau))\|^2_H \:\d\tau\right|^2\right]^\frac 12 \\
	& \leq C\lambda \|\Psi_\lambda'(\varphi_{\lambda,	 \xi}(\tau))\|^2_{L^4_\cP(\Omega;L^2(0,T;H))} \\
	& \leq C.
\end{align*}
In light of \eqref{eq:unif1} with $p = 2$ and
the Gronwall lemma, there exists a constant $C_5 > 0$ such that
\begin{align}
	\nonumber
	&\supp \E e^{-\Lambda(\tau)}
	\mathcal{E}_\lambda(\varphi_{\lambda, \xi}(\tau))
	+ \E \int_0^t e^{-\Lambda(\tau)}\|\varphi_{\lambda, \xi}(\tau)\|^2_{H^2(\Td)}
	\|\varphi_{\lambda, \xi}(\tau)\|^2_V\:\d\tau
	+\E \int_0^t e^{-\Lambda(\tau)}\|\mu_{\lambda,	 \xi}(\tau)\|^2_V \:\d \tau\\
	&\qquad
	+ \E \int_0^t e^{-\Lambda(\tau)}\|\varphi_{\lambda, \xi}(\tau)\|^2_{W^{2,6}(\Td)} \:\d\tau
	+ \E \int_0^t e^{-\Lambda(\tau)}
	\|\Psi'_\lambda(\varphi_{\lambda, \xi}(\tau))\|^2_{L^6(\Td)} \:\d\tau   \leq C_5.
	\label{eq:unif5}
\end{align}
\paragraph{\textit{Sixth estimate.}} Here we give an exponential version of the first estimate \eqref{eq:unif1} with $p = 2$. The approach resembles the one given in \cite[Subsection 4.1]{orr-roc-scar}. Taking \eqref{eq:unif11}-\eqref{eq:unif13_ter} into account in \eqref{eq:unif10} yields
\begin{align}
	\nonumber
	&\|\varphi_{\lambda,\xi}(t)\|_{H}^2 +
	\int_0^t \|\Delta \varphi_{\lambda,\xi}(\tau)\|^2_H + \|\nabla \varphi_{\lambda,\xi}(\tau)\|^2_\bH \:\d\tau +
	\int_0^t \left( \Psi''_\lambda(\varphi_{\lambda,\xi}(\tau)),
	|\nabla \varphi_{\lambda, \xi}|^2)\right)_H  \: \d\tau \\
	&\qquad\leq C \left[ \|\varphi_{0}\|_{H}^2 +
	\int_0^t \left(\varphi_{\lambda,\xi}(\tau),
	K_{\lambda,\xi}(\varphi_{\lambda,\xi}(\tau))\,\mathrm{d}W(\tau)\right)_H +
	\int_0^t \|\varphi_{\lambda,\xi}(\tau)\|^2_{H} \: \mathrm{d}\tau \right].
	\label{eq:unif60}
\end{align}
For some arbitrary but fixed $\kappa > 0$, sum and substract in \eqref{eq:unif60} the quantity
\[
\dfrac{\kappa}{2}\int_0^t \left|\left(\varphi_{\lambda,\xi}(\tau), K_{\lambda,\xi}(\varphi_{\lambda,\xi}(\tau))\right)_H\right|^2 \:\d\tau,
\]
where
\[
\left(\varphi_{\lambda,\xi}(\tau), K_{\lambda,\xi}(\varphi_{\lambda,\xi}(\tau))\right)_H := \sum_{k  = 1}^{+\infty}\left(\varphi_{\lambda,\xi}(\tau), \b{g}'_k(J_\lambda(\varphi_{\lambda,	 \xi})) \cdot \nabla \mathcal{R}_\xi \varphi_{\lambda,	 \xi}\right)_H.
\]
Observe that, by the extra assumption \ref{iii} introduced at the beginning of
Section~\ref{sec:existence},
\begin{align}  
	\nonumber
	\left|\left(\varphi_{\lambda,\xi}(\tau),
	K_{\lambda,\xi}(\varphi_{\lambda,\xi}(\tau))\right)_H\right|^2 & \leq
	2\left(\left|\left(J_\lambda(\varphi_{\lambda,\xi}(\tau)),
	K_{\lambda,\xi}(\varphi_{\lambda,\xi}(\tau))\right)_H\right|^2
	+ \left|\left(\lambda\Psi'_\lambda(\varphi_{\lambda,\xi}(\tau)),
	K_{\lambda,\xi}(\varphi_{\lambda,\xi}(\tau))\right)_H\right|^2\right) \\
	\nonumber
	& \leq CL_\b{G}^2
	\left(\|J_\lambda(\varphi_{\lambda,\xi}(\tau))\|^2_{L^\infty(\Td)}
	\|\nabla \mathcal{R}_\xi\varphi_{\lambda,\xi}\|^2_\bH +
	\lambda^2\|\nabla \mathcal{R}_\xi\varphi_{\lambda,\xi}\|^2_\bH\right) \\
	\nonumber
	& \leq CL_\b{G}^2\left( 1 + \lambda^2\right)\|\nabla \varphi_{\lambda,\xi}\|^2_\bH \\
	\label{eq:unif61}
	& \leq \varepsilon_4\|\Delta \varphi_{\lambda, \xi}\|^2_H +
	C\| \varphi_{\lambda,	\xi}\|^2_H,
\end{align}
where $\varepsilon_4 > 0$ is arbitrary and $C$ depends on $\varepsilon_4$. Let now $\varepsilon_4 = \frac 12$. In light of \eqref{eq:unif61}, we arrive at
\begin{multline*}
	\|\varphi_{\lambda,\xi}(t)\|_{H}^2 +\int_0^t \|\Delta \varphi_{\lambda,\xi}(\tau)\|^2_H \:\d\tau +\int_0^t \left( \Psi''_\lambda(\varphi_{\lambda,\xi}(\tau)), |\nabla \varphi_{\lambda,	 \xi}|^2)\right)_H  \: \d\tau \\ \leq C \bigg[ \|\varphi_{0}\|_{H}^2 + \int_0^t \left(\varphi_{\lambda,\xi}(\tau), K_{\lambda,\xi}(\varphi_{\lambda,\xi}(\tau))\,\mathrm{d}W(\tau)\right)_H - \dfrac{\kappa}{2}\int_0^t \left|\left(\varphi_{\lambda,\xi}(\tau), K_{\lambda,\xi}(\varphi_{\lambda,\xi}(\tau))\right)_H\right|^2 \:\d\tau \\+ \int_0^t \|\varphi_{\lambda,\xi}(\tau)\|^2_{H} \: \mathrm{d}\tau \bigg].
\end{multline*}
After multiplying the above inequality by $\ell > 0$ and taking exponentials, the Young inequality yields
\begin{align}
	\nonumber
	&\exp\left(\ell\|\varphi_{\lambda,\xi}(t)\|_{H}^2\right) +
	\exp\left( \ell \int_0^t \|\Delta \varphi_{\lambda,\xi}(\tau)\|^2_H \:\d\tau\right) +
	\exp \left( \ell \int_0^t \left( \Psi''_\lambda(\varphi_{\lambda,\xi}(\tau)),
	|\nabla \varphi_{\lambda,	 \xi}|^2)\right)_H  \: \d\tau \right) \\
	\nonumber
	&\leq C \bigg[ \exp \left( 3\ell\|\varphi_{0}\|_{H}^2\right) +
	\exp\left( 3\ell \int_0^t \| \varphi_{\lambda,\xi}(\tau)\|^2_H \:\d\tau\right)\\
	&\quad+ \exp \left( 3\ell \int_0^t \left(\varphi_{\lambda,\xi}(\tau),
	K_{\lambda,\xi}(\varphi_{\lambda,\xi}(\tau))\,\mathrm{d}W(\tau)\right)_H
	- \dfrac{3\ell\kappa}{2}\int_0^t \left|\left(\varphi_{\lambda,\xi}(\tau),
	K_{\lambda,\xi}(\varphi_{\lambda,\xi}(\tau))\right)_H\right|^2 \:\d\tau \right) \bigg].
	\label{eq:unif62}
\end{align}
Setting $\kappa = 3\ell$, we have that the process
\[
t \mapsto \exp \left( 3\ell \int_0^t \left(\varphi_{\lambda,\xi}(\tau), K_{\lambda,\xi}(\varphi_{\lambda,\xi}(\tau))\,\mathrm{d}W(\tau)\right)_H - \dfrac{9\ell^2}{2}\int_0^t \left|\left(\varphi_{\lambda,\xi}(\tau), K_{\lambda,\xi}(\varphi_{\lambda,\xi}(\tau))\right)_H\right|^2 \:\d\tau \right)
\]
is a real positive local martingale, hence a real supermartingale, and therefore its expectation is bounded by the expectation of the initial state, namely
\begin{equation} \label{eq:unif63}
	\E\exp \left( 3\ell \int_0^t \left(\varphi_{\lambda,\xi}(\tau), K_{\lambda,\xi}(\varphi_{\lambda,\xi}(\tau))\,\mathrm{d}W(\tau)\right)_H - \dfrac{9\ell^2}{2}\int_0^t \left|\left(\varphi_{\lambda,\xi}(\tau), K_{\lambda,\xi}(\varphi_{\lambda,\xi}(\tau))\right)_H\right|^2 \:\d\tau \right) \leq 1.
\end{equation}
Taking $\P$-expectations and suprema in time of \eqref{eq:unif62} and exploiting \eqref{eq:unif63} we have
\begin{multline} \label{eq:unif64}
	\supp \E\exp\left(\ell\|\varphi_{\lambda,\xi}(\tau)\|_{H}^2\right) + \E\exp \left( \ell \int_0^t \left( \Psi''_\lambda(\varphi_{\lambda,\xi}(\tau)), |\nabla \varphi_{\lambda,	 \xi}|^2)\right)_H  \: \d\tau \right) \\+ \E\exp\left( \ell \int_0^t \|\Delta \varphi_{\lambda,\xi}(\tau)\|^2_H \:\d\tau\right) \leq C \bigg[ 1+ \E\exp \left( 3\ell\|\varphi_{0}\|_{H}^2\right) + \E\exp\left( 3\ell \int_0^t \| \varphi_{\lambda,\xi}(\tau)\|^2_H \:\d\tau\right)\bigg].
\end{multline}
Finally, the Jensen inequality and Fubini theorem entail that
\begin{equation*}
	\begin{split}
		\E \exp \left( \dfrac{\ell}{t}\int_0^t \|\varphi_{\lambda,	 \xi}(\tau)\|^2_H \:\d\tau \right) & \leq \E \left[ \dfrac{1}{t}\int_0^t \exp\left( \ell\|\varphi_{\lambda,	 \xi}(\tau)\|^2_H\right) \:\d\tau \right] \\
		& = \dfrac{1}{t}\int_0^t \E\exp\left( \ell\|\varphi_{\lambda,	 \xi}(\tau)\|^2_H\right) \:\d\tau \\
		& \leq \supp \E\exp\left( \ell\|\varphi_{\lambda,	 \xi}(\tau)\|^2_H\right) \:\d\tau,
	\end{split}
\end{equation*}
and using this estimate in \eqref{eq:unif64} we have
\begin{multline} \label{eq:unif65}
	\E \exp \left( \dfrac{\ell}{t}\int_0^t \|\varphi_{\lambda,	 \xi}(\tau)\|^2_H \:\d\tau \right) + \E\exp \left( \ell \int_0^t \left( \Psi''_\lambda(\varphi_{\lambda,\xi}(\tau)), |\nabla \varphi_{\lambda,	 \xi}|^2)\right)_H  \: \d\tau \right) \\+ \E\exp\left( \ell \int_0^t \|\Delta \varphi_{\lambda,\xi}(\tau)\|^2_H \:\d\tau\right) \leq C \bigg[ 1+ \E\exp \left( 3\ell\|\varphi_{0}\|_{H}^2\right) + \E\exp\left( 3\ell \int_0^t \| \varphi_{\lambda,\xi}(\tau)\|^2_H \:\d\tau\right)\bigg].
\end{multline}
If we consider a time instant $\sigma > 0$ such that
\[
\dfrac \ell \sigma - 3\ell \geq 0 \Leftrightarrow \sigma \leq \min\left\{	\frac 13, T\right\},
\]
then we are able to complete the estimate as follows. For convenience, we set $\sigma = \min\left\{	\frac 16, T\right\}$. Then $\frac \ell \sigma \geq 6\ell$ and \eqref{eq:unif65} reads
\begin{multline*}
	\E \exp \left( 6\ell\int_0^\sigma \|\varphi_{\lambda,	 \xi}(\tau)\|^2_H \:\d\tau \right) + \E\exp \left( \ell \int_0^\sigma \left( \Psi''_\lambda(\varphi_{\lambda,\xi}(\tau)), |\nabla \varphi_{\lambda,	 \xi}|^2)\right)_H  \: \d\tau \right) \\+ \E\exp\left( \ell \int_0^\sigma \|\Delta \varphi_{\lambda,\xi}(\tau)\|^2_H \:\d\tau\right) \\ \leq C \bigg[ 1+ \E\exp \left( 3\ell\|\varphi_{0}\|_{H}^2\right)\bigg] + \varepsilon_5\E\exp\left( 6\ell \int_0^\sigma \| \varphi_{\lambda,\xi}(\tau)\|^2_H \:\d\tau\right),
\end{multline*}
where $\varepsilon_5 > 0$ is arbitrary and the constant $C$ depends on $\varepsilon_5$. Let $\varepsilon_5 = \frac 12$. Applying a patching argument, choosing integer multiples of $\sigma$ as initial times, we are able to infer the existence of $C_6 > 0$ such that
\begin{multline} \label{eq:unif6}
	\E \exp \left( 6\ell\int_0^T \|\varphi_{\lambda,	 \xi}(\tau)\|^2_H \:\d\tau \right) \\ + \E\exp \left( \ell \int_0^T \left( \Psi''_\lambda(\varphi_{\lambda,\xi}(\tau)), |\nabla \varphi_{\lambda,	 \xi}|^2)\right)_H  \: \d\tau \right) + \E\exp\left( \ell \int_0^T \|\Delta \varphi_{\lambda,\xi}(\tau)\|^2_H \:\d\tau\right) \leq C_6,
\end{multline}
and, in particular, we have, choosing $\ell = q\gamma$, with $q > 0$ arbitrary
\[
\E \exp q\Lambda(T) \leq C_6.
\]
The constant $C_6$ may depend on $\ell$ (hence on $q$).
\paragraph{\textit{Seventh estimate.}} We are now in a position to give classical estimates for the Cahn--Hilliard equation. Consider for example the chemical potential $\mu_{\lambda, \xi}$. For any $q \geq 1$ and $r > 1$, \eqref{eq:unif6} and the H\"{o}lder inequality entail that
\begin{equation*}
	\begin{split}
		\E \left| \ii{0}{T}{\|\mu_{\lambda, \xi}(\tau)\|^2_V}{\tau}\right|^\frac{q}{2} & =\E \left| \ii{0}{T}{e^{\Lambda(\tau)}e^{-\Lambda(\tau)}\|\mu_{\lambda, \xi}(\tau)\|^2_V}{\tau}\right|^\frac{q}{2} \\
		& \leq \E \left[ e^{\frac{q}{2}\Lambda(T)}\left| \ii{0}{T}{e^{-\Lambda(\tau)}\|\mu_{\lambda, \xi}(\tau)\|^2_V}{\tau}\right|^\frac{q}{2} \right] \\
		& \leq \E \left[ e^{\frac{qr}{2(r-1)} \Lambda(T)}\right]^\frac{r-1}{r} \left[\E\left| \ii{0}{T}{e^{-\Lambda(\tau)}\|\mu_{\lambda, \xi}(\tau)\|^2_V}{\tau}\right|^\frac{qr}{2}\right]^\frac 1r\\
		& \leq C\left[\E\left| \ii{0}{T}{e^{-\Lambda(\tau)}\|\mu_{\lambda, \xi}(\tau)\|^2_V}{\tau}\right|^\frac{qr}{2}\right]^\frac 1r,
	\end{split}
\end{equation*}
where $C$ depends on $p$ and $q$. Looking back at \eqref{eq:unif5}, we find that
\[
\dfrac{qr}{2} = 1 \Leftrightarrow q = \dfrac{2}{r},
\]
and as $r > 1$, we find that $q$ can be chosen arbitrarily close to $2$. For any Banach space $X$, let
\[
L^{2-}_\cP(\Omega; X) := \bigcap_{1 \leq q < 2} L^q_\cP(\Omega;X).
\]
Then we have that
\begin{align*}
	\mu_{\lambda, \xi} & \in L^{2-}_\cP(\Omega; L^2(0,T;V)), \\
	\varphi_{\lambda, \xi} & \in L^{2-}_\cP(\Omega; L^2(0,T;W^{2,6}(\Td))), \\
	\Psi_\lambda'(\varphi_{\lambda,\xi})
	& \in L^{2-}_\cP(\Omega; L^2(0,T;L^6(\Td))).
\end{align*}
Furthermore, one has
\[
\Xi := \|\varphi_{\lambda,\xi}\|^2_{H^2(\Td)}
\|\varphi_{\lambda, \xi}\|^2_V \in L^{2-}_\cP(\Omega; L^1(0,T)).
\]
Moreover, there exists a constant $C_7$ independent of $\lambda$ and $\xi$ such that
\begin{align}
	\nonumber
	&\|\mu_{\lambda, \xi}\|_{L^q_\cP(\Omega; L^2(0,T;V))}
	+ \|\varphi_{\lambda, \xi}\|_{L^q_\cP(\Omega; L^2(0,T;W^{2,6}(\Td)))} \\
	&\qquad+
	\|\Psi'_\lambda(\varphi_{\lambda,\xi})\|_{L^q_\cP(\Omega; L^2(0,T;L^6(\Td)))}
	+ \|\Xi\|_{L^q_\cP(\Omega;L^1(0,T))} \leq C_7,
	\label{eq:unif7}
\end{align}
for every $q \in [1,2)$. The constant $C_7$ may depend on $q$.
\paragraph{\textit{Eighth estimate.}} Let us turn back to the energy estimate \eqref{eq:unif50}. Now, owing to \eqref{eq:unif7}, we are able to work with simpler controls. Indeed, recalling \eqref{eq:unif53} and \eqref{eq:unif54} (setting therein $\gamma = 0$),
taking the result to the power $\frac s2$, with $s > 0$ to be defined later, suprema in time and expectations yield
\begin{align}
	\nonumber
	&\E \supp \|\nabla \varphi_{\lambda, \xi}(\tau)\|^s_\bH
	+ \E \supp \|F_\lambda(\varphi_{\lambda,\xi}(\tau))\|_{L^1(\Td)}^\frac s2
	+ \E \left|\int_0^t \|\nabla \mu_{\lambda,\xi}(\tau)\|^2_\bH \:\d \tau\right|^\frac s2 \\
	\nonumber
	&\leq  C\bigg[ 1 + \E \|\nabla \varphi_0\|^s_\bH +
	\E \|F_\lambda(\varphi_0)\|_{L^1(\Td)}^{\frac s2}
	+ \E \sup_{u \in [0,t]} \left|\int_0^u (\mu_{\lambda, \xi}(\tau),
	K_{\lambda, \xi}(\varphi_{\lambda,\xi}(\tau))\:\d W(\tau))_H\right|^\frac{s}{2} \\
	\label{eq:unif81}
	&\qquad+\E \left| \int_0^t \|\varphi_{\lambda,\xi}(\tau)\|^2_{W^{2,6}(\Td)}\:\d\tau
	\right|^\frac s2
	+\E \left| \int_0^t (1+\|\varphi_{\lambda,\xi}(\tau)\|_{H^2(\Td)}^2)
	\|\nabla \varphi_{\lambda,\xi}(\tau)\|_\bH^2\right|^\frac s2 \bigg].
\end{align}
The Burkh\"{o}lder--Davis--Gundy inequality yields
\[
\begin{split}
	&\E \sup_{u \in [0,t]} \left|\int_0^u (\mu_{\lambda,	 \xi}(\tau), K_{\lambda,	 \xi}(\varphi_{\lambda,	 \xi}(\tau))\:\d W(\tau))_H\right|^\frac{s}{2} \\
	& \hspace{4cm}\leq \E  \left|\int_0^t \|\mu_{\lambda,	 \xi}(\tau)\|^2_H\|K_{\lambda,	 \xi}(\varphi_{\lambda,	 \xi}(\tau))\|^2_{\cL^2(U,H)}\:\d\tau\right|^\frac{s}{4} \\
	& \hspace{4cm} \leq C\E  \left|\int_0^t \|\mu_{\lambda,	 \xi}(\tau)\|^2_H\|\varphi_{\lambda,	 \xi}(\tau)\|^2_{V}\:\d\tau\right|^\frac{s}{4} \\
	& \hspace{4cm} \leq C\E \left[ \supp\|\varphi_{\lambda,	 \xi}(\tau)\|^\frac s2_{V} \left|\int_0^t \|\mu_{\lambda,	 \xi}(\tau)\|^2_H\:\d\tau\right|^\frac{s}{4}\right] \\
	& \hspace{4cm} \leq \varepsilon_6\E  \supp\|\nabla \varphi_{\lambda,	 \xi}(\tau)\|^s_{\bH} + \varepsilon_6\E  \supp\| \varphi_{\lambda,	 \xi}(\tau)\|^s_{H} + C\E\left|\int_0^t \|\mu_{\lambda,	 \xi}(\tau)\|^2_V\:\d\tau\right|^\frac{s}{2},
\end{split}
\]
where $\varepsilon_6 > 0$ is arbitrary and $C$ depends on $\varepsilon_6$. Setting $\varepsilon_6 = \frac 12$, \eqref{eq:unif81} and the above control yield
\begin{align}
	\nonumber
	&\E \supp \|\nabla \varphi_{\lambda, \xi}(\tau)\|^s_\bH
	+ \E \supp \|F_\lambda(\varphi_{\lambda,	\xi}(\tau))\|_{L^1(\Td)}^\frac s2
	+ \E \left|\int_0^t \|\nabla \mu_{\lambda,	 \xi}(\tau)\|^2_\bH \:\d \tau\right|^\frac s2 \\
	\nonumber
	&\leq  C\bigg[ 1 + \E \|\nabla \varphi_0\|^s_\bH
	+ \E \|F_\lambda(\varphi_0)\|_{L^1(\Td)}^\frac s2
	+ \E\left|\int_0^t \|\mu_{\lambda,\xi}(\tau)\|^2_V\:\d\tau\right|^\frac{s}{2}
	+\E \left| \int_0^t \|\nabla \varphi_{\lambda,\xi}(\tau)\|^2_{\bH}\:\d\tau\right|^\frac s2 \\
	&\qquad
	+\E \left| \int_0^t \|\varphi_{\lambda,\xi}(\tau)\|^2_{W^{2,6}(\Td)}\:\d\tau\right|^\frac s2
	+\E \left| \int_0^t \|\varphi_{\lambda,\xi}(\tau)\|_{H^2(\Td)}^2
	\|\nabla \varphi_{\lambda,\xi}(\tau)\|_\bH^2\right|^\frac s2 \bigg].
	\label{eq:unif82}
\end{align}
In light of \eqref{eq:unif7}, we can set $1 \leq s < 2$ arbitrarily and apply the Gronwall lemma to get the energy estimate
\begin{equation} \label{eq:unif8}
	\|\nabla \varphi_{\lambda,	 \xi}\|_{L_\cP^s(\Omega;L^\infty(0,T;\bH))} \leq C_8.
\end{equation}
\paragraph{\textit{Further estimates.}} Here we collect further estimates that are needed in order to retrieve a solution to the original problem. In what follows, we set once and for all an arbitrary real quantity $s$ such that $1 \leq s < 2$. From \eqref{eq:unif1} and \eqref{eq:unif8} we get
\begin{equation} \label{eq:unif91}
	\|\varphi_{\lambda,	 \xi}\|_{L_\cP^s(\Omega;L^\infty(0,T;V))} \leq C_9.
\end{equation}
Since $C([0,T];H) \cap L^2(0,T;H^2(\Td))\embed L^4(0,T;V)$, from Lemma \ref{lem:lipschitz} and \eqref{eq:unif1} we have
\begin{equation*}
	\|K_{\lambda, \xi}(\varphi_{\lambda,	 \xi})\|_{L_\cP^p(\Omega;L^4(0,T;\cL^2(U,H)))} \leq C_{10}
\end{equation*}
As $p \geq 4$, applying \cite[Lemma 2.1]{fland-gat} we obtain the following estimate on the stochastic integrals:
\begin{equation} \label{eq:unif92}
	\left\| \int_0^\cdot K_{\lambda, \xi}(\varphi_{\lambda,	 \xi}(\tau))\,\mathrm{d}W(\tau)
	\right\|_{L^4_\cP(\Omega; W^{k,4}(0,T; H))} \leq C_{11},
\end{equation}
with $k \in (0, \frac 12)$. By comparison, using also the embedding $H^1(0,T;V^*) \embed W^{\frac 12, 4}(0,T;V^*)$ we infer that
\begin{equation} \label{eq:unif93}
	\|\varphi_{\lambda,	 \xi}\|_{L_\cP^s(\Omega;W^{k, 4}(0,T;V^*))} \leq C_{12}.
\end{equation}
with $k \in (0, \frac 12)$.
\subsection{Passage to the limit.} Let us now investigate tightness properties of the laws of the approximating solutions. Here, we consider a fixed value of $k\in (0, \frac 12)$ so that $4k > 1$.
\begin{lem} \label{lem:tight1}
	The family of laws of $\{\varphi_{\lambda,\xi}\}_{(\lambda, \xi) \in (0,\Lambda)\times (0,\lambda)}$ is tight in the space $\mathcal{Z} :=L^2(0,T;V) \cap C^0([0,T];H)$.
\end{lem}
\begin{proof}
	The proof of tightness follows a standard argument (see \cite[Subsection 3.3]{scarpa21} or \cite[Proposition 1]{Medjo21}). Since $4k > 1$, on account of \cite[Corollary 5]{simon}, we have that the embeddings
	\begin{equation*}
		L^\infty(0,T;V) \cap W^{k, 4}(0,T;V^*)  \embed C^0([0,T];H), \quad L^2(0,T;W^{2,6}(\Td)) \cap W^{k, 4}(0,T;V^*)  \embed L^2(0,T;V),
	\end{equation*}
	are compact. Here, the intersection spaces are endowed with their natural norm. For any $R > 0$, let $B_R$ denote the closed ball of radius $R$ in $L^\infty(0,T;V) \cap W^{k, 4}(0,T;V^*)$. Applying the Markov inequality, we have
	\[
	\begin{split}
		\P\left\{ \varphi_{\lambda,\xi} \in B_R^C \right\} & = \P\left\{ \|\varphi_{\lambda,\xi}\|_{L^\infty(0,T;V) \cap W^{k, 4}(0,T;V^*)} > R\right\} \\
		& \leq \dfrac{1}{R^s}\E\|\varphi_{\lambda,n}\|_{L^\infty(0,T;V) \cap W^{k, 4}(0,T;V^*)}^s \\
		& \leq \dfrac{C_{12}}{R^s},
	\end{split}
	\]
	where the last line follows from \eqref{eq:unif93}, provided that $1 \leq s < 2$. This yields
	\[
	\lim_{R\to+\infty} \sup_{\lambda \in (0,\Lambda)} \sup_{\xi \in (0,\lambda)} \, \P\left\{ \varphi_{\lambda,\xi} \in B_R^C \right\} = 0,
	\]
	hence we have tightness of laws in $C^0([0,T];H)$. The second claim can be proven analogously, with straightforward modifications.
\end{proof} \noindent
Next, defining
\[
K_{\lambda, \xi}(\varphi_{\lambda,	 \xi}) \cdot W := \int_0^\cdot K_{\lambda, \xi}(\varphi_{\lambda,	 \xi}(\tau))\,\mathrm{d}W(\tau).
\]
The proof of Lemma \ref{lem:tight1} can be adapted to prove a second tightness result.
\begin{lem} \label{lem:tight2}
	The family of laws of $(K_{\lambda, \xi}(\varphi_{\lambda,	 \xi}) \cdot W)_{(\lambda, \xi) \in (0,\Lambda)\times (0,\lambda)}$ is tight in the space $C^0([0,T];V^*)$.
\end{lem}
\begin{proof}
	In light of \cite[Theorem 2.2]{fland-gat}, we have the embedding
	\[
	W^{k, 4}(0,T;H) \embed C^0([0,T]; V^*)
	\]
	is compact.
	The argument of the proof of Lemma \ref{lem:tight1}, jointly with \eqref{eq:unif92}, is enough to complete the proof.
\end{proof} \noindent
For any couple $(\lambda,\xi) \in (0,\Lambda) \times (0,\lambda)$, we define the cylindrical Wiener process $W_{\lambda, \xi} \equiv W$. Since any measure on a complete separable metric space is tight, we directly infer that the family of laws of $\{W_{\lambda, \xi}\}_{(\lambda,\xi) \in (0,\Lambda) \times (0,\lambda)}$ is tight in $C^0([0,T]; U)$. Analogously, we set $\varphi_{0,\lambda,\xi} \equiv \varphi_0$ and it holds that the family of laws of  $\{\varphi_{0,\lambda,\xi}\}_{(\lambda,\xi) \in (0,\Lambda) \times (0,\lambda)}$ is tight in $H$. In the following, without loss of generality, we can set $\xi(\lambda) = \frac{\lambda}{2}$ and drop the index $\xi$ observing that $\lambda \to 0^+$ implies $\xi \to 0^+$. In order to pass to the limit, we apply a method due to Gy\"{o}ngy and Krylov (see \cite[Theorem 1.1]{gyo-kry}). Let $\{\lambda_n\}_{n \in \enne} \subset (0,\Lambda)$ be a strictly decreasing sequence such that $\lambda_n \to 0^+$ as $n \to +\infty$. Let $\{\lambda_{j(n)}\}_{n \in \enne}$ and $\{\lambda_{k(n)}\}_{n \in \enne}$ denote any two subsequences of $\lambda_n$.
\begin{lem} \label{lem:gkch}
	For any pair of subsequences $\{\varphi_{\lambda_{j(n)}}\}$ and $\{\varphi_{\lambda_{k(n)}}\}$ of $\{\varphi_{\lambda_n}\}$, there exists a joint subsequence $\{(\varphi_{\lambda_{\ell(j(n))}}, \varphi_{\lambda_{\ell(k(n))}})\}$ converging in law to some probability measure $\nu$ on the Borel subsets of $\mathcal{Z}^2$ such that
	\[
	\nu(\{(f_1, f_2) \in \mathcal{Z}^2: f_1 = f_2\}) = 1.
	\]
\end{lem}
\begin{proof}
	Owing to Lemmas \ref{lem:tight1} and \ref{lem:tight2}, and also thanks to the Prokhorov and Skorokhod theorems (see \cite[Theorem 2.7]{ike-wata} and \cite[Theorem 1.10.4, Addendum 1.10.5]{vaa-well}), there exists a probability space $(\tom, \widetilde{\mathscr{F}},\tP)$ and a sequence of random variables $Y_n:  (\tom, \tF)\to(\Omega, \mathscr{F})$ such that the law of $Y_n$ is $\P$ for every $n \in \enne$ and so that we can construct a joint converging subsequence of $(\varphi_{\lambda_{j(n)}}, \varphi_{\lambda_{k(n)}}) \circ Y_n$ such that
	\begin{align*}
		(\widetilde{\varphi}_{\lambda_{\ell(j(n))}}, \widetilde{\varphi}_{\lambda_{\ell(k(n))}})  := (\varphi_{\lambda_{\ell(j(n))}}, \varphi_{\lambda_{\ell(k(n))}}) \circ Y_n  \to (\widetilde{\varphi}_{1}, \widetilde{\varphi}_{2}), & \quad \text{in }\mathcal{Z}^2, \quad \P\text{-a.s.,} \\
		(\widetilde{J}_{\lambda_{\ell(j(n))}}, \widetilde{J}_{\lambda_{\ell(k(n))}}) := ((K_\lambda(\varphi_{\lambda_{\ell(j(n))}}) \cdot W_\lambda, K_\lambda(\varphi_{\lambda_{\ell(k(n))}}) \cdot W_\lambda) \circ Y_n \to (\widetilde{J}_{1}, \widetilde{J}_2)  & \quad \text{in }[C^0([0,T];V^*)]^2, \,\, \P\text{-a.s.,} \\
		(\widetilde W_{\lambda_{\ell(j(n))}}, \widetilde W_{\lambda_{\ell(k(n))}})  := (\widetilde W_{\lambda_{\ell(j(n))}}, \widetilde W_{\lambda_{\ell(k(n))}}) \circ Y_n  \to (\widetilde{W}_{1}, \widetilde{W}_{2}), & \quad \text{in }[C^0([0,T];U_0)]^2, \,\, \P\text{-a.s.,} \\
		(\widetilde{\varphi}_{0,\lambda_{\ell(j(n))}}, \widetilde{\varphi}_{0,\lambda_{\ell(k(n))}})  := (\varphi_{0,\lambda_{\ell(j(n))}}, \varphi_{0,\lambda_{\ell(k(n))}}) \circ Y_n  \to (\widetilde{\varphi}_{01}, \widetilde{\varphi}_{02}), & \quad \text{in }[H]^2, \,\, \P\text{-a.s.,}
	\end{align*}where the limit processes belong to the specified spaces.
	Furthermore, it is immediate to see that $\widetilde{W}_1 = \widetilde{W}_2 =: \widetilde{W}$ and $\widetilde{\varphi}_{01} =\widetilde{\varphi}_{02}$. Then, thanks to the uniform estimates proven throughout Subsection \ref{ssec:unifestCH}, the previous strong convergence property, the Vitali, Banach--Alaoglu, and Eberlein-\v{S}mulian theorems, we have the following converging subsequences:
	\begin{align*}
		(\widetilde{\varphi}_{\lambda_{\ell(j(n))}}, \widetilde{\varphi}_{\lambda_{\ell(k(n))}})  \to (\widetilde{\varphi}_{1}, \widetilde{\varphi}_{2})
		\quad & \text{in } [L^q(\tom;L^2(0,T;V) \cap C^0([0,T];H))]^2 \text{ if } q < p, \\
		(\widetilde{\varphi}_{\lambda_{\ell(j(n))}}, \widetilde{\varphi}_{\lambda_{\ell(k(n))}})  \rightharpoonup (\widetilde{\varphi}_{1}, \widetilde{\varphi}_{2})
		\quad & \text{in } [L^q(\tom;L^2(0,T;W^{2,6}(\Td)))]^2 \text{ if } q < 2, \\
		(\widetilde{\varphi}_{\lambda_{\ell(j(n))}}, \widetilde{\varphi}_{\lambda_{\ell(k(n))}})  \overset{\ast}{\rightharpoonup} (\widetilde{\varphi}_{1}, \widetilde{\varphi}_{2})
		\quad & \text{in } [L^q_w(\tom;L^\infty(0,T;V)) \cap L^q(\tom; W^{k,4}(0,T;V^*))]^2 \text{ if } q < 2, \\
		(\widetilde{J}_{\lambda_{\ell(j(n))}}, \widetilde{J}_{\lambda_{\ell(k(n))}}) \to (\widetilde{J}_{1}, \widetilde{J}_2)
		\quad & \text{in } [L^q(\tom;C^0([0,T];V^*))]^2 \text{ if }q < 4, \\
		(\widetilde W_{\lambda_{\ell(j(n))}}, \widetilde W_{\lambda_{\ell(k(n))}})  \to (\widetilde{W}, \widetilde{W}) \quad & \text{in }[L^q(\tom;C^0([0,T];U_0))]^2, \text{ if } q < p, \\
		(\widetilde{\varphi}_{0,\lambda_{\ell(j(n))}}, \widetilde{\varphi}_{0,\lambda_{\ell(k(n))}}) \to (\widetilde{\varphi}_{01}, \widetilde{\varphi}_{02}) \quad & \text{in }[L^q(\tom;H)]^2, \text{ if } q < p.
	\end{align*}
	Furthermore, defining $\widetilde{\mu}_{\lambda_n} := \mu_\lambda \circ Y_n$,
	we also have
	\[
	(\widetilde{\mu}_{\lambda_{\ell(j(n))}}, \widetilde{\mu}_{\lambda_{\ell(k(n))}})  \rightharpoonup (\widetilde{\mu}_{1}, \widetilde{\mu}_{2})
	\quad \text{in } [L^q(\tom;L^2(0,T;V))]^2 \text{ if } q < 2.
	\]
	By the weak-strong closure of maximal monotone operators (for example, \cite[Proposition 2.1]{barbu-monot}) and the Lipschitz continuity of $R'$, jointly with the strong convergence properties given above, we also have
	\[
	(F'_\lambda(\widetilde{\varphi}_{\lambda_{\ell(j(n))}}), F'_\lambda(\widetilde{\varphi}_{\lambda_{\ell(k(n))}}))  \rightharpoonup (F'(\widetilde{\varphi}_1), F'(\widetilde{\varphi}_2)) \text{ in } [L^q(\tom;L^2(0,T;L^6(\Td)))]^2 \text{ if } q < 2.
	\]
	Next, observe that we have
	\[
	\begin{split}
		& \|K_{\lambda}(\widetilde{\varphi}_{\lambda_{\ell(j(n))}})-\div \b{G}(\widetilde{\varphi}_{1})\|_{\cL^2(U,H)}^2 \\
		& \hspace{1cm} \leq 2\|K_{\lambda}(\widetilde{\varphi}_{\lambda_{\ell(j(n))}})-K_{\lambda}(\widetilde{\varphi}_{1})\|_{\cL^2(U,H)}^2 +2  \|K_{\lambda}(\widetilde{\varphi}_{1})-\div \b{G}(\widetilde{\varphi}_{1})\|_{\cL^2(U,H)}^2 \\
		& \hspace{1cm} \leq C\|\widetilde{\varphi}_{\lambda_{\ell(j(n))}} - \widetilde{\varphi}_1\|^2_V + \sum_{k=1}^{+\infty} \|\b{g}'_k(J_\lambda(\widetilde{\varphi}_{1})) \cdot \nabla \mathcal{R}_{\frac{\lambda}{2}}\widetilde{\varphi}_{1} - \b{g}'_k(\widetilde{\varphi}_1) \cdot \nabla \widetilde{\varphi}_1\|^2_{H} \\
		& \hspace{1cm} \leq C\left[ \|\widetilde{\varphi}_{\lambda_{\ell(j(n))}} - \widetilde{\varphi}_1\|^2_V + \sum_{k=1}^{+\infty} \|(\b{g}'_k(J_\lambda(\widetilde{\varphi}_{1})) - \b{g}'_k(\widetilde{\varphi}_{1})) \cdot \nabla \mathcal{R}_{\frac{\lambda}{2}}\widetilde{\varphi}_{1}\|^2_H  + \|\b{g}'_k(\widetilde{\varphi}_1) \cdot (\nabla \mathcal{R}_{\frac{\lambda}{2}}\widetilde{\varphi}_{1} - \nabla \widetilde{\varphi}_1)\|^2_{H} \right] \\
		& \hspace{1cm} \leq C\left( \|\widetilde{\varphi}_{\lambda_{\ell(j(n))}} - \widetilde{\varphi}_1\|^2_V + \lambda_{\ell(j(n))}^2\| \nabla \mathcal{R}_{\frac{\lambda}{2}}\widetilde{\varphi}_{1}\|^2_{\b{L}^3(\Td)}\|\Psi'_\lambda(\widetilde{\varphi}_{1})\|^2_{L^6(\Td)} + \dfrac{\lambda_{\ell(j(n))}}{4}\| \Delta\widetilde{\varphi}_{1}\|^2_H  \right) \\
	\end{split}
	\]
	where we also exploited \eqref{eq:ellipticestimate} and the H\"{o}lder inequality. Testing \eqref{eq:elliptic} by $-\Delta \mathcal{R}_{\frac \lambda 2} \widetilde{\varphi}_{1}$ in $H$, we arrive at the estimate
	\[
	\dfrac \lambda 2 \| \Delta \mathcal{R}_{\frac \lambda 2} \widetilde{\varphi}_{1}\|^2_H \leq \frac{1}{2}\|\nabla \widetilde{\varphi}_{1}\|^2_\bH,
	\]
	hence
	\[
	\dfrac \lambda 2\|  \mathcal{R}_{\frac{\lambda}{2}}\widetilde{\varphi}_{1}\|_{H^2(\Td)}^2 \leq \dfrac{\lambda C}{2} +  \dfrac 12\|\nabla \widetilde{\varphi}_{1}\|^2_\bH  \leq C,
	\]
	where the constant $C$ does not depend on time. Therefore, given the Sobolev embedding $H^2(\Td) \embed W^{1,3}(\Td)$ and the monotonicity property \ref{pty:monotonicityder} of the Yosida approximation, we arrive at
	\[
	\begin{split}
		& \|K_{\lambda}(\widetilde{\varphi}_{\lambda_{\ell(j(n))}})-\div \b{G}(\widetilde{\varphi}_{1})\|_{\cL^2(U,H)}^2 \\
		& \hspace{1cm} \leq C\left( \|\widetilde{\varphi}_{\lambda_{\ell(j(n))}} - \widetilde{\varphi}_1\|^2_V + \lambda_{\ell(j(n))}\|\Psi'(\widetilde{\varphi}_{1})\|^2_{L^6(\Td)} + \dfrac{\lambda_{\ell(j(n))}}{4}\| \Delta\widetilde{\varphi}_{1}\|^2_H  \right) \\
	\end{split}
	\]
	hence, integrating the inequality in time, raising the result to the power $\frac q2$ and taking expectations leads to
	\[
	K_{\lambda}(\widetilde{\varphi}_{\lambda_{\ell(j(n))}})\to \div \b{G}(\widetilde{\varphi}_{1}) \quad \text{in } L^{q}(\tom; L^2(0,T; \cL^2(U,H))) \text{ if } q < 2.
	\]
	In particular, this entails that $\widetilde{J}_1 = \div \b{G}(\widetilde{\varphi}_1) \cdot W$. The argument also works for $\widetilde{J}_2$ with clear modifications. Observe that, owing to the previous convergence properties, it is then clear that
	\begin{equation*}
		\begin{cases}
			\d (\widetilde{\varphi}_1-\widetilde{\varphi}_2) - \Delta(\widetilde{\mu}_1 - \widetilde{\mu}_2) \,\d t = \left[ \div \left( \b{G}(\widetilde{\varphi}_1) \right) - \div \left( \b{G}(\widetilde{\varphi}_2)\right) \right] \d \widetilde W & \quad \text{ in }\T^d \times (0,T), \\
			\widetilde{\mu}_1 - \widetilde{\mu}_2 = -\Delta (\widetilde{\varphi}_1-\widetilde{\varphi}_2) + F'(\widetilde{\varphi}_1)-F'(\widetilde{\varphi}_2) & \quad \text{ in }\T^d \times (0,T), \\
			\widetilde{\varphi}_1(\cdot \:, 0)-\widetilde{\varphi}_2(\cdot \:, 0) = 0 & \quad \text{ in } \Td.
		\end{cases}
	\end{equation*}
	Uniqueness of solutions (see Theorem \ref{thm:uniqueCH}) then implies that $	\widetilde{\varphi}_1(t) = \widetilde{\varphi}_2(t)$ for every $t \in [0,T]$, in the $\tP$-almost sure sense, and a result by Gy\"{o}ngy and Krylov (see \cite[Theorem 1.1]{gyo-kry}) yields the claim.
\end{proof} \noindent
A standard argument yields then existence of a strong
solution (with only finite $q$-moments, for $q < 2$).

\subsection{Recovery of higher-order moments and conclusion} \label{ssec:higher}
Finally, we show here that the solution actually has well-defined moments up to order $p$ (compare to \eqref{eq:unif7}) and we remove the extra assumptions \ref{i}--\ref{iii}
introduced at the beginning of the section. In this direction, let the assumptions \ref{hyp:potential}--\ref{hyp:G} be in order,
and let $\varphi_0$ be given as in Definition~\ref{def:solCH}. We introduce some approximations $(\varphi_{0n})_n$ and
$(\b{G}_n)_n$ on $\varphi_0$ and $\b G$, respectively, so that
$\varphi_{0n}$ and $\b G_n$ satisfy \ref{i}--\ref{iii} for every $n\in\enne$.
For every $n\in\enne$, we set
\[
  \varphi_{0n}:=T_n(\varphi_0)
\]
where $T_n:\erre\to\erre$ is the classical traction function given by
$T_n(x):=\max\{-n, \min\{x, n\}\}$, $x\in\erre$. It is immediate to check that
$\varphi_{0n}$ continues to satisfy the condition in Definition~\ref{def:solCH},
as well as the extra assumptions \ref{i}--\ref{ii}. As far as $\b G$ is concerned,
for every $n\in\enne$ the operator
$\b{G}_n:\mathbf B^{0,\infty}_1 \to \cL^2(U,\b{H})$ is defined as
	\[
	\b{G}_n(\psi)[u_k]= \b{g}_{k,n}(\psi) =
	(g_{k,n}^i(\psi))_{i=1}^d
	\qquad\forall\,k\in\enne_+\quad\forall\,\psi\in\mathbf B^{0,\infty}_1,
	\]
where the double sequence
$(\b g_{k,n})_{k,n}\subset W^{2,\infty}(-1,1;\erre^d)$ is chosen
such that, for all $k,n\in\enne$,
\[
  \b g_{k,n}(s)=
  \begin{cases}
  \b g_k(s) \quad&\forall\,s\in[-1+\frac1n, 1-\frac1n] ,\\
  \b0 \quad&\forall\,s\in[-1,-1+\frac1{2n}]\cup[1-\frac1{2n},1],
  \end{cases}
  \qquad
  \|\b g_{k,n}\|_{W^{2,\infty}(-1,1;\erre^d)}\leq \|\b g_k\|_{W^{2,\infty}(-1,1;\erre^d)}.
\]
Note that the choice of $(\b g_{k,n})_{k,n}$
is possible by using classical convolution
with a sequence of mollifiers with compact support, and
it implies that
\[
  \lim_{n\to\infty}\b g_{k,n}(s)= \b g_{k}(s) \quad\forall\,s\in(-1,1),\quad\forall\,k\in\enne.
\]
Since $\b g_{k,n}$ has compact support in $(-1,1)$ for every $k,n\in\enne$, it is clear
that $F'\b g_{k,n}', F''|\b g_{k,n}'|^2\in L^\infty(-1,1)$ for every $k,n\in\enne$.
Moreover, a direct computation and assumption \ref{hyp:G} readily give
\[
  \sum_{k  = 1}^{+\infty} \left[\|F'\b{g}'_{k,n}\|^2_{\b{L}^\infty(-1,1)}
	+\|F''|\b{g}'_{k,n}|^2\|_{\b{L}^\infty(-1,1)}
	\right]
  \leq\sup_{s\in[-1+\frac1{2n}, 1-\frac1{2n}]}\left(|F'(s)|^2+|F''(s)|\right)L_{\b G}^2,
\]
so that also \ref{iii} is satisfied for every $n\in\enne$. For every $n\in\enne$ we consider then the problem
\begin{equation} \label{eq:reg_n}
	\begin{cases}
		\d \varphi_n + [-\alpha \Delta \mu_n + \beta(\mu_n-\overline{\mu_n}) ] \,\d t
		=\operatorname{div}\b G_n(\varphi_n) \,\d W & \quad \text{ in }\T^d \times (0,T), \\
		\mu_n = -\Delta \varphi_n + F'(\varphi_n) & \quad \text{ in }\T^d \times (0,T), \\
		\varphi_n(\cdot \:, 0) = \varphi_{0n} & \quad \text{ in } \Td.
	\end{cases}
\end{equation}
Taking into account the remarks made above, we are here in the
setting given by the extra conditions \ref{i}--\ref{iii}: hence, by the results proved in the previous section we know that the regularised problem \eqref{eq:reg_n} admits a
unique strong solution $(\varphi_n, \mu_n)$ in the sense of Definition~\ref{def:solCH}
(with only finite $q$-moments, for $q < 2$).
Let us show some uniform estimates independent of $n$, which guarantee
that as $n\to\infty$ one can recover a strong solution to the original problem
with full moments up to order $p$ as well, and without the extra assumptions \ref{i}--\ref{iii}.
The main novel point is that here one can exploit the natural bound $|\varphi_n|\leq 1$,
which was not guaranteed at the approximation in $\lambda$ performed before. First of all, a direct computation shows that
\[
  \|\varphi_{0n}\|_H\leq\|\varphi_0\|_H \quad\forall\,n\in\enne,\quad\P\text{-a.s.}
\]
and
\[
  \|\operatorname{div}\b G_n(\psi)\|_{\cL^2(U,\b H)}^2\leq L_{\b G}^2\|\nabla\psi\|_{\b H}^2
  \quad\forall\,\psi\in V, \quad\forall\,n\in\enne.
\]
Consequently, proceeding as in the \textit{First estimate} of
Subsection~\ref{ssec:unifestCH} yields analogously that
\begin{equation} \label{eq:unif1_n}
	\|\varphi_n\|_{L^p_\cP(\Omega;C^0([0,T];H))}
	+ \|\sqrt{\Psi''(\varphi_n)}\nabla \varphi_n\|_{L^p_\cP(\Omega; L^2(0,T; \b H))}
	+ \|\varphi_n\|_{L^p_\cP(\Omega; L^2(0,T;H^2(\Td)))} \leq C,
\end{equation}
where $C>0$ is independent of $n$. Secondly,
we prove some an energy inequality arguing on the original problem.
By weak lower semicontinuity we get the energy inequality
\begin{multline} \label{eq:energy1}
	\mathcal{E}(\varphi_n(t)) +
	\alpha\int_0^t \|\nabla \mu_n(\tau)\|^2_\bH \:\d \tau
	+\beta\int_0^t \| \mu_n(\tau)-\overline{\mu_n(\tau)}\|^2_\bH \:\d \tau \leq
	\mathcal{E}(\varphi_{0n})
	+ \int_0^t (\mu_n(\tau), \div \b{G}_n(\varphi_n(\tau))\:\d W(\tau))_H \\
	+ \dfrac{1}{2}\int_0^t \left[ \|\nabla \div \b{G}_n(\varphi_n(\tau))\|^2_\b{\cL^2(U,\bH)}
	+ \sum_{k  = 1}^{+\infty}\int_{\Td} F''(\varphi_n(\tau))|\b{g}'_{k,n}(\varphi_n(\tau)) \cdot \nabla\varphi_n(\tau)|^2 \right] \:\d \tau.
\end{multline}
Then, we bound the last two terms as follows. First, we have, by Assumption \ref{hyp:G}
\begin{align} \notag
	\sum_{k  = 1}^{+\infty}\int_{\Td} F''(\varphi_n(\tau))|\b{g}'_{k,n}(\varphi_n(\tau)) \cdot \nabla\varphi_n(\tau)|^2 &\leq C
	\int_{\Td}(1+\Psi''(\varphi_n(\tau)))|\nabla\varphi_n(\tau)|^2\\
	\label{eq:energy2}
	&\leq
	C\left(\|\nabla\varphi_n(\tau)\|^2_{\b H} +
	 \|\sqrt{\Psi''(\varphi_n(\tau))}\nabla\varphi_n(\tau)\|_{\b H}^2\right)
\end{align}
and then we have, recalling that $|\varphi_n| \leq 1$ almost everywhere in $\Omega \times (0,T) \times \Td$,
\begin{equation} \label{eq:energy3}
	\begin{split}
		\|\nabla \div \b{G}_n(\varphi_n(\tau))\|^2_\b{\cL^2(U,\bH)} & =
		\sum_{k  = 1}^{+\infty} \| \b{g}''_{k,n}(\varphi_n(\tau)) |\nabla \varphi_n(\tau)|^2 +
		D^2\varphi_n(\tau)\,\b{g}'_{k,n}(\varphi_n(\tau))\|^2_H \\
		& \leq 2L_{\b G}^2\left( \|\varphi_n(\tau)\|^2_{H^2(\Td)}
		+ \|\varphi_n(\tau)\|_{W^{1,4}(\Td)}^4 \right) \\
		& \leq CL_{\b G}^2\|\varphi_n(\tau)\|^2_{H^2(\Td)},
	\end{split}
\end{equation}
where in the last line we used the Gagliardo--Nirenberg inequality and the fact that
$|\varphi_n| \leq 1$ almost everywhere in $\Omega \times (0,T)\times\Td$. Collecting \eqref{eq:energy2} and \eqref{eq:energy3} in \eqref{eq:energy1}, raising the result to the power $\frac p2$, then taking suprema in time and expectations,
by exploiting \eqref{eq:energy1} we arrive at
\begin{multline} \label{eq:energy4}
	\E \supp \|\nabla \varphi_n(\tau)\|^p_\bH +
	\E \left| \int_0^t \|F(\varphi_n(\tau))\|_{L^1(\Td)}\:\d\tau\right|^\frac p2 +
	\E \left|\int_0^t \|\nabla \mu_n(\tau)\|^2_\bH \:\d \tau\right|^\frac p2 \\
	\leq  C\bigg[ 1 + \E \|\nabla \varphi_0\|^p_\bH + \E \|F(\varphi_0)\|_{L^1(\Td)}^\frac p2 +
	 \E \sups \left| \int_0^s (\mu_n(\tau), \div \b{G}_n(\varphi_n(\tau))\:\d W(\tau))_H
	 \right|^\frac p2 \bigg].
\end{multline}
Finally, we give a control on the stochastic term. We have, by the Burk\"{o}lder--Davis--Gundy inequality,
\begin{equation*}
	\begin{split}
		\E \sups \left| \int_0^s (\mu_n(\tau), \div \b{G}_n(\varphi_n(\tau))\:\d W(\tau))_H \right|^\frac p2 & \leq
		C\E \left| \int_0^t \| \nabla \mu_n(\tau)\|^2_\bH\| \div \b{G}_n(\varphi_n(\tau))\|^2_{\cL^2(U,V_0^*)} \:\d \tau \right|^\frac p4 \\
		& \leq C\E \left| \int_0^t \| \nabla \mu_n(\tau)\|^2_\bH\| \varphi_n(\tau)\|^2_H
		 \:\d \tau \right|^\frac p4 \\
		& \leq C\E \left[ \supp \|\varphi_n(\tau)\|^\frac p2_\bH \left|
		\int_0^t \|\nabla \mu_n(\tau)\|^2_\bH \:\d \tau \right|^\frac p4 \right] \\
		& \leq C\left[ \E \supp \|\varphi_n(\tau)\|_H^p \right]^\frac12
		 \left[ \E \left| \int_0^t \|\nabla \mu_n(\tau)\|^2_\bH \:\d \tau
		 \right|^\frac p2 \right]^\frac 12 \\
		& \leq \dfrac 12 \E \left| \int_0^t \|\nabla \mu_n(\tau)\|^2_\bH \:\d \tau
		 \right|^\frac p2 + C\E \supp \|\varphi_n(\tau)\|_H^p,
	\end{split}
\end{equation*}
and therefore, we find by \eqref{eq:energy4}
\begin{multline}
	\E \supp \|\nabla \varphi_n(\tau)\|^p_\bH
	+ \E \left| \int_0^t \| F(\varphi_n(\tau))\|_{L^1(\Td)}\:\d\tau\right|^\frac p2
	+ \E \left|\int_0^t \|\nabla \mu_n(\tau)\|^2_\bH \:\d \tau\right|^\frac p2 \\
	\leq  C\bigg[ 1  + \E \supp \|\varphi_n(\tau)\|_H^p  \bigg].
\end{multline}
Thanks to the estimate \eqref{eq:unif1_n} we obtain also
\begin{equation}\label{eq:unif2_n}
	\|\varphi_n\|_{L^p(\Omega; L^\infty(0,T; V))}
	+\|\nabla \mu_n\|_{L^p(\Omega; L^2(0,T; \b H))} \leq C .
\end{equation}
In order to recover
an estimate on the full $V$-norm for $(\mu_n)_n$,
we argue as follows. Testing the equation for $\mu_n$ by
$\varphi_n-\overline{\varphi_{0n}}$, following \cite{MZ04}, we arrive at
\[
\|F'(\varphi_n)\|_{L^1(\Td)} \leq C\left( 1 + \|\nabla \varphi_n\|_\bH\|\nabla \mu_n\|_\bH \right),
\]
and observing that $\overline{\mu} = \overline{F'(\varphi)}$, we learn that
\[
\|F'(\varphi_n)\|_{L^1(\Td)} + |\overline{\mu_n}|
\leq C\left( 1 + \|\nabla \varphi_n\|_\bH\|\nabla \mu_n\|_\bH \right).
\]
Raising the previous inequality to the power 2, integrating the result, then taking powers $\frac p4$ and expectations yields
\[
\E \left|\int_0^t |\overline{\mu_n(\tau)}|^2 \:\d\tau \right|^\frac p4
\leq C\left( 1 + \E \supp \|\nabla \varphi_n(\tau)\|_\bH^p +
\E \left| \int_0^t \|\nabla \mu_n(\tau)\|_\bH \:\d\tau \right|^\frac p2\right),
\]
yielding immediately
\begin{equation}\label{eq:unif3_n}
\|\mu_n\|_{L^{\frac p2}(\Omega; L^2(0,T;V))}\leq C.
\end{equation}
This implies by comparison in the equation for $\mu_n$ also that
\begin{equation}\label{eq:unif4_n}
\|F'(\varphi_n)\|_{L^{\frac p2}(\Omega; L^2(0,T;H))}\leq C.
\end{equation}
By exploiting the estimates \eqref{eq:unif1_n}--\eqref{eq:unif4_n} we are now able to
pass to the limit as $n\to\infty$ and conclude. Indeed,
by following the proof of Lemmas~\ref{lem:tight1}, \ref{lem:tight2}, and \ref{lem:gkch},
via stochastic compactness we infer the existence of a filtered probability space
$(\widetilde{\Omega}, \widetilde{\cF}, (\widetilde{\cF}_t)_{t\in[0,T]}, \widetilde\P)$,
a cylindrical Wiener process $\widetilde W$ on $U$,
a sequence $(Y_n)_n$ of random variables
$Y_n:\widetilde\Omega\to\Omega$
such that ${Y_n}_\#\widetilde\P=\P$ for every $n\in\enne$,
and processes
\begin{align*}
  \widetilde\varphi&\in L^p_w(\widetilde\Omega; L^\infty(0,T; V))\cap
  L^p_\cP(\widetilde\Omega; L^2(0,T; H^2(\Td))),\\
  \widetilde\mu&\in L^{\frac p2}_\cP(\widetilde\Omega; L^2(0,T; V)),
  \quad\nabla\mu\in L^p_\cP(\widetilde\Omega; L^2(0,T; \b H)),
\end{align*}
such that, as $n\to\infty$,
\begin{align*}
  \widetilde\varphi_n:=\varphi_n\circ Y_n\to\widetilde\varphi
  \quad&\text{in } L^\ell_\cP(\widetilde\Omega;C^0([0,T]; H)\cap L^2(0,T; V))
  \quad\forall\,\ell\in[1,p),\\
  \widetilde\varphi_n\rightharpoonup\widetilde\varphi
  \quad&\text{in }  L^p_\cP(\widetilde\Omega; L^2(0,T; H^2(\Td))),\\
  \widetilde\mu_n:=\mu_n\circ Y_n\to\widetilde\mu
  \quad&\text{in } L^p_\cP(\widetilde\Omega; L^2(0,T; V)),\\
  \widetilde W_n:=W\circ Y_n\to\widetilde W
  \quad&\text{in } L^\ell_\cP(\widetilde\Omega;C^0([0,T]; U_0))
  \quad\forall\,\ell\in[1,+\infty),\\
  \widetilde\varphi_{0n}:=\varphi_{0n}\circ Y_n\to\widetilde\varphi_0
  \quad&\text{in } L^\ell_\cP(\widetilde\Omega;H)
  \quad\forall\,\ell\in[1,p),
\end{align*}
Proceeding as in the proof of Lemma~\ref{lem:tight2},
the strong-weak closure of maximal monotone graphs readily implies that
\[
  F'(\widetilde\varphi_n)\rightharpoonup
  F'(\widetilde\varphi) \quad\text{in } L^p_\cP(\widetilde\Omega; L^2(0,T; H)).
\]
In oder to pass to the limit we only need to handle the convergence of the stochastic
integrals: this is the only point that differs substantially
from the proof of Lemma~\ref{lem:tight2}.
First of all, we show that
\begin{equation}
  \label{conv1_n}
  \b G_n(\widetilde\varphi_n) \to \b G(\widetilde\varphi)
  \quad\text{in } L^\ell_\cP(\widetilde\Omega; L^\ell(0,T; \cL^2(U, \b H)))
  \quad\forall\,\ell\in[1,+\infty).
\end{equation}
Indeed, we have
\begin{align*}
  \|\b G_n(\widetilde\varphi_n) - \b G(\widetilde\varphi)\|_{\cL^2(U,\b H)}^2
  &=\sum_{k=1}^\infty\|\b g_{k,n}(\widetilde\varphi_n)-\b g_k(\widetilde\varphi_n)\|_{\b H}^2\\
  &\leq2\sum_{k=1}^\infty
  \|\b g_{k,n}(\widetilde\varphi_n)-\b g_{k}(\widetilde\varphi_n)\|_{\b H}^2
  +2\sum_{k=1}^\infty
  \|\b g_{k}(\widetilde\varphi_n)-\b g_{k}(\widetilde\varphi)\|_{\b H}^2
\end{align*}
where, by construction of $(\b g_{k,n})_{k,n}$, one has that
\[
  |\b g_{k,n}(s)-\b g_{k}(s)| \leq 2\|\b g_k\|_{W^{2,\infty}(-1,1; \erre^d)}
  \mathbbm1_{[-1,-1+\frac1n]\cup[1-\frac1n, 1]}(s) \quad\forall\,s\in[-1,1].
\]
It follows then, thanks to assumption \ref{hyp:G}, that
\begin{align*}
  \|\b G_n(\widetilde\varphi_n) - \b G(\widetilde\varphi)\|_{\cL^2(U,\b H)}^2
  &\leq 8L_{\b G}^2\mathbbm1_{\{|\widetilde\varphi_n|\geq1-\frac1n\}}
  +2L_{\b G}^2\|\widetilde\varphi_n-\widetilde\varphi\|_H^2
  \quad\text{a.e.~in } \widetilde\Omega\times(0,T).
\end{align*}
The second term on the right-hand side converges to $0$
almost everywhere in $\widetilde\Omega\times(0,T)$ by the convergences
proved above. As for the first one, we have thanks to the estimate \eqref{eq:unif4_n} that
\begin{align*}
  C\geq\int_{\widetilde\Omega\times(0,T)\times\Td}
  |F'(\widetilde\varphi_n)|
  &\geq\int_{\{|\widetilde\varphi_n|\geq1-\frac1n\}}
  |F'(\widetilde\varphi_n)|\\
  &\geq\min\left\{|F'(-1+1/n)|, |F'(1-1/n)|\right\}\cdot|\{|\widetilde\varphi_n|\geq1-1/n\}|,
\end{align*}
which yields that
\[
  |\{|\widetilde\varphi_n|\geq1-1/n\}|\leq\frac{C}{\min\left\{|F'(-1+1/n)|, |F'(1-1/n)|\right\}}
  \to 0.
\]
Taking these remarks into account, \eqref{conv1_n} follows by the dominated convergence theorem. As a direct consequence of \eqref{conv1_n} one obtains that
\[
  \operatorname{div}\b G_n(\widetilde\varphi_n) \to
  \operatorname{div}\b G(\widetilde\varphi)
  \quad\text{in } L^\ell_\cP(\widetilde\Omega; L^\ell(0,T; \cL^2(U, V^*)))
  \quad\forall\,\ell\in[1,+\infty).
\]
This allows to pass to the limit as $n\to\infty$ in the approximating problem
on the space $(\widetilde\Omega, \widetilde\cF, \widetilde\P)$. By uniqueness
of the solution proved in Theorem \ref{thm:uniqueCH}, this also shows existence
of a strong solution on the original probability space $(\Omega,\cF,\P)$,
in the sense of Definition~\ref{def:solCH}, with moments up to order $p$.
This concludes the proof of Theorem \ref{thm:weakCH}.

\section{Proof of Theorem~\ref{thm:weakAC}} \label{sec:existence2}
Finally, we are left to prove existence of solutions for the Allen--Cahn problem \eqref{eq:ac}. This can be achieved by means of a vanishing viscosity limit, i.e., letting $\alpha \to 0^+$ in \eqref{eq:chac}. For simplicity, we now set $\beta = 1$ and let $\alpha > 0$ be arbitrary but fixed. Consider the problem
\begin{equation} \label{eq:mixedac}
	\begin{cases}
		\d \varphi_\alpha + \left[-\alpha \Delta \mu_\alpha + \mu_\alpha - \overline{\mu_\alpha} \right]\d t = \div \left( \b{G}(\varphi_\alpha)\right) \d W & \quad \text{ in }\T^d \times (0,T) \\
		\mu_\alpha = -\Delta \varphi_\alpha + F'(\varphi_\alpha) & \quad \text{ in }\T^d \times (0,T) \\
		\varphi_\alpha(\cdot \:, 0) = \varphi_0 & \quad \text{ in } \Td,
	\end{cases}
\end{equation}
which admits a strong solution for any $\alpha\in (0,1]$. Moreover, we have almost exact conservation of mass, namely
\begin{equation} \label{eq:ac_1}
	\overline{\varphi_\alpha(t)} = \overline{\varphi_0}
\end{equation}
for all times, $\P$-almost surely, as well as the key bound
\begin{equation} \label{eq:ac_2}
	\|\varphi_\alpha(t)\|_{L^\infty(\Omega\times(0,T)\times\Td)} \leq 1.
\end{equation}
Applying the It\^{o} lemma to the $H$-norm of $\varphi_\alpha$ yields
\begin{multline*}
	\dfrac{1}{2}\|\varphi_{\alpha}(t)\|_{H}^2 +\alpha\int_0^t \left[ \|\Delta \varphi_{\alpha}(\tau)\|^2_H + \left( -\Delta\varphi_{\alpha}(\tau), F'(\varphi_{\alpha}(\tau))\right)_H \right] \: \mathrm{d}\tau \\
	+ \int_0^t \left[ \|\nabla \varphi_{\alpha}(\tau)\|^2_\bH + (F'(\varphi_{\alpha}(\tau)) - \overline{F'(\varphi_{\alpha}(\tau))}, \varphi_{\alpha}(\tau))_H \right]\: \d \tau
	\\= \dfrac{1}{2}\|\varphi_{0}\|_{H}^2 + \int_0^t \left(\varphi_{\alpha}(\tau), \div (\b{G}(\varphi_{\alpha}(\tau)))\,\mathrm{d}W(\tau)\right)_H + \dfrac{1}{2}\int_0^t \|\div \b{G}(\varphi_{\alpha}(\tau))\|^2_{\cL^2(U, H)} \: \mathrm{d}\tau.
\end{multline*}
Following the same computations leading to the first estimate \eqref{eq:unif1}, jointly with the very definition of $\div \b{G}$, we have
\begin{multline} \label{eq:ac_3}
	\dfrac{1}{2}\|\varphi_{\alpha}(t)\|_{H}^2 +\int_0^t
	\left[ \alpha\|\Delta \varphi_{\alpha}(\tau)\|^2_H
	+ \left(1 - \dfrac{L_\b{G}^2}{2}  \right)\|\nabla \varphi_{\alpha}(\tau)\|^2_\bH
	+\|\sqrt{\Psi''(\varphi_\alpha(\tau))}\nabla\varphi_\alpha(\tau)\|_\bH^2
	\right] \: \d \tau
	\\ \leq  \dfrac{1}{2}\|\varphi_{0}\|_{H}^2 +
	C\int_0^t\|\varphi_{\alpha}(\tau)\|^2_H \: \d \tau +
	 \int_0^t \left(\varphi_{\alpha}(\tau), \div (\b{G}(\varphi_{\alpha}(\tau)))\,\mathrm{d}W(\tau)\right)_H
\end{multline}
Let us recall that if $L_\b{G} \leq 2^\frac 12$, then $1 - \frac{L_\b{G}^2}{2} \geq 0$. Taking $\frac{p}{2}$-powers, suprema in time and expectations in \eqref{eq:ac_3} leads to
\begin{multline} \label{eq:ac_4}
	\E \supp \|\varphi_{\alpha}(\tau)\|_{H}^p +
	\E \left| \int_0^t \alpha\|\Delta \varphi_{\alpha}(\tau)\|^2_H \: \d \tau \right|^\frac p2
	\\+
	\E \left| \int_0^t \left(1 - \dfrac{L_\b{G}^2}{2}  \right)
	\|\nabla \varphi_{\alpha}(\tau)\|^2_\bH \: \d \tau \right|^\frac p2
	+\E \left| \int_0^t
	\|\sqrt{\Psi''(\varphi_\alpha(\tau))}
	\nabla \varphi_{\alpha}(\tau)\|^2_\bH \: \d \tau \right|^\frac p2
	\\ \leq  C \left[ \E \|\varphi_{0}\|_{H}^p +
	\E \left| \int_0^t \| \varphi_{\alpha}(\tau)\|^2_H \: \d \tau \right|^\frac p2
	+\E \sups \left| \int_0^s \left(\varphi_{\alpha}(\tau), \div (\b{G}(\varphi_{\alpha}(\tau)))\,\mathrm{d}W(\tau)\right)_H \right|^\frac p2\right].
\end{multline}
The Burkh\"{o}lder--Davis--Gundy inequality entails that
\begin{align}
	\nonumber
	&C\E \sups\left| \int_0^s  \left(\varphi_{\alpha}(\tau), \div (\b{G}(\varphi_{\alpha}(\tau)))\,\mathrm{d}W(\tau)\right)_H \right|^\frac p2  \\
	\nonumber
	& \qquad\leq C \E \left| \int_0^t  \|\varphi_{\alpha}(\tau)\|^2_H\|  \div (\b{G}(\varphi_{\alpha}(\tau)))\|^2_{\cL^2(U,H)}\:\d\tau \right|^\frac p4 \\
	\nonumber
	& \qquad \leq C\E \left| \supp \|\varphi_{\alpha}(\tau)\|_H^2\int_0^t \|  \div (\b{G}(\varphi_{\alpha}(\tau)))\|^2_{\cL^2(U,H)}\:\d\tau \right|^\frac p4 \\
	\nonumber
	& \qquad \leq C\E\left[ \supp \|\varphi_{\alpha}(\tau)\|_H^\frac p2 \left|\int_0^t  \| \nabla \varphi_{\alpha}(\tau)\|_\bH^2\:\d\tau \right|^\frac p4 \right] \\
	\nonumber
	& \qquad
	\leq C\left[ \E\supp  \|\varphi_{\alpha}(\tau)\|^p_H \right]^\frac 12 \left[\E \left|\int_0^t  \| \nabla \varphi_{\alpha}(\tau)\|_\bH^2\:\d\tau \right|^\frac p2 \right]^\frac 12 \\
	& \qquad \leq \left(\dfrac{1}{2} - \dfrac{L_\b{G}^2}{4} \right)\left[\E \left|\int_0^t  \| \nabla \varphi_{\alpha}(\tau)\|_\bH^2\:\d\tau \right|^\frac p2 \right] + C,
	\label{eq:ac_5}
\end{align}
where in the last line we used \eqref{eq:ac_2}. Collecting \eqref{eq:ac_5} in \eqref{eq:ac_4} and using the Gronwall lemma gives
\begin{align} \label{eq:ac_6}
	\|\varphi_\alpha\|_{L^p_\cP(\Omega;L^\infty(0,T;H))
	\cap L^p_\cP(\Omega;L^2(0,T;V))}
	+\alpha^{\frac12} \|\varphi_\alpha\|_{L^p_\cP(\Omega;L^2(0,T;H^2(\Td)))}&\leq C_1\\
	\label{eq:ac_6'}
	\|\sqrt{\Psi''(\varphi_\alpha)}\nabla\varphi_\alpha\|_{L^p_\cP(\Omega;L^2(0,T;\bH))}
	&\leq C_1,
\end{align}
where the constant $C_1 > 0$ is independent of $\alpha$.
The energy inequality proved in the previous subsection shows
by analogous computations that
\begin{multline} \label{eq:ac_7}
	\E \supp \|\nabla \varphi_\alpha(\tau)\|^p_\bH +
	\E \left| \int_0^t \| F(\varphi_\alpha(\tau))\|_{L^1(\Td)}\:\d\tau\right|^\frac p2 \\
	+
	\E \left|\int_0^t \alpha\|\nabla \mu_\alpha(\tau)\|^2_\bH \:\d \tau\right|^\frac p2+
	\E \left|\int_0^t \| \mu_\alpha(\tau)-\overline{\mu_\alpha(\tau)}\|^2_\bH
	 \:\d \tau\right|^\frac p2
	 \\ \leq  C\bigg[ 1 + \E \|\nabla \varphi_0\|^p_\bH + \E \|F(\varphi_0)\|_{L^1(\Td)}^\frac p2
	+ L_{\b G}^p\E \left|\int_0^t \|\varphi_\alpha(\tau)\|^2_{H^2(\Td)} \:\d \tau\right|^\frac p2\\
	+\E \left| \int_0^t
	\|\sqrt{\Psi''(\varphi_\alpha(\tau))}
	\nabla \varphi_{\alpha}(\tau)\|^2_\bH \: \d \tau \right|^\frac p2
	+\E \sups \left| \int_0^s (\mu_\alpha(\tau),
	\div \b{G}(\varphi_\alpha(\tau))\:\d W(\tau))_H \right|^\frac p2
	 \bigg],
\end{multline}
where
\begin{equation*}
	\begin{split}
		\E \sups \left| \int_0^s (\mu_\alpha(\tau),
		\div \b{G}(\varphi_\alpha(\tau))\:\d W(\tau))_H \right|^\frac p2 & \leq
		C\E \left| \int_0^t \| \mu_\alpha(\tau)-
		\overline{\mu_\alpha(\tau)}\|^2_H
		\| \div \b{G}(\varphi_\alpha(\tau))\|^2_{\cL^2(U,H)} \:\d \tau \right|^\frac p4 \\
		& \leq CL_{\b G}^{\frac p2}\E \left| \int_0^t \| \mu_\alpha(\tau)-
		\overline{\mu_\alpha(\tau)}\|^2_H
		\| \nabla\varphi_\alpha(\tau)\|^2_\bH
		 \:\d \tau \right|^\frac p4 \\
		& \leq CL_{\b G}^{\frac p2}
		\E \left[ \supp \|\nabla\varphi_\alpha(\tau)\|^\frac p2_\bH \left|
		\int_0^t \| \mu_\alpha(\tau)-
		\overline{\mu_\alpha(\tau)}\|^2_H \:\d \tau \right|^\frac p4 \right] \\
		& \leq CL_{\b G}^{\frac p2}
		\left[ \E \supp \|\nabla\varphi_\alpha(\tau)\|_\bH^p \right]^\frac12
		 \left[ \E \left| \int_0^t
		 \| \mu_\alpha(\tau)-
		\overline{\mu_\alpha(\tau)}\|^2_H \:\d \tau
		 \right|^\frac p2 \right]^\frac 12 \\
		& \leq \dfrac 12 \E \left| \int_0^t
		\| \mu_\alpha(\tau)-
		\overline{\mu_\alpha(\tau)}\|^2_H \:\d \tau
		 \right|^\frac p2 + CL_{\b G}^{p}\E \supp \|\nabla\varphi_\alpha(\tau)\|_\bH^p,
	\end{split}
\end{equation*}
Applying \eqref{eq:ac_6}--\eqref{eq:ac_6'} in \eqref{eq:ac_7} yields
\begin{multline}
	\E \supp \|\nabla \varphi_\alpha(\tau)\|^p_\bH +
	\E \left| \int_0^t \| F(\varphi_\alpha(\tau))\|_{L^1(\Td)}\:\d\tau\right|^\frac p2 \\
	+\E \left|\int_0^t \alpha\|\nabla \mu_\alpha(\tau)\|^2_\bH \:\d \tau\right|^\frac p2+
	\E \left|\int_0^t \| \mu_\alpha(\tau)-\overline{\mu_\alpha(\tau)}\|^2_\bH
	 \:\d \tau\right|^\frac p2 \\
	\leq  C\bigg[ 1 +L_{\b G}^{p}\E \supp \|\nabla\varphi_\alpha(\tau)\|_\bH^p+
	 L_{\b G}^{p}
	 \E \left|\int_0^t \|\varphi_\alpha(\tau)\|^2_{H^2(\Td)} \:\d \tau\right|^\frac p2 \bigg].
\end{multline}
Of course, the constant $C>0$ appearing above is independent of $\alpha$,
and depends only on the structure of the problem.
Noting also that from the equation for $\mu_\alpha$ and the monotonicity of
$\Psi'$ it follows that
\begin{align*}
  \|\Delta\varphi_\alpha\|_H^2&\leq \int_{\Td}\mu_\alpha(-\Delta\varphi_\alpha)
  -\int_{\Td} R''(\varphi_\alpha)|\nabla\varphi_\alpha|^2\\
  &\leq\frac12\|\Delta\varphi_\alpha\|_H^2 + \frac12\|\mu_\alpha-\overline{\mu_\alpha}\|_H^2
  +C\|\nabla\varphi_\alpha\|^2_\bH,
\end{align*}
it is clear that there exists a constant $C_0>0$, independent of $\alpha$,
such that for $L_{\b G}^2<C_0$ one has that
\begin{equation} \label{eq:ac_8'}
	\|\varphi_\alpha\|_{L^p_\cP(\Omega;L^\infty(0,T;V))\cap
	L^p_\cP(\Omega; L^2(0,T; H^2(\Td)))} +
	\alpha^{\frac12}\|\mu_\alpha\|_{L^p_\cP(\Omega;L^2(0,T;V))} +
	\|\mu_\alpha-\overline{\mu_\alpha}\|_{L^p_\cP(\Omega;L^2(0,T;H))}\leq C_2
\end{equation}
with $C_2$ being independent of $\alpha$.
Following now the computations in the previous subsection we readily obtain
\begin{equation} \label{eq:ac_8}
	\|\mu_\alpha\|_{L^\frac p2_\cP(\Omega;L^2(0,T;H))} +
	\|F'(\varphi_\alpha)\|_{L^\frac p2_\cP(\Omega;L^2(0,T;H))}\leq C_3.
\end{equation}
Furthermore, as a consequence of the above estimates, we have
\begin{equation} \label{eq:ac_9}
	\|\div \b{G}(\varphi_\alpha)\|_{L^p_\cP(\Omega; L^\infty(0,T;\cL^2(U,H))) \cap L^p_\cP(\Omega; L^2(0,T;\cL^2(U,V))) }\leq C_4,
\end{equation}
and in turn, applying \cite[Lemma 2.1]{fland-gat}, we have
\begin{equation} \label{eq:ac_10}
	\left\|\int_0^\cdot \div \b{G}(\varphi_\alpha) \: \d W\right\|_{ L^p_\cP(\Omega; W^{s, p}(0,T;H)) \cap L^2_\cP(\Omega; W^{s, 2}(0,T;V)) }\leq C_5.
\end{equation}
Here, $C_3, C_4, C_5$ are again positive constants independent of $\alpha$ and $s \in (0,\frac 12)$. In particular, by comparison we easily get
\[
\|\varphi_\alpha\|_{L^p_\cP(\Omega;W^{s,p}(0,T;V^*))} \leq C_6,
\]
with $C_5$ independent of $\alpha$. We are now in a position to perform a stochastic compactness argument. The following is a tightness lemma.
\begin{lem} \label{lem:tight3}
	Assume that $p > 2$. Then, the following tightness properties hold:
	\begin{enumerate}[(i)] \itemsep 0.5em
		\item the family of laws of $\{\varphi_\alpha\}_{\alpha > 0}$ is tight in the space $C^0([0,T];H) \cap L^2(0,T;V)$;
		\item the family of laws of the stochastic integrals $\{\div \b{G}(\varphi_\alpha) \cdot W \}_{\alpha > 0}$ is tight in the space $C^0([0,T];V^*)$;
		\item if we identify the Wiener process $W$ with a constant sequence $\{W_\alpha\}_{\alpha >0}$, then the corresponding family of laws
		is tight in the space $C^0([0,T];U_0)$;
		\item if we identify the initial condition $\varphi_0$ with a constant sequence $\{\varphi_{0,\alpha}\}_{\alpha >0} \subset V$, then the corresponding
		family of laws is tight in the space $H$.
	\end{enumerate}
\end{lem} \noindent
The proof of the above lemma is omitted as it is follows the ones of Lemmas \ref{lem:tight1} and \ref{lem:tight2}. Owing to the Prokhorov and Skorokhod theorems (see \cite[Theorem 2.7]{ike-wata} and \cite[Theorem 1.10.4, Addendum 1.10.5]{vaa-well}), there exists a probability space $(\tom, \widetilde{\mathscr{F}},\tP)$ and a family of random variables $X_\alpha:  (\tom, \tF)\to(\Omega, \mathscr{F})$ such that the law of $X_\alpha$ is $\P$ for every $\alpha > 0$, namely $\tP \circ X_\alpha^{-1} = \P$ (so that composition with $X_\alpha$ preserves laws), and the following convergences hold
\begin{align*}
	\widetilde{\varphi}_{\alpha} := \varphi_\alpha \circ X_\alpha  \to \widetilde{\varphi}
	\quad & \text{in } L^q(\tom;L^2(0,T;V) \cap C^0([0,T];H)) \text{ if } q < p, \\
	\widetilde{\varphi}_{\alpha} \rightharpoonup \widetilde{\varphi}
	\quad & \text{in } L^p(\tom;L^2(0,T;H^2(\Td))), \\
	\widetilde{\varphi}_{\alpha} \overset{\ast}{\rightharpoonup} \widetilde{\varphi}
	\quad & \text{in } L^p_w(\tom;L^\infty(0,T;V)) \cap L^p(\tom; W^{s,p}(0,T;V^*)), \\
	\widetilde{I}_{\alpha} := (\div \b{G}(\varphi_\alpha) \cdot W) \circ X_\alpha \to \widetilde{I}
	\quad & \text{in } L^q(\tom;C^0([0,T];V^*))\text{ if }q < p, \\
	\widetilde W_{\alpha} \to \widetilde{W} \quad & \text{in }L^q(\tom;C^0([0,T];U)), \text{ if } q < p, \\
	\widetilde{\varphi}_{0,\alpha} \to \widetilde{\varphi}_{0} \quad & \text{in }L^q(\tom;H), \text{ if } q < p,
\end{align*}
provided that $sp > 1$. Here, the limit processes also belong to the specified spaces. Moreover, we also have
\[
\widetilde{\mu}_\alpha := \mu_\alpha \circ X_\alpha \rightharpoonup \widetilde \mu \quad \text{in } L^\frac p2 (\tom; L^2(0,T;H)),
\]
as well as
\begin{align*}
  \alpha\widetilde\varphi_\alpha\to 0 \quad&\text{in }
  L^p_\cP(\widetilde\Omega; L^\infty(0,T; V)\cap L^2(0,T; H^2(\Td))),\\
  \alpha\widetilde\mu_\alpha\to 0 \quad&\text{in }
  L^p_\cP(\widetilde\Omega; L^2(0,T; V)).
\end{align*}
We tackle now the two nonlinearities. First, we show that $F'(\widetilde{\varphi}_\alpha) \to F'(\widetilde{\varphi})$ in a suitable sense. Of course, we immediately have $R'(\widetilde{\varphi}_\alpha) \to R'(\widetilde{\varphi})$ by Lipschitz continuity. As for the singular part, observe that $\Psi'$ is maximal monotone. Clearly, by the uniform bound \eqref{eq:ac_8}, we have that
\[
\Psi'(\widetilde{\varphi}_\alpha) \rightharpoonup \xi \quad \text{in } L^\frac p2 (\tom; L^2(0,T;H)).
\]
By the weak-strong closure of maximal monotone operators (see, for instance, \cite[Proposition 2.1]{barbu-monot}), we identify $\xi = \Psi'(\widetilde{\varphi})$. Finally, the stochastic integral can be identified as follows (see also \cite{scarpa21}). First of all, observe that
\begin{align*}
\|\div \b{G}(\widetilde{\varphi}_\alpha) - \div \b{G}(\widetilde{\varphi})\|^2_{\cL^2(U,V^*)}
&\leq C\|\widetilde \varphi_\alpha - \widetilde \varphi\|^2_H
\end{align*}
hence, we immediately infer
\[
\div \b{G}(\widetilde{\varphi}_\alpha) \to \div \b{G}(\widetilde{\varphi}) \quad
\text{in }L^q(\tom; L^\infty(0,T;\cL^2(U,V^*))) \text{ if } q < p.
\]
The procedure to correctly identify the limit integral $\widetilde{I}$ is standard, for instance see \cite[Section 8.4]{dapratozab}. Consider the family of filtrations on $(\tom, \tF, \tP)$ by setting
\[
\tF_{\alpha,t} := \sigma\left\{\widetilde{\varphi}_{\alpha}(s),\,  \widetilde{I}_{\alpha}(s),\, \widetilde{W}_{\alpha}(s),\, \widetilde{\varphi}_{0,\alpha}, \, s \in [0,t]\right\},
\]
for any $t \geq 0$ and $\alpha > 0$, in such a way that  $\widetilde{W}$ is an adapted process. In particular, by preservation of laws and the definitions of Wiener process and stochastic integral, it is immediate to prove that $\widetilde W_{\alpha}$ is a $Q^0$-Wiener process on $U_0$ and the stochastic integral
\[
\widetilde{I}_{\alpha} = (\div\b{G}(\widetilde\varphi_\alpha) \cdot W) \circ X_\alpha = \int_0^t \div\b{G}(\widetilde\varphi_\alpha(\tau)) \: \mathrm{d}\widetilde{W}_{\alpha}(\tau),
\]
defines a martingale with values in the space $H$. In a similar fashion, construct the limit filtration
\[
\tF_t := \sigma\left\{\widetilde{\varphi}_{\lambda}(s),\, \widetilde{I}(s),\, \widetilde{W}(s),\, \widetilde{\varphi}_0,\, s \in [0,t]\right\}.
\]
Since $\widetilde{W}_\alpha(0) = 0$ for all $\alpha > 0$, the proven convergences immediately yield that that both $\widetilde{W}(0) = 0$ as well. Let now $t > 0$, $s \in [0,t]$ and define the space
\begin{align*}
	\X_s &:= (L^2(0,s;V) \cap C^0([0,s];H)) \times C^0([0,s]; H)  \times  C^0([0,s]; U_0)  \times H.
\end{align*}
For any arbitrary $\psi: \X_s \to \erre$ bounded and continuous, we have
\begin{equation} \label{eq:limit_n0}
	\tE \left[ \left(\widetilde{W}_{\alpha}(t) -\widetilde{W}_{\alpha}(s) \right) \psi\left( \widetilde{\varphi}_{\alpha},\, \widetilde{I}_{\alpha},\, \widetilde{W}_{\alpha},\,\widetilde{\varphi}_{0,\alpha} \right)  \right] = 0.
\end{equation}
In order to make sense of \eqref{eq:limit_n0}, the arguments of $\psi$ are intended to be restricted over $[0,s]$ when necessary and $\tE$ denotes the expectation with respect to $\tP$. Letting $\alpha \to 0^+$ in \eqref{eq:limit_n0} and applying the Lebesgue dominated convergence theorem, owing to the proven convergences and the properties of $\psi$, entails
\begin{equation} \label{eq:limit_n1}
	\tE \left[ \left(\widetilde{W}(t) -\widetilde{W}(s) \right) \psi\left( \widetilde{\varphi},\, \widetilde{I},\,   \widetilde{W},\,\widetilde{\varphi}_0  \right)  \right] = 0,
\end{equation}
which expresses the fact that $\widetilde{W}$ is a $U_0$-valued $(\tF_{t})_t$-martingale. The characterization of $Q$-Wiener processes given in \cite[Theorem 4.6]{dapratozab} leads us to compute the quadratic variation of $\widetilde{W}$. In particular, notice that \eqref{eq:limit_n0} means that, for every $v, w \in U_0$
\begin{multline*}
	\tE \left[ \left(\left( \widetilde{W}_{\alpha}(t), v \right)_{U_0}
	\left( \widetilde{W}_{\alpha}(t), w \right)_{U_0}
	-\left( \widetilde{W}_{\alpha}(s), v \right)_{U_0}
	\left( \widetilde{W}_{\alpha}(s), w \right)_{U_0}
	\right. \right. \\ \left. \left.
	- (t-s)\left(Q^0v,w\right)_{U_0}\right)
	\psi\left( \widetilde{\varphi}_{\alpha},\,
	\widetilde{I}_{\alpha},\,
	\widetilde{W}_{\alpha},\, \widetilde{\varphi}_{0,\alpha} \right)  \right] = 0,
\end{multline*}
and using once more the dominated convergence theorem, we get
\begin{equation*}
	\tE \left[ \left(\left( \widetilde{W}(t), v \right)_{U_0}
	\left( \widetilde{W}(t), w \right)_{U_0}
	-\left( \widetilde{W}(s), v \right)_{U_0}
	\left( \widetilde{W}(s), w \right)_{U_0}
	\right. \right. \left. \left.
	- (t-s)\left(Q^0v,w\right)_{U_0}\right)
	\psi\left( \widetilde{\varphi},\,
	\widetilde{I},\,
	\widetilde{W},\, \widetilde{\varphi}_{0} \right)  \right] = 0,
\end{equation*}
namely
\begin{equation*}
	\left\llangle \widetilde{W}\right\rrangle(t) = tQ^0, \qquad t \in [0,T],
\end{equation*}
which, by means of \cite[Theorem 4.6]{dapratozab}, implies that $\widetilde{W}$ is a $Q^0$-Wiener process, adapted to $(\tF_{\lambda,t})_t$. Concerning the stochastic integrals, we can argue exactly as in \eqref{eq:limit_n0}-\eqref{eq:limit_n1} to find that $\widetilde{I}_{\alpha}$ is an $H$-valued martingale. Then, \cite[Theorem 4.27]{dapratozab} yields
\begin{equation*}
	\left\llangle \widetilde{I}_{\alpha}\right\rrangle(t) = \int_0^t \div\b{G}(\widetilde{\varphi}_\alpha(\tau)) \circ [\div\b{G}(\widetilde{\varphi}_\alpha(\tau)) ]^*\: \mathrm{d}\tau
\end{equation*}
for every $t \in [0,T]$. Once again, fixing $v,w \in H$, we have
\begin{multline*}
	\tE \left[ \left(\left( \widetilde{I}_{\alpha}(t), v \right)_H \left( \widetilde{I}_{\alpha}(t), w \right)_H - \left( \widetilde{I}_{\alpha}(s), v \right)_H\left( \widetilde{I}_{\alpha}(s), w \right)_H \right. \right. \\ \left. \left. -\int_0^t \left( \div\b{G}(\widetilde{\varphi}_\alpha(\tau)) \circ [\div\b{G}(\widetilde{\varphi}_\alpha(\tau)) ]^*v, w\right)_H\d\tau\right) \psi\left( \widetilde{\varphi}_{\alpha},\, \widetilde{I}_{\alpha},\, \widetilde{W}_{\alpha},\,\widetilde{\varphi}_{0,\alpha}\right)  \right] = 0,
\end{multline*}
and, as we pass to the limit $\alpha \to 0^+$, the dominated convergence theorem implies that
\begin{multline*}
	\tE \left[ \left(\left( \widetilde{I}(t), v \right)_{V^*}\left( \widetilde{I}(t), w \right)_{V^*} - \left( \widetilde{I}(s), v \right)_{V^*}
	\left( \widetilde{I}(s), w \right)_{V^*} \right. \right. \\ \left. \left.
	-\int_0^t \left( \div\b{G}(\widetilde{\varphi}(\tau)) \circ
	 [\div\b{G}(\widetilde{\varphi}(\tau)) ]^*v, w\right)_{V^*}
	 \d\tau\right)
	 \psi\left( \widetilde{\varphi},\, \widetilde{I},\,
	 \widetilde{W},\,\widetilde{\varphi}_{0}\right)  \right] = 0,
\end{multline*}
The quadratic variation of $\widetilde{I}$ as a $V^*$-martingale is therefore
\[
\left\llangle \widetilde{I} \right\rrangle(t) = \int_0^t \div\b{G}(\widetilde{\varphi}(\tau)) \circ [\div\b{G}(\widetilde{\varphi}(\tau)) ]^*\: \mathrm{d}\tau, \qquad t \in [0,T].
\]
Finally, we can identify $\widetilde{I}$ with the martingale
\[
\widetilde{M}(t) := \int_0^t \div\b{G}(\widetilde{\varphi}(\tau))\: \mathrm{d}\widetilde{W}(\tau),
\]
which is an $H$-valued $(\tF_{t})_t$-martingale having the same quadratic variation of $ \widetilde{I}$. By \cite[Theorem 3.2]{Pard}, we can compute the quadratic variation of the difference
\begin{equation} \label{eq:limit_n2}
	\begin{split}
		\left\llangle \widetilde{M} - \widetilde{I} \right\rrangle & = \left\llangle \widetilde{M} \right\rrangle + \left\llangle \widetilde{I} \right\rrangle -2 \left\llangle \widetilde{M}, \widetilde{I} \right\rrangle \\
		& = 2\int_0^\cdot \div\b{G}(\widetilde{\varphi}(\tau)) \circ [\div\b{G}(\widetilde{\varphi}(\tau)) ]^*\: \mathrm{d}\tau - 2\int_0^\cdot \div\b{G}(\widetilde{\varphi}(\tau))\: \mathrm{d}\left\llangle\widetilde{W}, \widetilde{I} \right\rrangle (\tau).
	\end{split}
\end{equation}
To evaluate the cross quadratic variation in \eqref{eq:limit_n2}, notice that by \cite[Theorem 3.2]{Pard}, we have
\[
\begin{split}
	\left\llangle\widetilde{I}_{\alpha}, \widetilde{W}_{\alpha} \right\rrangle & = \int_0^\cdot \div\b{G}(\widetilde{\varphi}_{\alpha}(\tau)) \circ \iota^{-1}_1 \: \mathrm{d}\left\llangle\widetilde{W}_{\alpha}, \widetilde{W}_{\alpha}\right\rrangle (\tau) \\
	& = \int_0^\cdot \div\b{G}(\widetilde{\varphi}_{\alpha}(\tau)) \circ \iota^{-1} \circ Q^0 \: \mathrm{d}\tau \\
	& = \int_0^\cdot \div\b{G}(\widetilde{\varphi}_{\alpha}(\tau)) \circ \iota^{-1} \circ \iota \circ \iota^* \: \mathrm{d}\tau \\
	& = \int_0^\cdot \div\b{G}(\widetilde{\varphi}_{\alpha}(\tau)) \circ \iota^* \: \mathrm{d}\tau,
\end{split}
\]
where we also used the fact that $Q^0 = \iota \circ \iota^*$, where $\iota: U \to U^0$ is the Hilbert-Schmidt embedding needed to enlarge the space $U$. This implies that
\[
\left\llangle \widetilde{W}_{\alpha}, \widetilde{I}_{\alpha} \right\rrangle = \int_0^\cdot \iota \circ \div\b{G}(\widetilde{\varphi}_{\alpha}(\tau))^* \: \mathrm{d}\tau.
\]
A further application of the dominated convergence theorem entails that, as $\alpha \to 0^+$,
\begin{equation} \label{eq:limit_n3}
	\left\llangle \widetilde{W}, \widetilde{I} \right\rrangle = \int_0^\cdot \iota \circ \div\b{G}(\widetilde{\varphi}(\tau))^* \: \mathrm{d}\tau.
\end{equation}
The identification follows injecting \eqref{eq:limit_n3} in \eqref{eq:limit_n2}. Finally, passing to the limit in the weak formulation of the problem can be done by repeated applications of the dominated convergence theorem.

\section*{Acknowledgments}
The authors have been partially funded by MIUR-PRIN research grant no.~2020F3NCPX ``Mathematics for Industry 4.0 (Math4I4)''. The authors are also members of Gruppo Nazionale per l'Analisi Ma\-te\-ma\-ti\-ca, la Probabilit\`{a} e le loro Applicazioni (GNAMPA),
Istituto Nazionale di Alta Matematica (INdAM).
\printbibliography

\end{document}

\begin{defin}[strong solution to the conserved Allen--Cahn problem] \label{def:solAC}
Let $p \geq 2$ and assume that $\varphi_0$ satisfies
\[
\varphi_0 \in L^p(\Omega, \cF_0, \P; V), \qquad F(\varphi_0) \in L^\frac p2(\Omega, \cF_0, \P; L^1(\Td)).
\]
A strong solution to the Allen--Cahn problem \eqref{eq:ac} originating from the initial datum $\varphi_0$ is a stochastic process $\varphi$ such that
\begin{align}
\label{acphi}
&\varphi \in L^p_\cP(\Omega; C^0([0,T]; H))\cap
L^p_w(\Omega; L^\infty(0,T; V)) \cap
L^p_\cP(\Omega; L^2(0,T; H^2(\Td)))\,,\\
& |{\varphi}(\omega, x, t)| < 1 \text{ for a.a. } (\omega, x,t) \in \Omega \times \OO \times (0,T)\,, \\
\label{acmu}
&\mu:=-\Delta\varphi+F'(\varphi) \in L^{p}_\cP(\Omega; L^2(0,T; H))\,,\\
\label{acinitial}
&\varphi(0) = \varphi_0\,,
\end{align}
and
\begin{equation} \label{eq:weakAC_ps}
( \varphi(t),\psi)_H +
\int_0^t\!\int_{\Td}  \mu(s) \psi\,\d s
= ( \varphi(0),\psi)_{H} +
\left(\int_0^t \div \b{G}(\varphi(s)) \,\d  W(s), \psi\right)_{H}
\end{equation}
for every $\psi\in V$, $t \in [0,T]$, $\P$-almost surely.
\end{defin}

Now, let us start back by applying the It\^{o} formula to the $H$-norm of $\varphi$. This yields the standard estimate
\begin{multline}
	\dfrac 12 \|\varphi(t)\|_H^2 + \int_0^t \|\nabla \varphi(\tau)\|_\bH^2 \: \d\tau
	+ \int_0^t \left( \varphi(\tau), F'(\varphi_{1}(\tau)) -F'(\varphi_{2}(\tau))\right)_H \: \d \tau \\= \dfrac 12 \|\varphi_0\|_H^2 + \int_0^t \left(\nabla\varphi(\tau),\,
	\b{G}(\varphi_{1}(\tau))
	- \b{G}(\varphi_{2}(\tau))\,\mathrm{d}W(\tau)\right)_{\bH}  \\ + \int_0^t\|\div \b{G}(\varphi_1(\tau)) - \div \b{G}(\varphi_2(\tau))\|^2_{\cL^2(U,H)} \: \d \tau.
	\label{eq:uniqac3}
\end{multline}
On account of \eqref{eq:uniqac2}
\begin{align*}
	&\E\supp\|\varphi(\tau)\|_{\sharp}^2
	+\frac12(2-L_{\b G}^2)\E\int_0^t \|\varphi(\tau)\|^2_H \: \mathrm{d}\tau \\
	&\qquad\leq
	\E\| \varphi_0\|_{\sharp}^2 +
	2C\E\int_0^T \|\varphi(\tau)\|^2_\sharp \: \mathrm{d}\tau
	+ 2\E\sups\int_0^s \left(\varphi(\tau),\,
	\mathcal{N}[\div \b{G}(\varphi_{1}(\tau))
	- \div \b{G}(\varphi_{2}(\tau))]\,\mathrm{d}W(\tau)\right)_H
\end{align*}
In order to control the stochastic term, we make once again use of the
Burkholder--Davis--Gundy inequality.
Therefore, we have
\begin{align}
	\nonumber
	&2\E \sups \left|\int_0^s
	\left(\varphi(\tau)
	,\, \mathcal{N}[\div \b{G}(\varphi_{1}(\tau))
	- \div \b{G}(\varphi_{2}(\tau))]\,\mathrm{d}W(\tau)\right)_H \right| \\
	\nonumber
	& \qquad \leq
	C
	\E \left|\int_0^t \|\varphi(\tau)\|^2_{V^*_0}
	\|\mathcal{N}[\div \b{G}(\varphi_{1}(\tau))
	- \div \b{G}(\varphi_{2}(\tau))]\|^2_{\cL^2(U,V_0)} \:\d\tau \right|^\frac 12 \\
	\nonumber
	&\qquad\leq
	C
	\E \left|\int_0^t \|\varphi(\tau)\|^2_\sharp
	\|\varphi(\tau)\|^2_{H} \:\d\tau \right|^\frac 12 \\
	\nonumber
	&\qquad\leq
	C\E\left[ \supp\|\varphi(\tau)\|_{\sharp}
	\left|\int_0^t \| \varphi(\tau)\|^2_H  \:\d\tau \right|^\frac 12  \right] \\
	\label{eq:uniqac5}
	& \qquad \leq
	\luca{\frac12\E \supp\|\varphi(\tau)\|_{\sharp}^2 +
		\frac{C^2}2
		\E\int_0^t \|\varphi(\tau)\|^2_H  \:\d\tau,}
\end{align}
\luca{from which we infer that
	\[
	\frac12\E\supp\|\varphi(\tau)\|_{\sharp}^2
	+\frac12(2-L_{\b G}^2)\E\int_0^t \|\varphi(\tau)\|^2_H \: \mathrm{d}\tau
	\leq
	\E\| \varphi_0\|_{\sharp}^2 +
	2C
	\E\int_0^T \|\varphi(\tau)\|^2_\sharp \: \mathrm{d}\tau
	+\frac{C^2}2
	\E\int_0^T \|\varphi(\tau)\|^2_H \: \mathrm{d}\tau.
	\]
	The conclusion follows then from \eqref{contAC_prelim}.}}
